\UseRawInputEncoding
\documentclass[twoside,11pt]{article}

\usepackage{jmlr2e}

\usepackage{lipsum}
\usepackage{amsfonts}
\usepackage{graphicx}
\usepackage{epstopdf}
\usepackage{algorithm}
\usepackage{algorithmic}
\usepackage{amsmath}
\usepackage{amssymb}
\usepackage{bm}
\usepackage{mathbbold}
\usepackage{mathrsfs}
\usepackage{multirow}
\usepackage{dcolumn}
\usepackage{listings}
\usepackage{subfigure}
\usepackage{enumerate}

\newcommand{\dataset}{{\cal D}}
\newcommand{\fracpartial}[2]{\frac{\partial #1}{\partial  #2}}

\newcommand{\ud}{\mathrm{d}}
\newtheorem{thm}{Theorem}
\newtheorem{thma}{Theorem}
\newtheorem{lem}{Lemma}
\newtheorem{lema}{Lemma}
\newtheorem{cor}{Corollary}
\newtheorem{prop}{Proposition}
\newtheorem{defn}{Definition}

\newtheorem{rem}{Remark}
\newtheorem{assum}{Assumption}

\newtheorem{assuma}{Assumption}

\jmlrheading{--}{2021}{1-62}{1/20; Revised 8/21}{--/--}{}{Xiaopeng Luo and Xin Xu}

\ShortHeadings{Contraction methods for continuous optimization}{Luo and Xu}
\firstpageno{1}

\begin{document}

\title{Contraction methods for continuous optimization}

\author{\name Xiaopeng Luo \email luo.permenant@gmail.com\\
	\name Xin Xu \email xu.permenant@gmail.com\\
\addr Department of Control and Systems Engineering\\
	School of Management and Engineering\\
    Nanjing University\\
	Nanjing, 210093, China\\
	\addr Department of Chemistry\\
	Princeton University\\
	Princeton, NJ 08544, USA}

\editor{}

\maketitle

\begin{abstract}
 Motivated by the grid search method and Bayesian optimization, we introduce the concept of contractibility and its applications in model-based optimization. First, a basic framework of contraction methods is established to construct a nonempty closed set sequence that contracts from the initial domain to the set of global minimizers. Then, from the perspective of whether the contraction can be carried out effectively, relevant conditions are introduced to divide all continuous optimization problems into three categories: (i) logarithmic time contractible, (ii) polynomial time contractible, or (iii) noncontractible. For every problem from the first two categories, there exists a contraction sequence that converges to the set of all global minimizers with linear convergence; for any problem from the last category, we discuss possible troubles caused by contraction. Finally, a practical algorithm is proposed with high probability bounds for convergence rate and complexity. It is shown that the contractibility contributes to practical applications and can also be seen as a complement to smoothness for distinguishing the optimization problems that are easy to solve.
\end{abstract}


\section{Introduction}
\label{CM:s1}

For a possibly nonlinear and nonconvex continuous function $f:\Omega\subset\mathbb{R}^n\to\mathbb{R}$ with the global minima $f^*$ and the set of all global minimizers $X^*$ in $\Omega$, we consider the constrained optimization problem
\begin{equation}\label{CM:eq:COP}
\min_{x\in\Omega}f(x),
\end{equation}
where $\Omega$ is a (not necessarily convex) closed domain with $\max_{x,y\in\Omega}\|x-y\|_2\leqslant1$ (which implies $\mu(\Omega)=\int_{\Omega}\ud x\leqslant1$); and especially, assume that observing $f$ is costly. When $f$ is cheap to evaluate, there are many feasible methods, such as genetic algorithms \citep{MitchellM1996M_GeneticAlgorithm}, evolution strategies \citep{SchwefelH1995_ES}, differential evolution \citep{StornR1997M_DifferentialEvolution} and simulated annealing \citep{KirkpatrickS1983M_SimulatedAnnealing}. But when $f$ is expensive, we have to pay more attention to how to maximize the use of information obtained.

Bayesian optimization (BO), which is a sequential model-based approach and the model is often obtained using a Gaussian process, is a typical strategy to use existing information. It was first introduced by \citet{MockusJ1974M_BO} and later popularized by \citet{JonesD1998M_BO}. The BO method applies the Gaussian process to construct a model for generating the acquisition functions. The acquisition function, which trade-offs exploration and exploitation, is used to determine the next candidate point. Theoretical results on the convergence behaviour of BO is provided in \citet{BullA2011T_ConvergenceRatesOfBO}. Various extensions have been suggested by further authors, and a recent review can be found in \citet{ShahriariB2016R_BO}.

However, BO always updates the model on the original domain and uses all the historical samples \citep{JonesD1998M_BO,JonesD1998M_EGOpt,
KleijnenJ2012M_BO}, the computational cost of modeling will continue to increase. To overcome this, we first introduce the concept of \emph{contractibility}, which extends from the characteristics of the level sets. Actually, the problem \eqref{CM:eq:COP} is closely related to the $u$-sublevel set \citep{RockafellarR1970M_ConvexAnalysis,
LinQ2018_LSM,AravkinA2019_LSM}, i.e.,
\begin{equation}\label{CM:eq:LS}
E(u)=\{x\in\Omega:f(x)\leqslant u\},~~u\in\Big[f^*,\max_{x\in\Omega}f(x)\Big].
\end{equation}
Its boundary set $\partial E(u)=\{x\in\Omega:f(x)=u\}$ is a contour surface and $E(u)$ contracts monotonically from $\Omega$ to $X^*$ when $u$ continuously decreases from $\max_{x\in\Omega}f(x)$ to $f^*$. More precisely, let $\max_{x\in\Omega}f(x)\geqslant u^{(k)}>u^{(k+1)}>f^*$, then
\begin{equation*}
  \Omega\supseteq E(u^{(k)})\supset E(u^{(k+1)})\supset X^*.
\end{equation*}
In fact, a sequence that satisfies a similar inclusion relation can be fully independent of the concept of level set. Hence, we tend to use the key inclusion relation to define a contraction sequence or contraction sets. It emphasizes that any global minimizer cannot be excluded in the reduction of sets; otherwise, the global convergence will not be guaranteed.
\begin{defn}[Contraction sets]\label{CM:defn:CS}
For problem \eqref{CM:eq:COP}, a contraction sequence $\{D^{(k)}\}_{k\in\mathbb{N}_0}$ is a sequence of decreasing nonempty closed sets satisfying $\Omega\supseteq D^{(k)}\supset D^{(k+1)}\supset X^*$; and further, we call $\{D^{(k)}\}_{k\in\mathbb{N}_0}$ a \textbf{strictly} contraction sequence if they also satisfy
\begin{equation}\label{CM:eq:SCS}
\max_{x\in D^{(k+1)}}f(x)<\max_{x\in D^{(k)}}f(x).
\end{equation}
\end{defn}

Now, we can translate the original optimization problem into a construction problem of a strictly contraction sequence $\{D^{(k)}\}_{k\in\mathbb{N}_0}$ with $D^{(0)}=\Omega$. Each contraction set $D^{(k+1)}$ could be sequentially determined by an approximation model constructed from samples on $D^{(k)}$. Then, a \emph{contraction method} (CM) is defined as a model-based approach to construct such a contraction sequence. Obviously, CM helps control the computational cost of modelling, since it updates the model on the gradually decreasing contracted sets and uses only those samples that are located in the contracted sets. Due to this characteristic, a CM might also be viewed as a special case of classical branch and bound techniques \citep{LawlerE1966R_Branch&Bound,TornA1989M_GlobalOptimization}, which is aimed to sequentially guarantee $X^*\subset D^{(k+1)}$ and exclude $\left\{D^{(k)}-D^{(k+1)}\right\}_k$.

More importantly, we further discuss what conditions can ensure that the contraction could be carried out effectively. Some conditions, especially \emph{independent of smoothness}, are introduced to divide all continuous optimization problems into three reasonable categories. Therefore, the contractibility can be seen as a complement to smoothness for distinguishing the continuous optimization problems that are easy to solve.

\subsection{Related Work}

Here we briefly discuss the relationships between the contractibility and existing ideas, some of which are related to the contraction strategies, some are related to the conditions that should be met during contraction, and others are related to the generation of new samples in the contraction set. It is not difficult to see that, although the concept of contractibility has not been considered formally and completely before, some prototypes of similar thinking have been already used in related fields.

\textbf{Grid search and random search.} Most machine learning algorithms come with some hyperparameters that control their behavior. And automatic hyperparameter optimization is one of the problems we are discussing due to its high computational cost. Generally, grid search is the most widely used strategy when there are three or fewer hyperparameters, because its computational cost increases
exponentially with the number of hyperparameters \citep{Goodfellow2016M}. In one dimension, when the number of grid nodes is doubled, the upper bound of the search accuracy is stably reduced to one half of the previous amount. But in general, grid search has the advantage of finding more accurate solutions at the cost of much higher computation time \citep{ReifM2012_HyperParaOpt}. Random search is a convenient and more effective alternative to grid search in multi-dimensional cases \citep{BergstraJ2012M_RS}. A practical trick they share is to recursively refine the search by reducing the search space based on the results of the previous run. Since the finer discrete search space is often centered on the results of the previous run, this local refinement trick is easy to implement but a bit too simple. It can be somewhat seen as an informal form of contraction strategy. In a strict sense, of course, a satisfactory strategy should include an essential requirement to ensure that at least one specific optimal solution is always covered by the next search space. Although slightly different, this requirement can almost be considered the contraction condition mentioned above.

\textbf{Bayesian optimization.} As a similar model-based approach, we have mentioned that BO has a characteristic of relying on all historical samples in modeling \citep{JonesD1998M_BO, ShahriariB2016R_BO}. This is one of the motivations for us to introduce contractibility. But here we focuses on the posterior distribution of model prediction in BO, which is related to the contraction condition we shall consider later. First, a BO method requires a choice of Gaussian process (GP) prior related to a reproducing-kernel Hilbert space (RKHS). Then for a selected prior and given historical sample set, the posterior can be established to determine the next candidate point, so that the minimum can be found with theoretical guarantee for an arbitrary function in its RKHS \citep{Srinivas2010_BO,BullA2011T_ConvergenceRatesOfBO,
Scarlett2017_LowerBoundonBO}. However, from a different perspective, instead of using the posterior to select the candidate point, but gradually excluding domains that do not contain the optimal solution, one could establish a contraction strategy and corresponding contraction conditions. And it is worth pointing out that, we do not continue the way from a GP prior to a posterior, but introduce the hierarchical low-frequency dominant function (HLFDF) as a fundamental assumption from the characteristics of contraction, then establish the relevant contraction condition by cross validation (CV). The HLFDFs do not require a bounded RKHS norm and cover many Lipschitz continuous or even H\"{o}lder continuous functions that commonly arise in machine learning, while the CV strategy avoids the requirement to choose a GP prior as well as the smoothness requirement for many priors in practice.

\textbf{Level Set Estimation.} The level set estimation (LSE) algorithm is mainly developed to determine the set of points, for which an unknown function takes value above or below some given threshold level, from a fixed discrete search space \citep{Gotovos2013_LSE}. The main idea is the same as BO: a GP prior is selected and then the corresponding posterior is established from all historical samples by the Bayesian inference to deal with this underlying classification problem. \cite{Bogunovic2016_LSE&BO} discussed the connection between BO and LSE in a unified way. For continuous search spaces, the LSE problems are also considered in communities such as reliability engineering \citep{AzzimontiD2021_LSE}. Analogous to LSE, an algorithm of gradually reducing a fixed finite discrete search space is proposed based on the Bayesian inference with a GP prior \citep{deFreitas2012_contraction}, and then further expanded to tree-based approaches \citep{Wang2014_contraction&tree,ShekharS2018A_contraction&tree}. Since these methods mentioned above also use GP priors and all historical samples, they do not gain a better convergence rate and lower model training costs than BO, except that they do reduce searches when optimizing the model established already. Furthermore, although these ideas are a little bit close to the proposed contraction method, there are still important things left: (i) these algorithms still depend on all historical samples; (ii) the abstract inclusion relationship of contraction sets has not been formally extracted; (iii) these assumptions include GP priors, some of which imply the uniqueness of the global minimum; and (iv) most algorithms lack a flexible sampling strategy so that they rely on a fixed discrete lattice. The first one mainly leads to a monotonic increase in training costs, while the latter three limit further development and the scope of various applications. In addition, one of the main reasons for restricting the development of contractibility may be the GP assumption itself, that is, correlation supports the viewpoint that points outside a contraction set can reveal more information about the objective function in the set. However, from the approximation theory of functions in bandlimited or reproducing-kernel Hilbert spaces, except near the boundary, points outside a certain domain have no effect on the approximation accuracy for any given function within this domain, because the accuracy depends only and essentially on the density of the samples \citep{NarcowichF2004A_BandLimited&RBF,BonamiaA2017T_SpectralDecayofBandlimited}. Therefore, from the basic assumption to inference method, there are significant differences between the proposed framework and existing ideas based on Gaussian processes.

\textbf{Heuristic search.} Not limited to model-based methods, heuristic global optimization algorithms \citep{RechenbergI1973B_EvolutionsStrategie,
KirkpatrickS1983M_SimulatedAnnealing,MitchellM1996M_GeneticAlgorithm,
StornR1997M_DifferentialEvolution} also attempt to trade-off exploration and exploitation in a fully different way. Obviously, sufficient exploration provides a guarantee of convergence while exploitation of limited knowledge is the key to improve efficiency. However, in fact, every exploitation on the basis of inadequate information may reduce efficiency or even cause trouble. Although there is a certain randomness in ensuring the adequacy of exploration, heuristic strategies are also used for hyperparameter optimization due to its simplicity and ease of implementation \citep{ReifM2012_HyperParaOpt}. Some of them sometimes have surprising performance, but they converge more slowly than model-based methods in many cases. Nevertheless, they do not require additional computational costs to maintain a sequential model. As mentioned above, the additional cost for the proposed framework is not monotonically increasing as long as the contraction can be executed. Therefore, the competition between contraction methods and heuristic algorithms also needs to consider the computational complexity of the objective function and the multiple costs of repeated runs in order to ensure global convergence.

\subsection{Contributions}

Our main contributions in this work are as follows:
\begin{enumerate}
  \item We formally introduce the concept of contractibility and the definition of contraction sets, so that any continuous optimization problem over an arbitrary closed domain can be described as a construction problem of contraction sets. The proposed contraction sets are developed from and include the level sets, but are fully independent of them.
  \item We propose a framework for constructing contraction sets and further establish a basic convergence conclusion (Theorem \ref{CM:thm:Conv}) to ensure that every sequence constructed by this framework is a sequence of contraction sets. Furthermore, the strong convergence conditions are also introduced and guarantee that the constructed sequence enjoys a linearly convergent upper bound (Theorems \ref{CM:thm:SConv} and \ref{CM:thm:SConv2}).
  \item From the perspective of contraction, relevant conditions are introduced to divide all continuous optimization problems into three categories. For every problem from the first two categories related to HLFDFs, we show that there is a contraction sequence that satisfies the expected complexity (Theorems \ref{CM:thm:LTC} and \ref{CM:thm:PTC}). And for any problem from the last category, we discuss possible troubles caused by contraction and consider the complexity bounds for model-based methods (Theorems \ref{CM:thm:sSmooth} and \ref{CM:thm:HC}). The CM is independent of any Gaussian process prior, that is, the objective does not require a bounded RKHS norm. Moreover, we show that the model required in the method can be satisfactorily constructed by kernel-based interpolations (Lemma \ref{CM:lem:HLFDF&KI}).
  \item Due to the requirement to expand an existing quasi-uniform sample set in any closed set, we consider a sampling strategy, then analyze why this strategy can continuously generate quasi-uniform samples and what quality they can maintain (Theorem \ref{CM:thm:QU}). In addition, we also discuss the problem of spatial discretization related to sampling.
  \item Based on the contraction framework, an algorithm is developed with high probability bounds for convergence rate and complexity (Theorems \ref{CM:thm:LTC2} and \ref{CM:thm:PTC2}). The running process of this algorithm is also the process of identifying the characteristics of the objective, i.e., if the contractions are continuously executed on an objective, then it must belong to the HLFDF class. Generally, most black-box optimization problems with expensive function evaluations will benefit from the proposed algorithm, unless no contraction is performed at all. Numerical comparisons demonstrate the expected benefits.
\end{enumerate}

\subsection{Paper Organization}

The remainder of the paper is organized as follows. The next section first establishes a basic framework of CMs. Then three assumptions and some lemmas related to the contractibility are introduced in details in Section \ref{CM:s3}. And these conditions, especially the hierarchical low-frequency dominant property, allow us to divide all possible continuous problems into three categories in Section \ref{CM:s4}. Since a practical algorithm needs to expand an existing quasi-uniform sample set in any closed domain, we further discuss the theory of sampling strategy in Section \ref{CM:s5}. In Section \ref{CM:s6}, a specific contraction algorithm is considered in detail, and we demonstrate the benefits of the algorithm by several numerical experiments and comparisons in Section \ref{CM:s7}. Finally, we draw some conclusions in Section \ref{CM:s8}.

\section{Framework of contraction methods}
\label{CM:s2}

Due to the potential requirements of modeling, we start with the concept of quasi-uniformity. For any given closed subdomain $D\subset\Omega$, a fixed sample set $\chi=\{\chi_i\}_{i=1}^N$ over $D$ is called quasi-uniform with uniformity constant $\tau>0$, if the inequality
\begin{equation}\label{CM:eq:h&q}
  \frac{1}{\tau}q_\chi\leqslant h_{D,\chi}\leqslant\tau q_\chi
\end{equation}
holds, where
\begin{equation}\label{CM:eq:h}
  h_{D,\chi}:=\sup_{x\in D}\min_{\chi_i\in\chi}\|x-\chi_i\|_2
\end{equation}
is called the fill distance of $\chi$ with respect to $D$ describing the geometric relation of the set $\chi$ to the bounded domain $D$, and
\begin{equation}\label{CM:eq:q}
  q_\chi:=\frac{1}{2}\min_{\chi_i\neq\chi_j}\|\chi_i-\chi_j\|_2
\end{equation}
is called the separation distance of $\chi$. The uniformity constant $\tau$ provides a measure of how uniformly points in $\chi$ are distributed in $D$. When $n=1$, that is, $D$ is an interval, $\tau=1$ means that the point set is an equidistant grid of nodes. And in all other cases, $\tau>1$. When the size $N$ is fixed, the smaller $\tau$ means a smaller $h_{D,\chi}$ as well as a larger $q_\chi$. Recall that the fill distance and the separation distance are two fundamental contributory factors for standard error and stability estimates for multivariate interpolants \citep{WuZ1993A_ErrorEstimatesRBF,
SchabackR1995A_ErrorEstimates&ConditionNumbersRBF,WendlandH2005B_ScatteredData}.

\subsection{Method}
As mentioned above, we want to establish a model-based optimization method to construct a sequence of decreasing nonempty closed sets containing all global minimizers. Regardless of how to obtain the approximate models and quasi-uniform samples, the framework can be formally described as follows.

\begin{defn}\label{CM:defn:D}
For problem \eqref{CM:eq:COP} and any $\omega\in(0,1]$, a model-based sequence $\{D^{(k)}\}_{k\in\mathbb{N}_0}$ is defined recursively by $D^{(0)}=\Omega$ and
\begin{equation}\label{CM:eq:D}
D^{(k+1)}=\left\{x\in D^{(k)}:\mathcal{A}^{(k)}f(x)\leqslant u^{(k)}\right\},
~~\forall u^{(k)}\in\Big[f_{\chi^{(k)}}^*,f_{\chi^{(k)}}^{**}\Big],
\end{equation}
where $\chi^{(k)}= \{\chi^{(k)}_i\}$ are quasi-uniformly distributed over $D^{(k)}$ with the size $N^{(k)}$ and the inheritance relationship $\chi^{(k)}\cap\chi^{(k+1)}=\chi^{(k)}\cap D^{(k+1)}$, the relevant data values $f_{\chi^{(k)}}=\{f(\chi^{(k)}_i)\}$ with $f_{\chi^{(k)}}^*=\min(f_{\chi^{(k)}})$ and $f_{\chi^{(k)}}^{**}=\max(f_{\chi^{(k)}})$, and $\mathcal{A}^{(k)}f$ is an approximation of $f$ w.r.t. the data pairs $(\chi^{(k)},f_{\chi^{(k)}})$ such that the following error bound condition holds, i.e.,
\begin{equation}\label{CM:eq:DC}
\max_{x\in D^{(k)}}\left|\mathcal{A}^{(k)}f(x)-f(x)\right|
\leqslant\omega\Big(u^{(k)}-f_{\chi^{(k)}}^*\Big).
\end{equation}
\end{defn}
\begin{rem}
The inheritance relationship $\chi^{(k)}\cap\chi^{(k+1)}=\chi^{(k)}\cap D^{(k+1)}$ guarantees that any point in $\chi^{(k)}$ is preserved if it is located in $D^{(k+1)}$. In other words, $\chi^{(k+1)}$ is further expanded from $\chi^{(k)}\cap D^{(k+1)}$; that is, $\chi^{(k+1)}/(\chi^{(k)}\cap D^{(k+1)})$ is the newly added sample set at step $k$. This is the way that the CM uses historical samples.
\end{rem}
\begin{rem}
Obviously, for all $c\in(0,100)$, $u^{(k)}=\textrm{prctile} \big(f_{\chi^{(k)}},c\big)$ is a feasible choice, where $\textrm{prctile}\big(f_{\chi^{(k)}}, c\big)$ is the percentile of $f_{\chi^{(k)}}$ for the percentage $c$.
\end{rem}
\begin{rem}
The bound parameter $\omega$ is used to enhance convergence, see Theorem \ref{CM:thm:SConv}. And the condition \eqref{CM:eq:DC} is the key for ensuring a sufficient exploration and issuing a judgment on the conversion of exploration to exploitation. Statistical methods allow us to estimate the model errors in a sense of probability, and we will discuss this in Subsection \ref{CM:s6:EME}; noting that a statistical estimator will lead to a high probability bound in each contraction, we will also consider how to make the union bound hold with a fixed probability.
\end{rem}

Let $\mu(S)=\int_S\mathrm{d}t$ denote the $n$-dimensional Lebesgue measure of $S\subset\mathbb{R}^n$, then we say that the ratio
\begin{equation*}
\lambda^{(k+1)}=\frac{\mu(D^{(k+1)})}{\mu(D^{(k)})}
\end{equation*}
is the $(k+1)$th contraction factor. If $u^{(k)}$ is chosen as $\textrm{prctile} \big(f_{\chi^{(k)}},c\big)$, under Assumption \ref{CM:ass:CR} (see Section \ref{CM:s3}), the sequence $\{D^{(k)}\}_{k\in\mathbb{N}_0}$ has a contraction factor $\lambda^{(k+1)}=\frac{c}{100}$ in expectation for a percentage $c\in(0,100)$. Obviously, a large factor leads to a slow contraction, but fewer function evaluations are required in each step, and vice versa. A typical choice is the median-type, i.e., $c=50$; for \emph{illustrative examples} see Subsection \ref{CM:s2:3}.

\subsection{Convergence}
Now we show that every $\{D^{(k)}\}$ constructed by Definition \ref{CM:defn:D} is a contraction sequence.
\begin{thm}[Convergence]\label{CM:thm:Conv}
Suppose the sequence $\{D^{(k)}\}_{k\in\mathbb{N}_0}$ is constructed as Definition \ref{CM:defn:D} for problem \eqref{CM:eq:COP}. Then, for any possible choices of $u^{(k)}$ and $\omega$, if $f$ is not a constant, then for all $k\in\mathbb{N}_0$,
\begin{equation*}
 \Omega\supseteq D^{(k)}\supset D^{(k+1)}\supset X^*,
\end{equation*}
that is, $\{D^{(k)}\}_{k\in\mathbb{N}_0}$ is a contraction sequence, as defined in Definition \ref{CM:defn:CS}.
\end{thm}
\begin{proof}
The convergence proceeds by induction on $k$. First, $\Omega=D^{(0)}\supset X^*$ is trivial since $f$ is not a constant on $\Omega$. Assume that $D^{(k)}\supset X^*$, then $D^{(k)}\supset S^{(k)}\supset X^*$, where
\begin{equation*}
 S^{(k)}=\left\{x\in D^{(k)}:f(x)\leqslant f^*_{\chi^{(k)}}\right\};
\end{equation*}
now we will show that $D^{(k)}\supset D^{(k+1)}\supset S^{(k)}$.

It is clear that $D^{(k)}\supset D^{(k+1)}$, so we need to prove that $D^{(k+1)}\supset S^{(k)}$. By Definition \ref{CM:defn:D}, the approximate model can be decomposed into
\begin{equation*}
 \mathcal{A}^{(k)}f(x)=f(x)+\varepsilon^{(k)}(x),
\end{equation*}
and since the error bound condition, it follows that
\begin{equation*}
 |\varepsilon^{(k)}(x)|\leqslant\omega\big(u^{(k)}-f_{\chi^{(k)}}^*\big)
 \leqslant u^{(k)}-f_{\chi^{(k)}}^*,~~\omega\in(0,1].
\end{equation*}
Hence, for any $x'\in S^{(k)}$, we have $f(x')\leqslant f^*_{\chi^{(k)}}$, which can be further rewritten as
\begin{equation*}
 \mathcal{A}^{(k)}f(x')=f(x')+\varepsilon^{(k)}(x')
 \leqslant f^*_{\chi^{(k)}}+|\varepsilon^{(k)}(x')|,
\end{equation*}
and further,
\begin{equation*}
 \mathcal{A}^{(k)}f(x')\leqslant f^*_{\chi^{(k)}}+u^{(k)}-f_{\chi^{(k)}}^*=u^{(k)},
\end{equation*}
therefore, $x'\in D^{(k+1)}$, that is, $D^{(k+1)}\supset S^{(k)}$, and the proof is complete.
\end{proof}

Theorem \ref{CM:thm:Conv} does not guarantee that the upper bound of $f$ on $D^{(k)}$, say, $\max_{x\in D^{(k)}}f(x)$, is strictly monotonically decreasing. In fact, by choosing suitable $u^{(k)}$ and $\omega\in(0,1]$, the corresponding convergence can be further enhanced. First, the following lemma gives an upper bound of the optimality gap on $D^{(k+1)}$.
\begin{lem}[Upper bound]\label{CM:lem:SConv}
Suppose the sequence $\{D^{(k)}\}_{k\in\mathbb{N}_0}$ is constructed as Definition \ref{CM:defn:D} for problem \eqref{CM:eq:COP}. Then, for any possible choices of $u^{(k)}$ and $\omega$, if $f$ is not a constant, then for all $k\in\mathbb{N}_0$,
\begin{equation*}
 \max_{x\in D^{(k+1)}}\big[f(x)-f^*\big]\leqslant(1+\omega)\big(u^{(k)}-f^*\big).
\end{equation*}
\end{lem}
\begin{rem}
This lemma also reflects the effect of the error bound parameter $\omega$, which makes the error bound condition relatively independent of $D^{(k+1)}$. Note that both the error bound condition and $D^{(k+1)}$ depend on $u^{(k)}$ and $\mathcal{A}^{(k)}f$, but the condition also relies on $\omega$, and $D^{(k+1)}$ is independent of $\omega$.
\end{rem}
\begin{proof}
From Definition \ref{CM:defn:D} and $f^*\leqslant f_{\chi^{(k)}}^*$, the approximate model can be decomposed into
\begin{equation*}
 \mathcal{A}^{(k)}f(x)=f(x)+\varepsilon^{(k)}(x)~~\textrm{with}~~
 |\varepsilon^{(k)}(x)|\leqslant\omega\big(u^{(k)}-f_{\chi^{(k)}}^*\big)
 \leqslant\omega\big(u^{(k)}-f^*\big),
\end{equation*}
then for all $t\in D^{(k+1)}$, since
\begin{equation*}
 \mathcal{A}^{(k)}f(t)\leqslant u^{(k)},
\end{equation*}
it follows that
\begin{equation*}
 f(t)=\mathcal{A}^{(k)}f(t)-\varepsilon^{(k)}(t)
 \leqslant u^{(k)}+\omega\big(u^{(k)}-f^*\big),
\end{equation*}
subtracting $f^*$ from both sides, this yields
\begin{equation*}
 f(t)-f^*\leqslant u^{(k)}-f^*+\omega\big(u^{(k)}-f^*\big)
 \leqslant(1+\omega)\big(u^{(k)}-f^*\big),
\end{equation*}
noting that the last inequality holds for all $t\in D^{(k+1)}$, we obtain
\begin{equation*}
 \max_{x\in D^{(k+1)}}[f(x)-f^*]\leqslant(1+\omega)\big(u^{(k)}-f^*\big),
\end{equation*}
as claimed.
\end{proof}

In the following, we establish the first strong convergence conclusion. It not only gives the conditions for strong convergence, but also reflects the relationship between $u^{(k)}$ and $\omega$ in a strong convergence behavior and their respective roles. Notice that the following condition \eqref{CM:eq:CondSC1} depends explicitly on the unknown $f^*$ and $\max_{x\in D^{(k)}}f(x)$, but this will be released in Theorem \ref{CM:thm:SConv2}.

\begin{thm}[Strong convergence]\label{CM:thm:SConv}
Suppose the sequence $\{D^{(k)}\}_{k\in\mathbb{N}_0}$ is constructed as Definition \ref{CM:defn:D} for problem \eqref{CM:eq:COP} and $f$ is not a constant. Then, if there exists a $0<q<\infty$ such that
\begin{equation}\label{CM:eq:CondSC1}
 u^{(k)}-f^*\leqslant\frac{1}{1+q}\max_{x\in D^{(k)}}[f(x)-f^*],
\end{equation}
and the parameter $\omega\in(0,1]$ satisfies
\begin{equation}\label{CM:eq:CondSC2}
 0<\omega<q,
\end{equation}
then for all $k\in\mathbb{N}_0$,
\begin{equation}\label{CM:eq:LCR}
 \max_{x\in D^{(k+1)}}[f(x)-f^*]\leqslant
 \frac{1+\omega}{1+q}\max_{x\in D^{(k)}}[f(x)-f^*].
\end{equation}
Here, \eqref{CM:eq:CondSC1} and \eqref{CM:eq:CondSC2} are called the strong convergence conditions, and \eqref{CM:eq:LCR} is called linear convergence with factor $\frac{1+\omega}{1+q}<1$.
\end{thm}
\begin{proof}
According to \eqref{CM:eq:CondSC2}, we have $\frac{1+\omega}{1+q}<1$, then it follows from Lemma \ref{CM:lem:SConv} and \eqref{CM:eq:CondSC1} that
\begin{equation*}
 \max_{x\in D^{(k+1)}}[f(x)-f^*]\leqslant(1+\omega)\big(u^{(k)}-f^*\big)
 \leqslant\frac{1+\omega}{1+q}\max_{x\in D^{(k)}}[f(x)-f^*],
\end{equation*}
as claimed.
\end{proof}

The conclusion above is not only convenient for theoretical analysis in Section \ref{CM:s4} but also generalized for practical applications. Before doing this, let us extend our model-based sequence to the case of unfixed bound parameter, i.e., $\omega_k$.

\begin{defn}\label{CM:defn:D2}
For problem \eqref{CM:eq:COP}, possible sequences $\{u^{(k)}\}_{k\in\mathbb{N}_0}$ and $\{\omega_k\}_{k\in\mathbb{N}_0}$, a model-based sequence $\{D^{(k)}\}_{k\in\mathbb{N}_0}$ is defined recursively by $D^{(0)}=\Omega$ and
\begin{equation}\label{CM:eq:D2}
D^{(k+1)}=\left\{x\in D^{(k)}:\mathcal{A}^{(k)}f(x)\leqslant u^{(k)}\right\},
\end{equation}
where, except $\omega_k$, the remaining notations come from Definition \ref{CM:defn:D}, and $\mathcal{A}^{(k)}f$ is a model of $f$ w.r.t. the data pairs $(\chi^{(k)},f_{\chi^{(k)}})$ such that the upper bound of model error
\begin{equation}\label{CM:eq:DC2}
\max_{x\in D^{(k)}}\big|\mathcal{A}^{(k)}f(x)-f(x)\big|
\leqslant\omega_k\big(u^{(k)}-f_{\chi^{(k)}}^*\big).
\end{equation}
\end{defn}

Note that Lemma \ref{CM:lem:SConv} also holds for the sequence given by the above definition. Further, the following theorem, which does not rely on any unknown information, may be helpful to make a trade-off between $u^{(k)}$ and $\omega_k$, e.g., choose a suitable $\omega_k$ for a certain $u^{(k)}$ or vice versa, in each contraction.

\renewcommand{\thethma}{2a}
\begin{thma}[Strong convergence]\label{CM:thm:SConv2}
For problem \eqref{CM:eq:COP}, possible sequences $\{u^{(k)}\}$ and $\{\omega_k\}$, suppose a sequence $\{D^{(k)}\}_{k\in\mathbb{N}_0}$ is constructed as Definition \ref{CM:defn:D2}, $f$ is not a constant, and there exists a sequence of positive real numbers $\{q_k\}$ such that
\begin{equation}\label{CM:eq:CondSCa1}
 u^{(k)}-\min\big(A_k^*,f_{\chi^{(k)}}^*\big)
 \leqslant\frac{1}{1+q_k}\big(f_{\chi^{(k)}}^{**}-f_{\chi^{(k)}}^*\big),
 ~\textrm{where}~A_k^*=\!\min_{x\in D^{(k)}}\!\mathcal{A}^{(k)}f(x).
\end{equation}
If the parameter $\omega_k$ satisfies
\begin{equation}\label{CM:eq:CondSCa2}
 (1+\omega_k)^2<1+q_k,
\end{equation}
then for all $k\in\mathbb{N}_0$,
\begin{equation}\label{CM:eq:LCRa}
 \max_{x\in D^{(k+1)}}[f(x)-f^*]\leqslant
 \frac{(1+\omega_k)^2}{1+q_k}\max_{x\in D^{(k)}}[f(x)-f^*].
\end{equation}
\end{thma}
\begin{proof}
For a certain $k$ and any $x^*\in X^*\subset D^{(k)}$, it is clear that
\begin{equation*}
  A_k^*=\min_{x\in D^{(k)}}\mathcal{A}^{(k)} f(x)\leqslant\mathcal{A}^{(k)}f(x^*),
\end{equation*}
subtracting $f^*$ from both sides, we get
\begin{equation*}
  A_k^*-f^*\leqslant\mathcal{A}^{(k)}f(x^*)-f^*
  \leqslant\big|\mathcal{A}^{(k)}f(x^*)-f^*\big|
  \leqslant\max_{x\in D^{(k)}}\big|\mathcal{A}^{(k)}f(x)-f(x)\big|.
\end{equation*}
Together with the upper bound of model error, that is,
\begin{equation*}
 \max_{x\in D^{(k)}}\big|\mathcal{A}^{(k)}f(x)-f(x)\big|
 \leqslant\omega_k\big(u^{(k)}-f_{\chi^{(k)}}^*\big),
\end{equation*}
it holds that
\begin{equation*}
 A_k^*-f^*\leqslant\omega_k\big(u^{(k)}-f_{\chi^{(k)}}^*\big),
\end{equation*}
that is,
\begin{equation*}
 -f^*\leqslant-A_k^*+\omega_k\big(u^{(k)}-f_{\chi^{(k)}}^*\big).
\end{equation*}
Then from Lemma \ref{CM:lem:SConv} and the inequality above, we obtain
\begin{align*}
 \max_{x\in D^{(k+1)}}[f(x)-f^*]\leqslant&~(1+\omega_k)\big(u^{(k)}-f^*\big) \\
 \leqslant&~(1+\omega_k)\left[u^{(k)}-A_k^* +\omega_k\big(u^{(k)}-f_{\chi^{(k)}}^*\big)\right] \\
 \leqslant&~(1+\omega_k)\big(u^{(k)}-A_k^*\big)+
 (\omega_k+\omega_k^2)\big(u^{(k)}-f_{\chi^{(k)}}^*\big) \\
 \leqslant&~(1+\omega_k)^2\big[u^{(k)}-\min\big(A_k^*,f_{\chi^{(k)}}^*\big)\big].
\end{align*}
Together with \eqref{CM:eq:CondSCa1}, we get
\begin{equation*}
  \max_{x\in D^{(k+1)}}[f(x)-f^*]\leqslant\frac{(1+\omega_k)^2}{1+q_k}
  \big(f_{\chi^{(k)}}^{**}-f_{\chi^{(k)}}^*\big)\leqslant
  \frac{(1+\omega_k)^2}{1+q_k}\max_{x\in D^{(k)}}[f(x)-f^*],
\end{equation*}
and \eqref{CM:eq:CondSCa2} ensured the factor $\frac{(1+\omega_k)^2}{1+q_k}<1$, so the desired result follows.
\end{proof}

\subsection{Median-type method and two illustrative examples}
\label{CM:s2:3}

Inspired by the traditional halving method, we also consider the median-type method, i.e., $u^{(k)}=\textrm{prctile}\big(f_{\chi^{(k)}},c_k\big)$ with a fixed $c_k=\bar{c}=50$, which leads to a contraction factor $\lambda^{(k+1)}=\frac{1}{2}$ in expectation under Assumption \ref{CM:ass:CR}. In this case, the median-type sequence $\{D^{(k)}\}_{k\in\mathbb{N}_0}$ can be rewritten by $D^{(0)}=\Omega$ and
\begin{equation}\label{CM:eq:DMed}
D^{(k+1)}=\left\{x\in D^{(k)}:\mathcal{A}^{(k)}f(x)\leqslant \textrm{Median}f_{\chi^{(k)}}\right\},
\end{equation}
where the $k$th model $\mathcal{A}^{(k)}f$ satisfies
\begin{equation*}
\max_{x\in D^{(k)}}\left|\mathcal{A}^{(k)}f(x)-f(x)\right|< \omega\left(\textrm{Median}f_{\chi^{(k)}}-f_{\chi^{(k)}}^*\right),
~~\forall\omega\in(0,1].
\end{equation*}

As an illustrative example, Figure \ref{CM:fig:A1} visually shows how the median-type algorithm (see Section \ref{CM:s5}) compresses the original domain $\Omega$ to $X^*$ step by step; and three important pieces of information that can be obtained from Figure \ref{CM:fig:A1} are: (i) the sample size required for each model is greatly reduced due to the domain contractions; and (ii) none of the minimizers will be missed for the cases of multiple global minima, which verifies the conclusion established by Theorem \ref{CM:thm:Conv}; and (iii) this median-type method enjoys linear convergence, which verifies the conclusions established by Theorems \ref{CM:thm:SConv} and \ref{CM:thm:SConv2}.

\begin{figure}[tbhp]
\centering
\subfigure{\includegraphics[width=0.325\textwidth]{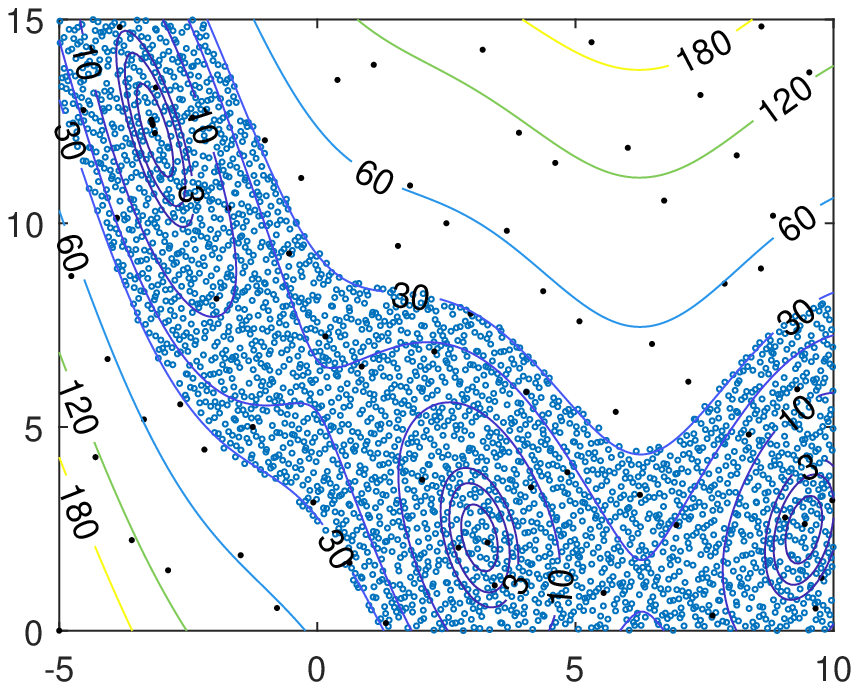}}
\subfigure{\includegraphics[width=0.325\textwidth]{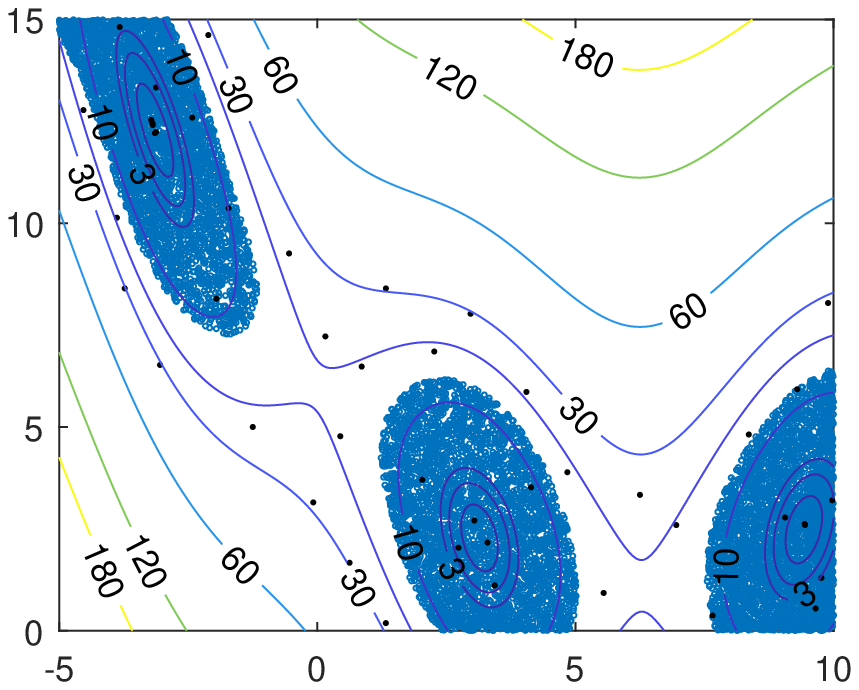}}
\subfigure{\includegraphics[width=0.325\textwidth]{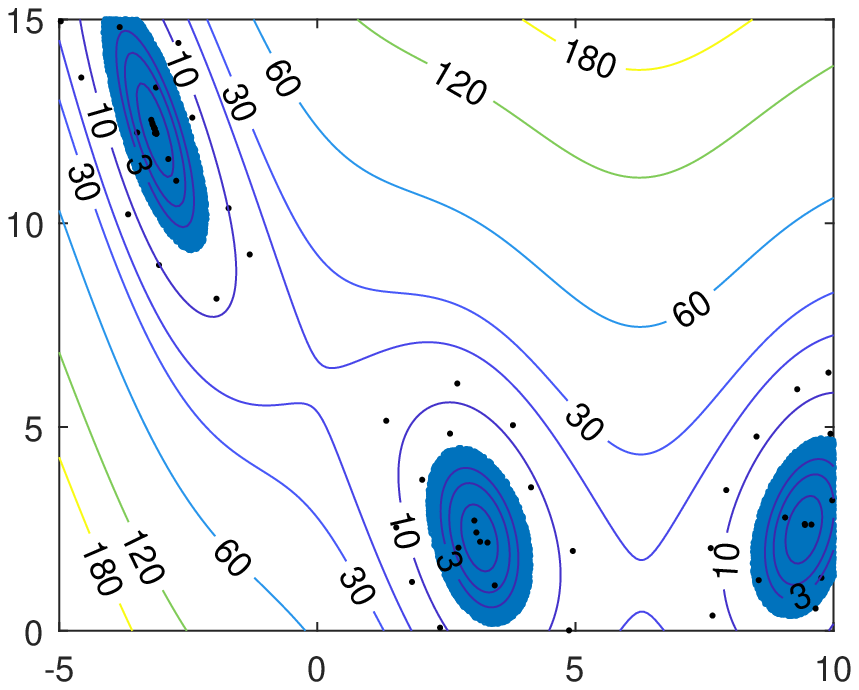}}\\
\subfigure{\includegraphics[width=0.325\textwidth]{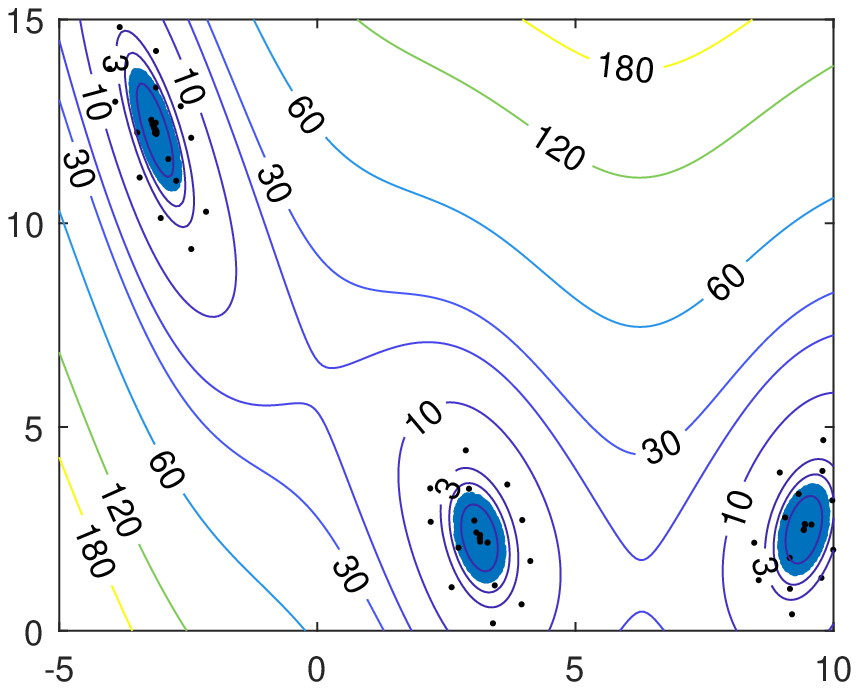}}
\subfigure{\includegraphics[width=0.325\textwidth]{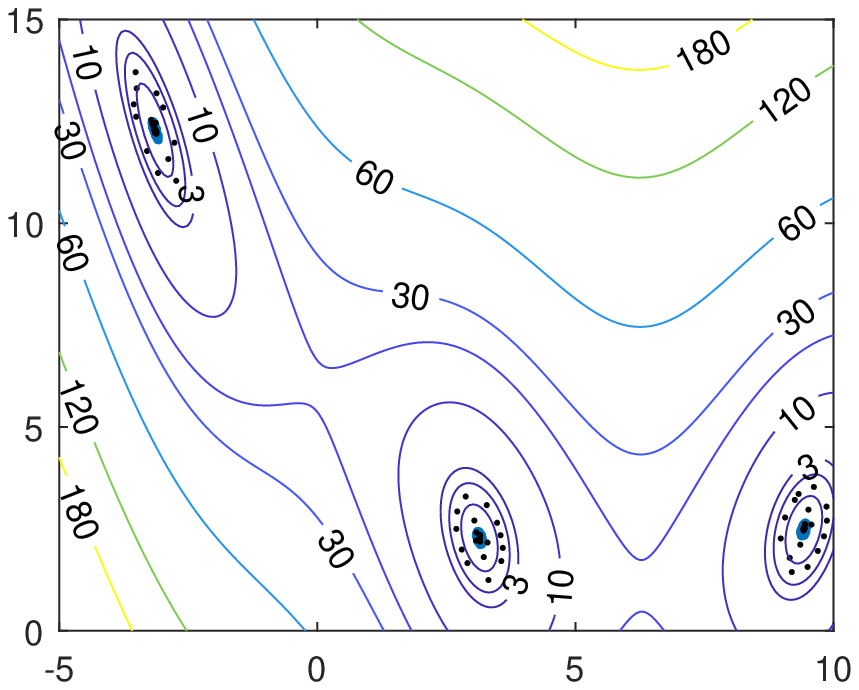}}
\subfigure{\includegraphics[width=0.325\textwidth]{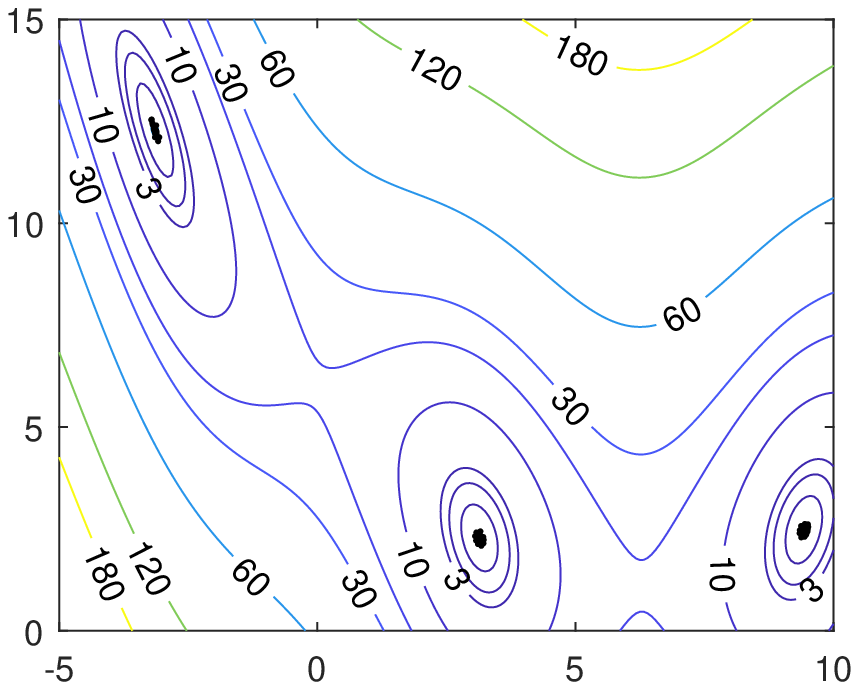}}\\
\subfigure{\includegraphics[width=0.325\textwidth]{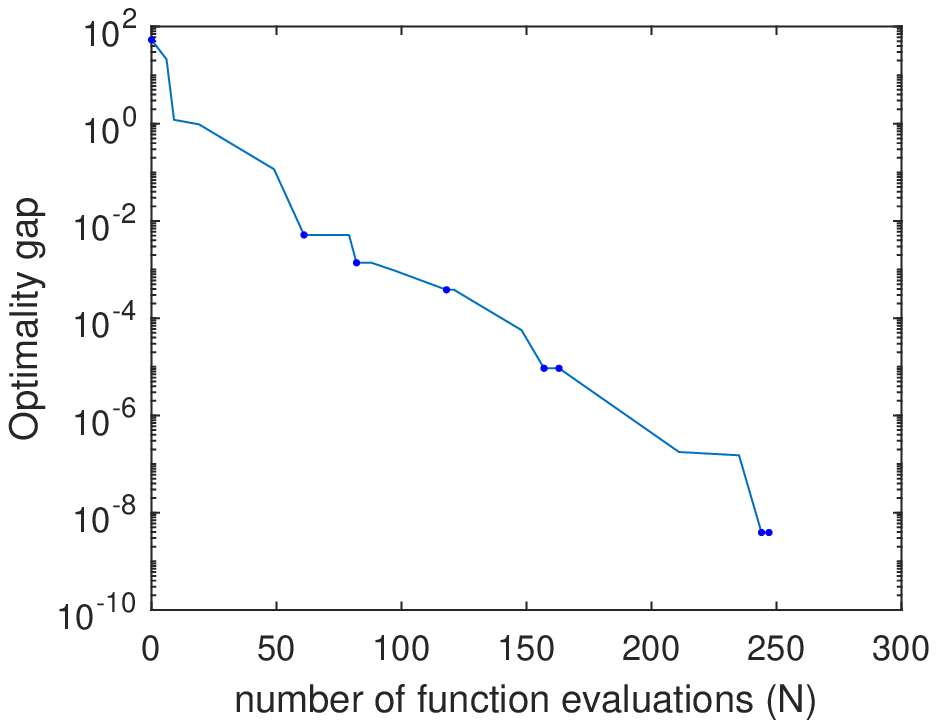}}
\subfigure{\includegraphics[width=0.325\textwidth]{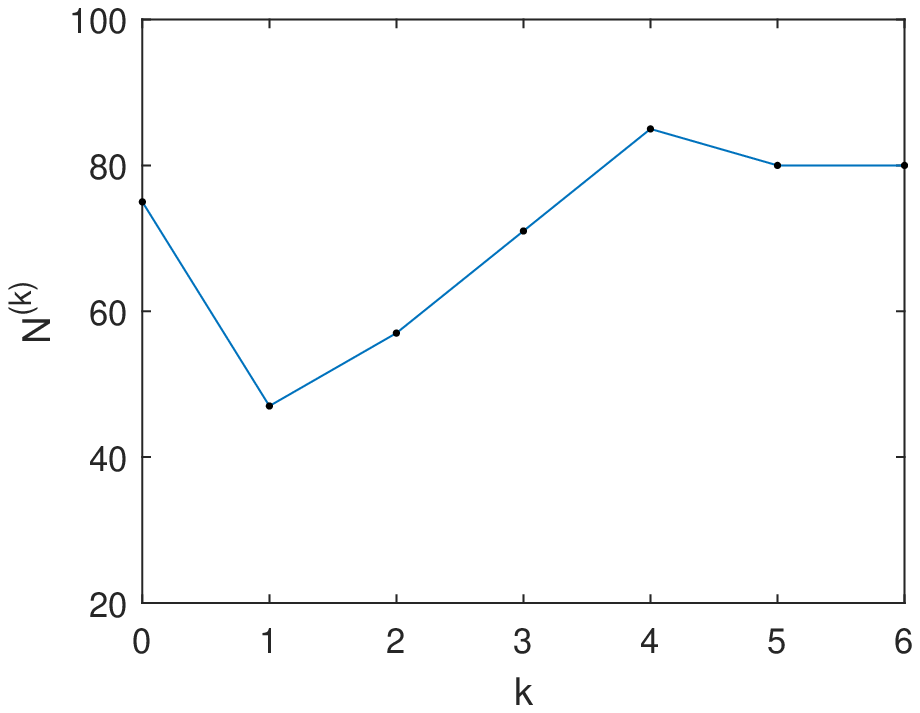}}
\subfigure{\includegraphics[width=0.325\textwidth]{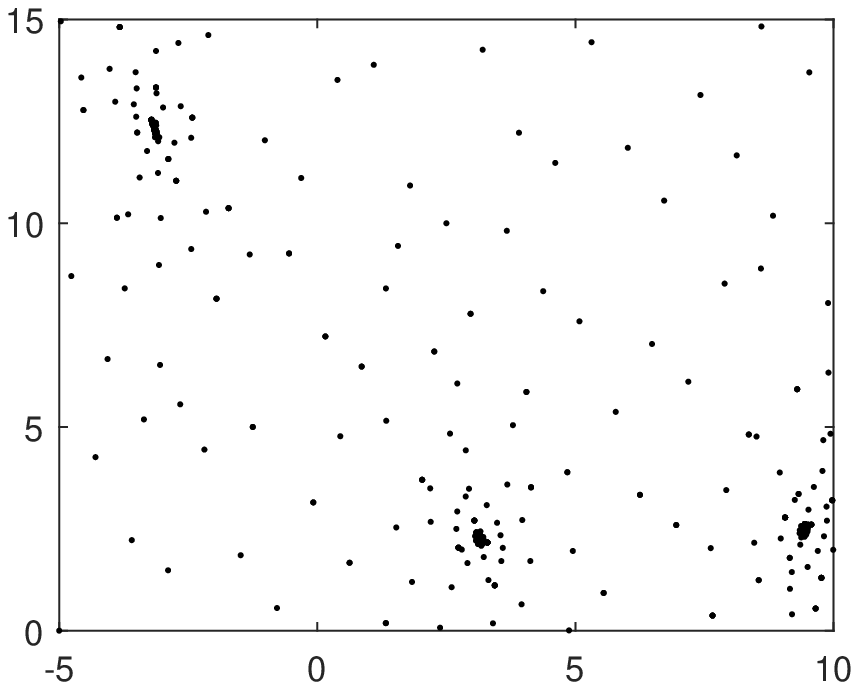}}
\caption{Performance of the median-type method shown in Algorithm \ref{CM:alg:CM} for the function $f(x_1,x_2)=(x_2-\frac{5.1} {4\pi^2}x_1^2+\frac{5}{\pi}x_1-6)^2 +10(1-\frac{1}{8\pi})\cos(x_1)+10, x_1\in[-5,10], x_2\in[0,15]$, where the multiple global minima on the domain $\Omega=[-5,10]\times[0,15]$ are located at $(-\pi,12.275)$, $(\pi,2.275)$ and $(3\pi,2.475)$. And the parameter setting for Algorithm \ref{CM:alg:CM} (see Section \ref{CM:s6} for details) are $K=7,m=2,\textrm{minIterInner}=1,\omega=1$, $c_k=\bar{c}=50$ and $t_k=\bar{t}=3.5$ for $k=0,1,\cdots,K-1$. Upper and middle rows: the first six contractions; for example, the upper left plot indicates that the first contraction from $D^{(0)}=\Omega$ to $D^{(1)}$, the objective function is shown by contour lines with some suitable level marks, the samples used to build the first model on $D^{(0)}$ are visible as dots in black, the random candidate points, i.e., $\mathcal{T}$ generating by the reflected random walk (Algorithm \ref{CM:alg:CPS}), over the contracted subdomain $D^{(1)}$, are visible as circledots in blue; the remaining five plots are similar. Obviously, each model $\mathcal{A}^{(k)}f$ only depends on the corresponding sample set $\chi^{(k)}$ and all points in $\chi^{(k)}$ are inside $D^{(k)}$. Lower left: the convergence plot about the current best $f_{\textrm{best}}^*$, where the optimality gap is defined as $f_{\textrm{best}}^*-f^*$. Lower middle: the sample size used in each model. Lower right: $x$ trace after all seven contractions.}
\label{CM:fig:A1}
\end{figure}

Figure \ref{CM:fig:A1} may cause an illusion that these contraction sets are usually very close to some of the level sets; in fact, a contraction set can be, but does not have to be, a level set, as shown in Figure \ref{CM:fig:A2}. In the early stages of contraction, they can differ greatly from the level set, but in the later stage, due to the stronger error bound conditions, they will almost be equal to certain level sets. The difference in performance of the algorithms in Figures \ref{CM:fig:A1} and \ref{CM:fig:A2} is because the former uses a large confidence parameter and the latter applies a small one. When two algorithms with different confidence parameters converge, the one with small parameter usually has a lower computational cost. Because the confidence parameters mainly affect the probability of convergence of algorithms, as shown in Section \ref{CM:s6}.

\begin{figure}[tbhp]
\centering
\subfigure{\includegraphics[width=0.325\textwidth]{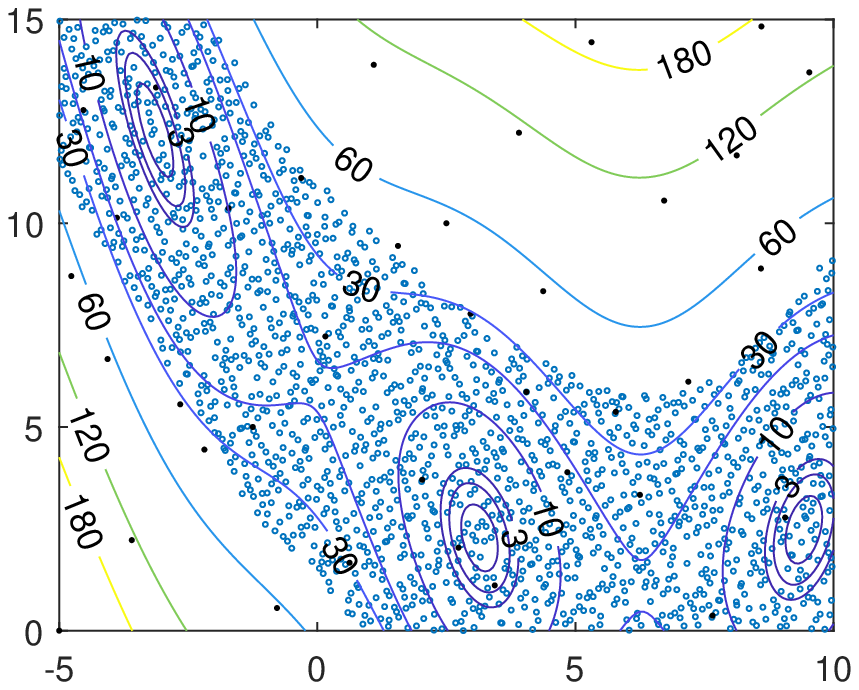}}
\subfigure{\includegraphics[width=0.325\textwidth]{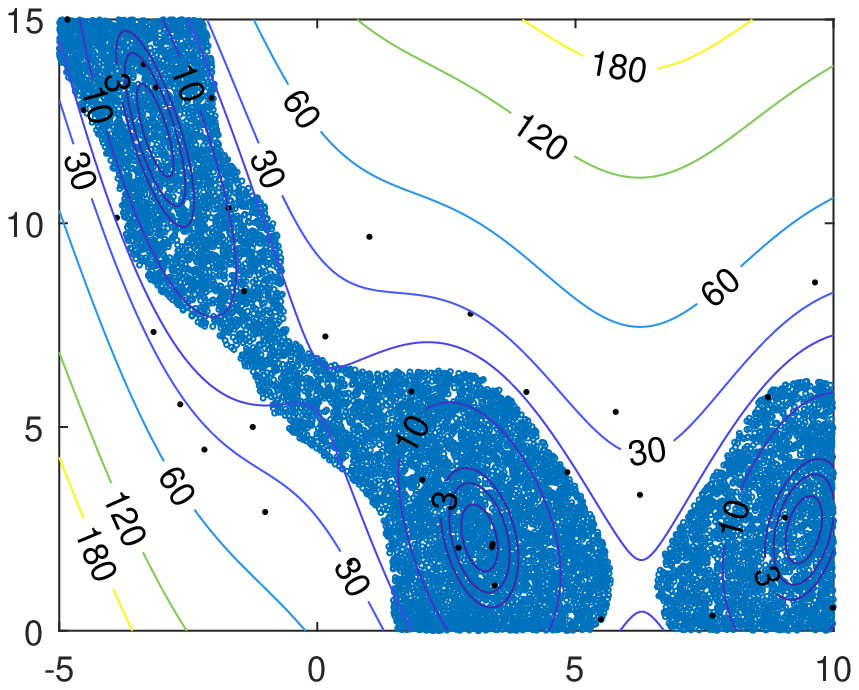}}
\subfigure{\includegraphics[width=0.325\textwidth]{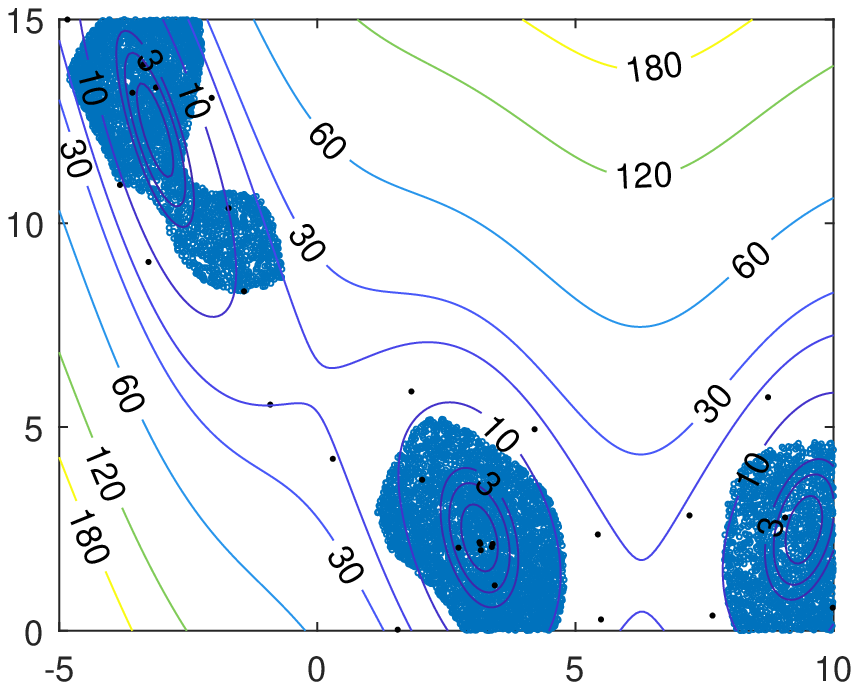}}\\
\subfigure{\includegraphics[width=0.325\textwidth]{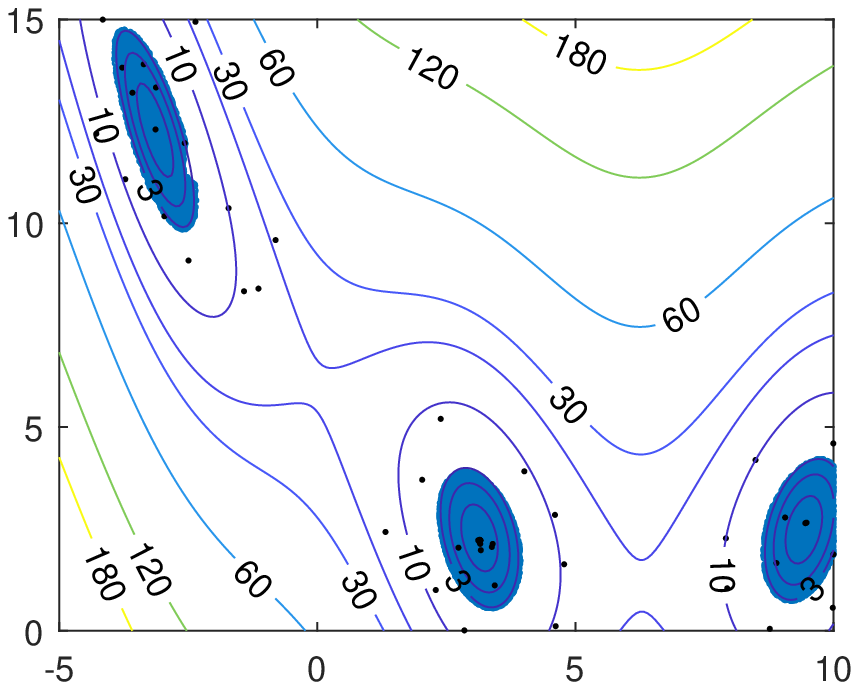}}
\subfigure{\includegraphics[width=0.325\textwidth]{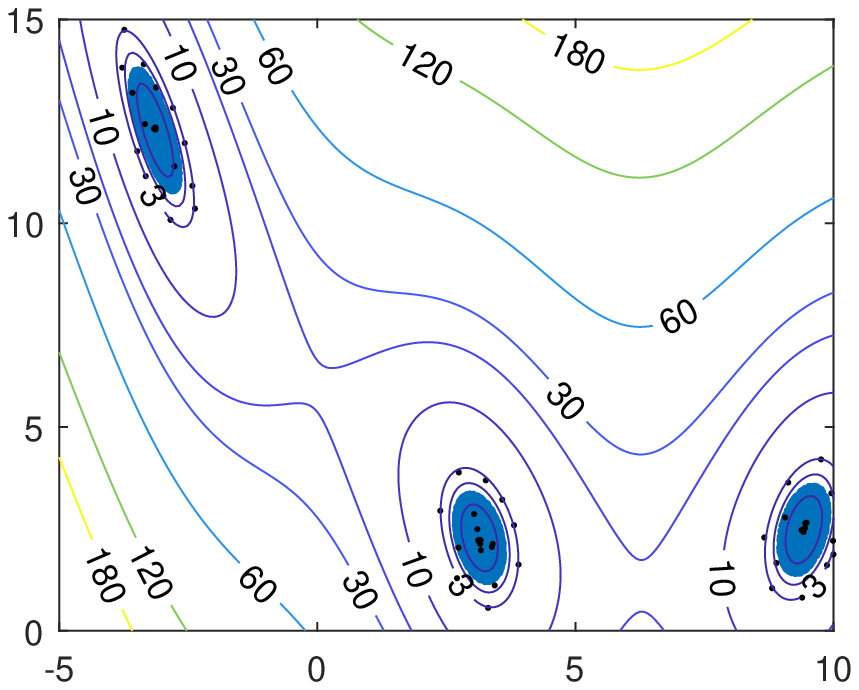}}
\subfigure{\includegraphics[width=0.325\textwidth]{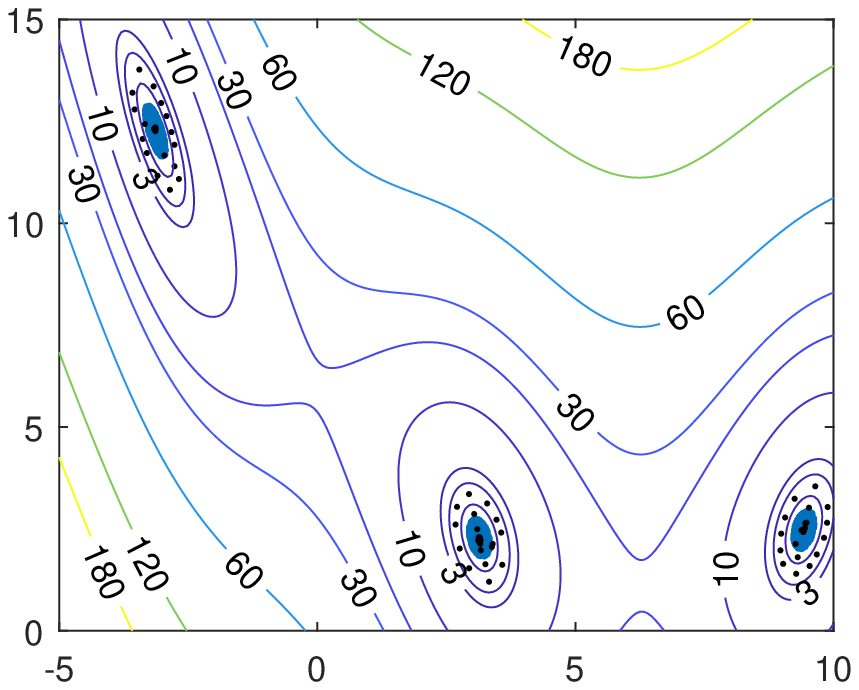}}\\
\subfigure{\includegraphics[width=0.325\textwidth]{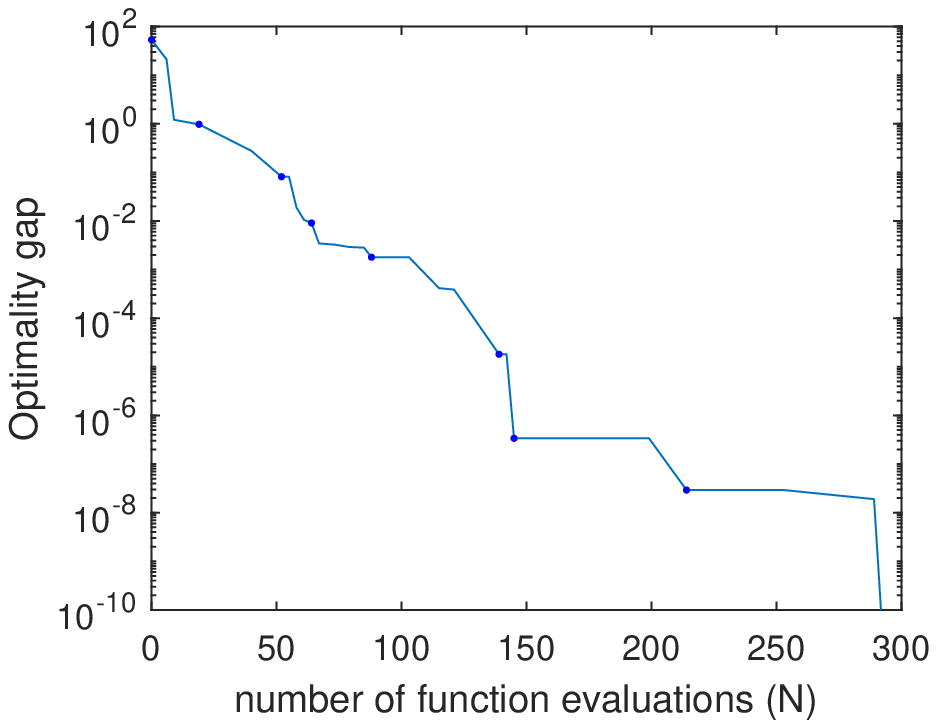}}
\subfigure{\includegraphics[width=0.325\textwidth]{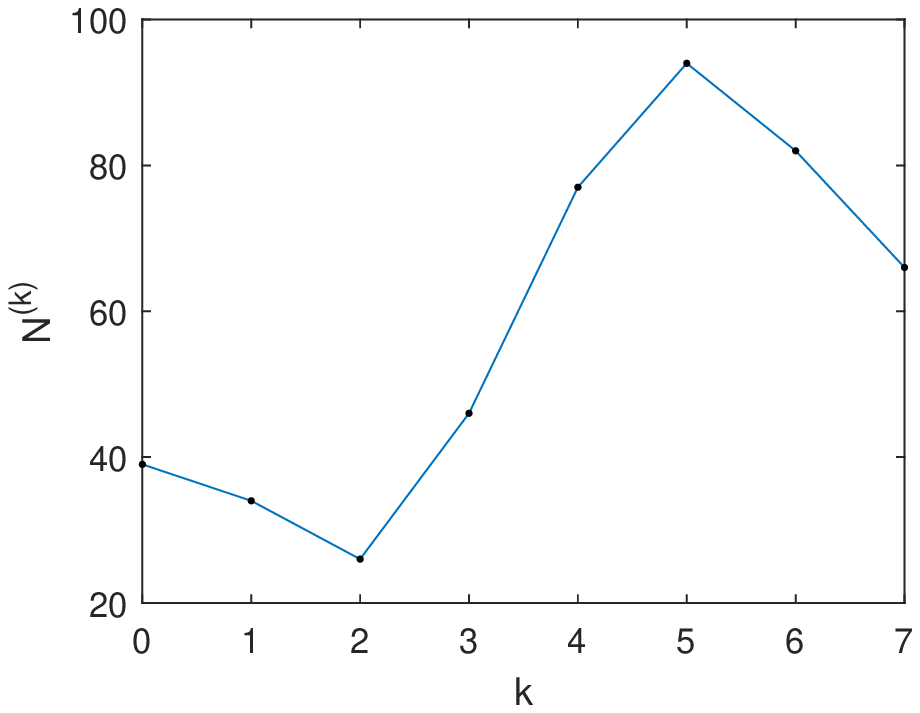}}
\subfigure{\includegraphics[width=0.325\textwidth]{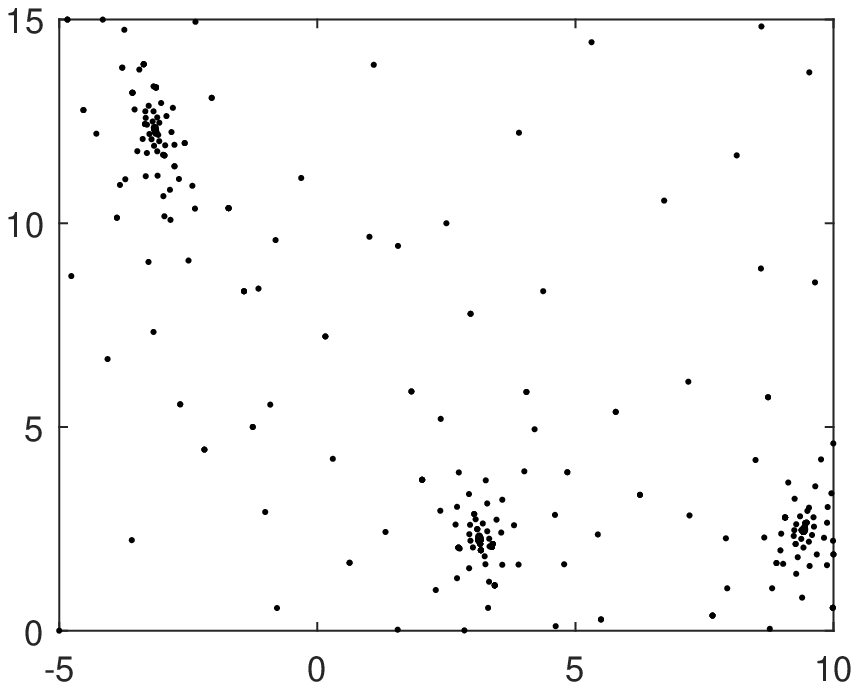}}
\caption{Performance of the median-type method with different confidence parameter for the function given in Figure \ref{CM:fig:A1}. The parameter setting for Algorithm \ref{CM:alg:CM} are $K=8,m=2,\textrm{minIterInner}=1,\omega=1$, $c_k=\bar{c}=50$ and $t_k=\bar{t}=1.75$ for $k=0,1,\cdots,K-1$. Upper and middle rows: the first six contractions and the expression of the plots is consistent with Figure \ref{CM:fig:A1}. It is shown that a contraction set can be, but does not have to be, a level set. Lower left: the convergence plot about the current best $f_{\textrm{best}}^*$, where the optimality gap is defined as $f_{\textrm{best}}^*-f^*$. Compared to the lower-left plot in Figure \ref{CM:fig:A1}, the smaller confidence parameter have almost no significant impact on convergence. Lower middle: the sample size used in each model. Compared to the lower-middle plot in Figure \ref{CM:fig:A1}, when the algorithm uses a small confidence parameter, the number of samples used by the models in the early stage is significantly lower. This helps control the total computational cost, but it is at the expense of reducing the probability of convergence. Lower right: $x$ trace after all seven contractions.}
\label{CM:fig:A2}
\end{figure}

Of course, there is another informal reason to support the use of median-type methods. For fixed $c$ and $\omega$, consider $u^{(k)}=\textrm{prctile} \big(f_{\chi^{(k)}},c\big)$. Assume that \eqref{CM:eq:CondSC1} holds and $\frac{1}{1+q}=\frac{c}{100}$, then $c=50$ implies $q=1$ according to Theorem \ref{CM:thm:SConv}. Hence, every $\omega\in(0,1)$ satisfies the second strong convergence condition $\omega<q$, and the linear convergence factor is
\begin{equation*}
  \frac{1+\omega}{1+q}=\frac{1+\omega}{2}\in\Big(\frac{1}{2},1\Big)
\end{equation*}
Similarly, assume that \eqref{CM:eq:CondSCa1} holds and $\frac{1}{1+q}=\frac{c}{100}$, then $c=25$ implies $q=3$ by Theorem \ref{CM:thm:SConv2}, then every $\omega\in(0,1)$ satisfies the condition $(1+\omega)^2<1+q$, and the linear factor is
\begin{equation*}
  \frac{(1+\omega)^2}{1+q}=\frac{(1+\omega)^2}{4}.
\end{equation*}
Hence, to be conservative, the lower-quartile method is also recommended in practice. More flexibly, for a fixed $c_k$, each $q_k$ can be given as
\begin{equation*}
  q_k=\frac{\textrm{prctile}\big(f_{\chi^{(k)}},c_k\big)
 -\min\big(A_k^*,f_{\chi^{(k)}}^*\big)}
 {f_{\chi^{(k)}}^{**}-f_{\chi^{(k)}}^*}-1,
\end{equation*}
so one can obtain the upper bound of $\omega_k$ and the linear factor for the $k$th contraction.

\section{Assumptions and lemmas related to the contractibility}
\label{CM:s3}

In order to categorise continuous optimization problems from the perspective of whether the contraction can be carried out effectively, we shall introduce three assumptions related to the contractibility and some useful lemmas in this section.

For convenience of theoretical analysis, we temporarily assumed that we could choose
\begin{equation}\label{CM:eq:SDMed}
u^{(k)}=M_{\Omega,f}^{(k)}:=\textrm{prctile}\left(f(\xi),2^{-k}\cdot100\%\right),
\end{equation}
where $\xi$ is a uniformly distributed random variable on $\Omega$. Correspondingly, for $k\in\mathbb{N}_0$, the error bound condition is
\begin{equation*}
\max_{x\in D^{(k)}}\big|\mathcal{A}^{(k)}f(x)-f(x)\big|<\omega
\big(M_{\Omega,f}^{(k)}-f^*\big),~~\forall\omega\in(0,1],
\end{equation*}
and the strong convergence conditions of Theorem \ref{CM:thm:SConv} are
\begin{equation*}
 M_{\Omega,f}^{(k)}-f^*\leqslant\frac{1}{1+q}
 \max_{x\in D^{(k)}}[f(x)-f^*]~~\textrm{and}~~\omega<q.
\end{equation*}
Note that this choice \eqref{CM:eq:SDMed} is not necessary in practice and will finally be released.

\subsection{Hierarchical low-frequency dominant functions}
\label{CM:s3:1}

\subsubsection{Motivation and concepts}

The efficiency of the contractions can be guaranteed if one can quickly sketch out the overall landscape of the valley in each step and the relevant deviation will not be large enough to dig one or more deep holes in the highlands. This requires that the low-frequency components of $f$ always play a dominant role on every subset $D^{(k)}$. This meaningful observation prompted us to impose certain restrictions on the Fourier transform of $f$.

\begin{assum}[Hierarchical low-frequency dominant function]
\label{CM:ass:HLFD}
The function $f$ is a $(\rho,p)$-type hierarchical low-frequency dominant function (HLFDF), where $\rho,p>0$, that is, $f\in L(\mathbb{R}^n)$ and there exist $p_1,p_2>0$ such that for any $j=1,2,\cdots$, it holds that
\begin{equation}\label{CM:eq:HLFDF1}
\int_{\|t\|_2>2^{\frac{j-1}{n}}\rho}|\hat{f}(t)|\mathrm{d}t
<(1+p_1)\!\int_{2^{\frac{j-1}{n}}\rho<\|t\|_2\leqslant2^{\frac{j}{n}}\rho}
\!|\hat{f}(t)|\mathrm{d}t~~\textrm{with}~~p_1\leqslant p,
\end{equation}
and
\begin{equation}\label{CM:eq:HLFDF2}
\int_{\|t\|_2>2^{\frac{j-1}{n}}\rho}|\hat{f}(t)|^2\mathrm{d}t
<(1+p_2^2)\!\int_{2^{\frac{j-1}{n}}\rho<\|t\|_2\leqslant2^{\frac{j}{n}}\rho}
\!|\hat{f}(t)|^2\mathrm{d}t~~\textrm{with}~~
\Big(\frac{p_2^2}{1+p_2^2}\Big)^{\frac{1}{2}}\!\!\leqslant\!\frac{p}{1+p},
\end{equation}
where $\hat{f}$ is the Fourier transform of $f$.
\end{assum}
\begin{rem}
It seems weird to consider $L_2$ here, however, this is to establish a connection with the reproducing-kernel Hilbert function. And it is worth noting that, Assumption \ref{CM:ass:HLFD} does not require $\hat{f}$ to decay very quickly; as a univariate instance, for all $\omega>0$ and $\rho>0$, $t^{-1-\omega}$ satisfies the conditions \eqref{CM:eq:HLFDF1} and \eqref{CM:eq:HLFDF2} with $p_1=1/(2^\omega-1)$ and $p_2=\sqrt{1/(2^{1+2\omega}-1)}$. This shows that an HLFDF is not necessarily differentiable. So this makes the CM applicable to Lipschitz continuous, or even H\"{o}lder continuous objective functions.
\end{rem}

Our discussion below is related to the Paley-Wiener space of bandlimited functions, i.e., $\mathcal{PW}_{B_2(\sigma)}$, which is defined by
\begin{equation*}
  \mathcal{PW}_{B_2(\sigma)}:=\left\{f\in L_2(\mathbb{R}^n):\textrm{supp}(\hat{f})
  \subseteq B_2(\sigma)\right\},~~\textrm{for any}~~\sigma>0,
\end{equation*}
where $\hat{f}$ is the Fourier transform of $f$ and $B_2(\sigma):= \{t\in\mathbb{R}^n:\|t\|_2\leqslant\sigma\}$ is the $2$-ball in $\mathbb{R}^n$ having center $0$ and radius $\sigma$. Since $B_2(\sigma)\subset B_\infty(\sigma):= \{t\in\mathbb{R}^n:\|t\|_\infty\leqslant\sigma\}$, it holds that
\begin{equation*}
  \mathcal{PW}_{B_2(\sigma)}\subset\mathcal{PW}_{B_\infty(\sigma)}:=\left\{f\in L_2(\mathbb{R}^n):\textrm{supp}(\hat{f})\subseteq B_\infty(\sigma)\right\}.
\end{equation*}
Furthermore, we use $C_0\cap L_2:=C_0(\mathbb{R}^n)\cap L_2(\mathbb{R}^n)$ as a natural class of functions that includes all $\mathcal{PW}_{B_2(\sigma)}$, which is a Banach space by employing the norm
\begin{equation*}
  \|f\|_{C_0\cap L_2}=\max\big(\|f\|_\infty,\|f\|_{L_2}\big).
\end{equation*}
Notice that if $\hat{f}\in L_1(\mathbb{R}^n)\cap L_2(\mathbb{R}^n)$, then $\|f\|_\infty\leqslant\|\hat{f}\|_{L_1}$ and $\|f\|_{L_2}=\|\hat{f}\|_{L_2}$, and then
\begin{equation*}
  \|f\|_{C_0\cap L_2}\leqslant\|\hat{f}\|_{L_1\cap L_2}:=
  \max\big(\|\hat{f}\|_{L_1(\mathbb{R}^n)},\|\hat{f}\|_{L_2(\mathbb{R}^n)}\big).
\end{equation*}

In the following, we shall see that every HLFDF belongs to $C_0(\mathbb{R}^n)\cap L_2(\mathbb{R}^n)$, and further, its bandlimited components have a good approximation property which is related to the sampling density (Lemma \ref{CM:lem:HLFDF}). More importantly, a bandlimited approximation can be sufficiently constructed by kernel-based interpolation (Lemma \ref{CM:lem:HLFDF&KI}) so that the error bound condition of CMs is easy to be satisfied, as we expected at the beginning of this section.

\subsubsection{Bandlimited component of HLFDFs}

For $j\in\mathbb{N}_0$ and $\rho>0$, we define the $(j,\rho)$-bandlimited component of $f\in L(\mathbb{R}^n)$ by
\begin{equation}\label{CM:eq:BC}
f_\rho^{(j)}(x)=\int_{\|t\|_2\leqslant2^{\frac{j}{n}}\rho}
\hat{f}(t)e^{2\pi\mathrm{i}x^{\mathrm{T}}t}\mathrm{d}t,
\end{equation}
where $x^{\mathrm{T}}t=\sum_{i=1}^nx_it_i$ is the inner product of two vectors $x$ and $t$; then we have
\begin{equation*}
  f_\rho^{(j)}\in\mathcal{PW}_{B_2(2^{j/n}\rho)}\subset
  \mathcal{PW}_{B_\infty(2^{j/n}\rho)}.
\end{equation*}
According to the Nyquist-Shannon sampling theorem, $f_\rho^{(j)}$ can be completely reconstructed by its samples corresponding to a sampling density of $2^j\rho^n/\pi^n$. And further, the following lemma indicates the characteristic property of HLFDFs: there exist a class of bandlimited approximations such that the corresponding approximation error bounds are reduced by a factor of $p/(1+p)$ every time the number of function evaluations doubles. Actually, this feature implies the effectiveness of contractions.

\begin{lem}\label{CM:lem:HLFDF}
For any $j\in\mathbb{N}_0$, if $f$ is a $(\rho,p)$-type HLFDF and $f_\rho^{(j)}$ is a bandlimited component defined as \eqref{CM:eq:BC}, then $f\in C_0(\mathbb{R}^n)\cap L_2(\mathbb{R}^n)$, $\hat{f}\in L(\mathbb{R}^n)\cap L_2(\mathbb{R}^n)$ and
\begin{equation*}
 \big\|f-f_\rho^{(j)}\big\|_{C_0\cap L_2}
 <\left(\frac{p}{1+p}\right)^j\big\|\hat{f}\big\|_{L_1\cap L_2}.
\end{equation*}
\end{lem}
\begin{proof}
It is clear that
\begin{equation*}
  \max\left(\frac{p_1}{1+p_1},\sqrt{\frac{p_2^2}{1+p_2^2}}\right)
  \leqslant\frac{p}{1+p},
\end{equation*}
and we need to prove that
\begin{equation*}
 \big\|f-f_\rho^{(j)}\big\|_\infty
 <\left(\frac{p_1}{1+p_1}\right)^j\big\|\hat{f}\big\|_{L_1(\mathbb{R}^n)}
 ~~\textrm{and}~~\big\|f-f_\rho^{(j)}\big\|_{L_2}
 <\left(\frac{p_2^2}{1+p_2^2}\right)^{\frac{j}{2}}
 \big\|\hat{f}\big\|_{L_2(\mathbb{R}^n)}.
\end{equation*}

First, let $R_\rho^{(j)}=\sum_{i=j}^\infty I_\rho^{(i)}$, where
\begin{equation*}
 I_\rho^{(0)}=\int_{\|t\|_2\leqslant\rho} |\hat{f}(t)|\mathrm{d}t~~\textrm{and}~~
 I_\rho^{(j)}=\int_{2^{\frac{j-1}{n}}\rho<\|t\|_2
 \leqslant2^{\frac{j}{n}}\rho}|\hat{f}(t)|\mathrm{d}t
 ~~\textrm{for all}~j\in\mathbb{N},
\end{equation*}
then the condition \eqref{CM:eq:HLFDF1} can be rewritten as
\begin{equation*}
  R_\rho^{(j)}<(1+p)I_\rho^{(j)},~~\textrm{or equivalently},
  ~~R_\rho^{(j+1)}<pI_\rho^{(j)}.
\end{equation*}
Since $f\in L(\mathbb{R}^n)$, we have $|\hat{f}(t)|\leqslant\|f\|_{ L_1(\mathbb{R}^n)}<\infty$, and then
\begin{equation*}
  I_\rho^{(0)}=\int_{\|t\|_2\leqslant\rho} |\hat{f}(t)|\mathrm{d}t<\infty.
\end{equation*}
Therefore, it follows that $\|\hat{f}\|_{ L_1(\mathbb{R}^n)}=R_\rho^{(0)} \leqslant(1+p)I_\rho^{(0)}<\infty$, that is, $\hat{f}\in L(\mathbb{R}^n)$. Moreover, we further have the decay ratio
\begin{equation*}
 \frac{R_\rho^{(j+1)}}{R_\rho^{(j)}}
 =\frac{R_\rho^{(j+1)}}{I_\rho^{(j)}+R_\rho^{(j+1)}}
 \leqslant\frac{R_\rho^{(j+1)}}{R_\rho^{(j+1)}/p+R_\rho^{(j+1)}}=\frac{p}{1+p}.
\end{equation*}
Hence, by noting that
\begin{equation*}
 \frac{R_\rho^{(j+1)}}{\|\hat{f}\|_{ L_1(\mathbb{R}^n)}}
 =\frac{R_\rho^{(j+1)}}{I_\rho^{(0)}+R_\rho^{(1)}}
 <\frac{R_\rho^{(j+1)}}{R_\rho^{(1)}}
 =\frac{R_\rho^{(2)}}{R_\rho^{(1)}}\frac{R_\rho^{(3)}}{R_\rho^{(2)}}
 \cdots\frac{R_\rho^{(j+1)}}{R_\rho^{(j)}}
 \leqslant\left(\frac{p}{1+p}\right)^j,
\end{equation*}
the error bound can also be rewritten as
\begin{equation*}
 \big\|f-f_\rho^{(j)}\big\|_\infty\leqslant
 \big\|\hat{f}-\hat{f}_\rho^{(j)}\big\|_{ L_1(\mathbb{R}^n)}
 =R_\rho^{(j+1)}<\left(\frac{p}{1+p}\right)^j
 \big\|\hat{f}\big\|_{ L_1(\mathbb{R}^n)}.
\end{equation*}

Similarly, we can prove that $\hat{f}\in L_2(\mathbb{R}^n)$ and
\begin{equation*}
 \big\|f-f_\rho^{(j)}\big\|_{L_2}^2
 =\big\|\hat{f}-\hat{f}_\rho^{(j)}\big\|_{L_2(\mathbb{R}^n)}^2
 <\left(\frac{p_2^2}{1+p_2^2}\right)^j
 \big\|\hat{f}\big\|_{L_2(\mathbb{R}^n)}^2,
\end{equation*}
so the proof is complete.
\end{proof}

\subsubsection{Kernel-based interpolation of HLFDFs}

Now we focus on how to obtain a satisfactory approximation by kernel-based interpolation. Here is an outline of our idea: first, there exists a sufficiently accurate bandlimited function that interpolates any HLFDF on a certain set of points; second, this bandlimited function can be fully constructed by kernel-based interpolation.

Lemma \ref{CM:lem:HLFDF} shows that every HLFDF belongs to $C_0(\mathbb{R}^n)\cap L_2(\mathbb{R}^n)$, hence, as an extension of Theorem 3.5 in \cite{NarcowichF2004A_BandLimited&RBF}, for any HLFDF $f$, we can find a sufficiently accurate bandlimited function which interpolates $f$ on a fixed set of points.

\begin{lem}\label{CM:lem:HLFDFinterpolating}
Suppose $\Omega$ is defined in problem \eqref{CM:eq:COP}, $\chi=\{\chi_1,\chi_2,\cdots,\chi_N\}$ is quasi-uniformly distributed over $D\subset\Omega$ with the separation distance $q_\chi=\tau'\cdot C^{-\frac{1}{n}}2^{-\frac{j}{n}}\pi/\rho$ for a certain $\tau'\leqslant1$, i.e., corresponding to a sampling density of $C2^j\rho^n/\pi^n$ for a certain $C\geqslant1$, and $\sigma$ is chosen so that
\begin{equation}\label{CM:eq:kpara}
  \sigma\geqslant\sigma_0:=\frac{48\tau'}{q_\chi}\left[\frac{\pi}{18}
  \Gamma^2\left(\frac{n+2}{2}\right)\right]^{\frac{1}{n+1}}
  =48\frac{C^\frac{1}{n}2^\frac{j}{n}\rho}{\pi}\left[\frac{\pi}{18}
  \Gamma^2\left(\frac{n+2}{2}\right)\right]^{\frac{1}{n+1}},
\end{equation}
where $\Gamma$ is the Gamma function. If $f$ is a $(\rho,p)$-type HLFDF, there exists $g_\sigma\in \mathcal{PW}_{B_2(\sigma)}$ such that
\begin{equation*}
  f|_\chi=g_\sigma|_\chi~~\textrm{and}~~
  \big\|f-g_\sigma\big\|_{C_0\cap L_2}
  <9\left(\frac{p}{1+p}\right)^{j+1}\big\|\hat{f}\big\|_{L_1\cap L_2}.
\end{equation*}
\end{lem}
\begin{proof}
According to Lemma \ref{CM:lem:HLFDF}, we have $f\in C_0(\mathbb{R}^n)\cap L_2(\mathbb{R}^n)$, then from Proposition 3.4 and Theorem 3.5 in \cite{NarcowichF2004A_BandLimited&RBF}, for any $\sigma\geqslant\sigma_0$, we get the existence of $g_\sigma\in\mathcal{PW}_{B_2(\sigma)}$ for which $f|_\chi=g_\sigma|_\chi$ and
\begin{equation*}
  \big\|f-g_\sigma\big\|_{C_0\cap L_2}
  \leqslant9~\textrm{dist}_{C_0\cap L_2}\big(f,\mathcal{PW}_{B_2(\sigma)}\big).
\end{equation*}
Notice that for all $n\in\mathbb{N}$,
\begin{equation*}
  \left[\frac{\pi}{18}\Gamma^2\left(\frac{n+2}{2}\right)
  \right]^{\frac{1}{n+1}}>1,
\end{equation*}
it follows that
\begin{equation*}
  \sigma\geqslant\sigma_0>48C^\frac{1}{n}2^\frac{j}{n}\rho/\pi
  >2^\frac{j+1}{n}\rho,~~\textrm{for every}~~n\in\mathbb{N},
\end{equation*}
since the bandlimited component $f_\rho^{(j+1)}$ defined by \eqref{CM:eq:BC} belongs to $\mathcal{PW}_{B_2(2^{(j+1)/n}\rho)}$, so we further obtain
\begin{equation*}
  \textrm{dist}_{C_0\cap L_2}\big(f,\mathcal{PW}_{B_2(\sigma)}\big)
  \leqslant\textrm{dist}_{C_0\cap L_2}\big(f,\mathcal{PW}_{B_2(2^{(j+1)/n}\rho)}\big)
  \leqslant\|f-f_\rho^{(j+1)}\|_{C_0\cap L_2},
\end{equation*}
together with Lemma \ref{CM:lem:HLFDF}, the desired result follows.
\end{proof}

A bandlimited function is also a reproducing-kernel Hilbert function with respect to the Gaussian kernel, i.e., $\phi_\sigma(x)=e^{-\sigma^2\|x\|_2^2}$ with its Fourier transform $\hat{\phi}_\sigma(t)= \frac{\sqrt{\pi}}{\sigma}e^{-\|t\|_2^2/(4\sigma^2)}$. Therefore, a bandlimited function $g_{\sigma'}$ with its Fourier transform supported in $\{t\in\mathbb{R}^n:\|t\|_2\leqslant\sigma'\}$ can be effectively reconstructed by the Gaussian kernel based interpolation with parameter $\sigma\geqslant\frac{\sigma'}{2}$.
\begin{lem}\label{CM:lem:BLS&RKHS}
Suppose that $\mathcal{N}_\sigma$ is a reproducing-kernel Hilbert function space with the Gaussian kernel $\phi_\sigma(x)= e^{-\sigma^2\|x\|_2^2}$, i.e.,
\begin{equation*}
  \mathcal{N}_\sigma=:\left\{f\in L_2(\mathbb{R}^n):
  \|f\|_{\mathcal{N}_{G_\sigma}}^2=\int_{\mathbb{R}^n}
  \frac{|\hat{s}(t)|^2}{\hat{\phi}_\sigma(t)}\ud t
  =\frac{\sigma}{\sqrt{\pi}}\int_{\mathbb{R}^n}|\hat{s}(t)|^2
  e^{\frac{\|t\|_2^2}{4\sigma^2}}\ud t<\infty\right\}.
\end{equation*}
Then, for any $\sigma'\leqslant2\sigma$ and $g\in\mathcal{PW}_{B_2(\sigma')}$,
\begin{equation*}
  \|g\|_{\mathcal{N}_{\sigma}}<\sqrt{2\sigma}\|\hat{g}\|_{ L_2(\mathbb{R}^n)}.
\end{equation*}
\end{lem}
\begin{proof}
Since $g\in\mathcal{PW}_{B_2(\sigma')}$ implies that its Fourier transform $\hat{g}$ is supported in
\begin{equation*}
  B_2(\sigma')=\{t\in\mathbb{R}^n:\|t\|_2\leqslant\sigma'\},
\end{equation*}
it follows that
\begin{equation*}
  \|g\|_{\mathcal{N}_{\sigma}}^2=\frac{\sigma}{\sqrt{\pi}}\int_{\mathbb{R}^n}
  |\hat{g}(t)|^2e^{\frac{\|t\|_2^2}{4\sigma^2}}\ud t
  =\frac{\sigma}{\sqrt{\pi}}\int_{\|t\|_2\leqslant\sigma'}
  |\hat{g}(t)|^2e^{\frac{\|t\|_2^2}{4\sigma^2}}\ud t
  <2\sigma\|\hat{g}\|_{L_2(\mathbb{R}^n)}^2,
\end{equation*}
so the desired result follows.
\end{proof}

Now we shall show that, for an appropriate sample set, the kernel-based interpolation of an HLFDF has a similar approximation property to its bandlimited components.
\begin{lem}\label{CM:lem:HLFDF&KI}
Suppose $\Omega$ is defined in problem \eqref{CM:eq:COP}, $\chi=\{\chi_1,\chi_2,\cdots,\chi_N\}$ is quasi-uniformly distributed over $D\subset\Omega$ with the separation distance $q_\chi=\tau'\cdot C^{-\frac{1}{n}}2^{-\frac{j}{n}}\pi/\rho$ and the fill distance $h_{D,\chi}=\tau''\cdot C^{-\frac{1}{n}}2^{-\frac{j}{n}}\pi/\rho$ for $\tau'\leqslant1\leqslant\tau''$, i.e., corresponding to a sampling density of $C2^j\rho^n/\pi^n$. Then, if $f$ is a $(\rho,p)$-type HLFDF, there exists $C\geqslant1$ and Gaussian kernel interpolant $\mathcal{I}_\chi f$ such that for all $j\in\mathbb{N}_0$,
\begin{equation*}
  \big\|\mathcal{I}_\chi f-f\big\|_{L_\infty(D)}
  <9\left(\frac{p}{1+p}\right)^j\big\|\hat{f}\big\|_{L_1\cap L_2}.
\end{equation*}
\end{lem}
\begin{rem}
This lemma provides a constructive interpolation equivalent to the bandlimited component $f_\rho^{(j)}$. On the entire $\mathbb{R}^n$, from the Nyquist-Shannon sampling theorem, $f_\rho^{(j)}$ can be fully reconstructed by its samples corresponding to a sampling density of $2^{j}\rho^n/\pi^n$; and on the bounded domain $D\subset\mathbb{R}^n$, the restriction of $f_\rho^{(j)}$ to $D$, denoted by $f_\rho^{(j)}|D$, can also be reconstructed by samples over $D$ with size $C2^{j}\mu(D)\rho^n/\pi^n$ for a certain $C>1$. Specifically, if $D$ is a cube, $f_\rho^{(j)}|D$ is closely related to a threshold value $2^{j}\mu(D)\rho^n/\pi^n$ and the prolate spheroidal functions $\psi_i(x)=\psi_i(x; 2^{\frac{j}{n}}\rho,D)$ which are the relevant eigenfunctions of the time and frequency limiting operator $Q=Q(2^{\frac{j}{n}}\rho,D)$ \citep{LandauH1961A_PSWFii, LandauH1962A_PSWFiii,SlepianD1964A_PSWFiv,SlepianD1976A_Onbandwidth, SlepianD1961A_PSWFi}. More clearly, if we denote a prolate series up to and including the $N$th term by
\begin{equation*}
\mathcal{S}_Nf(x)=\sum_{i=1}^{N}\psi_i(x)\int_Df(x)\psi_i(x)\mathrm{d}x,
\end{equation*}
then we have a super-exponential decay rate of the error bound
\begin{equation*}
\max_{x\in D}\left|f_\rho^{(j)}(x)-\mathcal{S}_Nf_\rho^{(j)}(x)\right|
\end{equation*}
as soon as $N$ goes beyond the plunge region around the threshold value $2^{j}\mu(D)\rho^n/\pi^n$ \citep{BoydJ2003T_BandlimitedFunction&PSF, BonamiaA2017T_SpectralDecayofBandlimited}. This also supports that $f_\rho^{(j)}$ can be fully constructed by samples over $D$ with size $C2^{j}\mu(D)\rho^n/\pi^n$; however, of course, our kernel-based approach is constructive and much simpler.
\end{rem}
\begin{proof}
By Lemma \ref{CM:lem:HLFDFinterpolating}, there exist $\sigma>2^\frac{j+1}{n}\rho$ and $g_\sigma\in\mathcal{PW}_{B_2(\sigma)}$ such that $f|_\chi=g_\sigma|_\chi$ and
\begin{equation}\label{CM:eq:HLFDF&KIa}
  \big\|f-g_\sigma\big\|_{C_0\cap L_2}
  <9\left(\frac{p}{1+p}\right)^{j+1}\big\|\hat{f}\big\|_{L_1\cap L_2},
\end{equation}
then for all $j\in\mathbb{N}_0$,
\begin{align*}
  \|\hat{g}_\sigma\|_{L_2}=\|g_\sigma\|_{L_2}=&
  \|g_\sigma-f+f\|_{L_2} \\
  \leqslant&\|g_\sigma-f\|_{L_2} +\|f\|_{L_2}
  <9\|\hat{f}\|_{L_1\cap L_2}+\|\hat{f}\|_{L_2}
  <10\|\hat{f}\|_{L_1\cap L_2},
\end{align*}
together with Lemma \ref{CM:lem:BLS&RKHS}, it follows that
\begin{equation*}
  \|g_\sigma\|_{\mathcal{N}_{\sigma}}<10\sqrt{2\sigma}\|\hat{f}\|_{L_1\cap L_2}.
\end{equation*}
Since the domain $D\subset\Omega$ is bounded, there are $C'>1$ and a regular domain $D'\supset D$ with $h'=h_{D',\chi}=C'h_{D,\chi}=C'\tau''\cdot C^{-\frac{1}{n}}2^{-\frac{j}{n}}\pi/\rho$, then according to Theorems 3.5 and 7.5 in \cite{RiegerC2010_RBFSamplingInequalities}, it follows that there exists a $c>0$ such that
\begin{align*}
  \big\|\mathcal{I}_\chi g_\sigma-g_\sigma\big\|_{L_\infty(D)}
  \leqslant&\big\|\mathcal{I}_\chi g_\sigma-g_\sigma\big\|_{L_\infty(D')} \\
  \leqslant&e^{c\log(h')/\sqrt{h'}}\|g_\sigma\|_{\mathcal{N}_{\sigma}} \\
  \leqslant&10\sqrt{2\sigma}e^{c\log(h')/\sqrt{h'}}\|\hat{f}\|_{L_1\cap L_2},
\end{align*}
where $\mathcal{I}_\chi g_\sigma$ is the Gaussian kernel interpolant associated with $\chi$ and kernel parameter $\sigma$. Since $e^{c\log(t)/\sqrt{t}}$ decays faster than any polynomial as $t\to0$, for any $s\in\mathbb{N}$, there is $t_s>0$ such that for all $0\leqslant t<t_s\leqslant1$,
\begin{equation*}
  e^{c\log(t)/\sqrt{t}}<t^{sn},
\end{equation*}
so there exists $C>C_s>1$ such that $h'<t_s\leqslant 1$ and
\begin{equation*}
  10\sqrt{2\sigma}e^{c\log(h')/\sqrt{h'}}<10\sqrt{2\sigma}(h')^{sn}.
\end{equation*}
Note that $\max_{x,y\in\Omega}\|x-y\|\leqslant1$, then $h'<1$ for every $j\in\mathbb{N}_0$, it holds that $\frac{C'\tau''\pi}{C^{1/n}\rho}<1$; let $s$ is an integers satisfying
\begin{equation*}
  (h')^{sn}=\left(\frac{C'\tau''\pi}{C^{1/n}\rho}\right)^{ns}
  \left(\frac{1}{2^{s}}\right)^j
  <\frac{1}{10\sqrt{2\sigma}}\frac{9}{1+p}\left(\frac{p}{1+p}\right)^j,
  ~\forall j\in\mathbb{N}_0,
\end{equation*}
or more specifically,
\begin{equation*}
  \left(\frac{C'\tau''\pi}{C^{1/n}\rho}\right)^{ns}
  <\frac{1}{10\sqrt{2\sigma}}\frac{9}{1+p}~~\textrm{and}~~
  \frac{1}{2^{s}}<\frac{p}{1+p},
\end{equation*}
then
\begin{equation*}
  10\sqrt{2\sigma}e^{c\log(h')/\sqrt{h'}}<\frac{9}{1+p}\left(\frac{p}{1+p}\right)^j,
\end{equation*}
thus, we have
\begin{equation}\label{CM:eq:HLFDF&KIb}
  \big\|\mathcal{I}_\chi g_\sigma-g_\sigma\big\|_{L_\infty(D)}
  <\frac{9}{1+p}\left(\frac{p}{1+p}\right)^j\|\hat{f}\|_{L_1\cap L_2}.
\end{equation}
Moreover, $f|_\chi=g_\sigma|_\chi$, so the uniqueness of the Gaussian kernel interpolant implies that
\begin{equation}\label{CM:eq:HLFDF&KIc}
  \mathcal{I}_\chi g_\sigma=\mathcal{I}_\chi f~~\Rightarrow~~
  \big\|\mathcal{I}_\chi f-\mathcal{I}_\chi g_\sigma\big\|_{L_\infty(D)}=0.
\end{equation}

Therefore, it follows from \eqref{CM:eq:HLFDF&KIa}-\eqref{CM:eq:HLFDF&KIc} that
\begin{align*}
  \big\|\mathcal{I}_\chi f-f\big\|_{L_\infty(D)}
  \leqslant&\big\|\mathcal{I}_\chi f-\mathcal{I}_\chi g_\sigma\big\|_{L_\infty(D)}
  +\big\|\mathcal{I}_\chi g_\sigma-g_\sigma\big\|_{L_\infty(D)}
  +\big\|g_\sigma-f\big\|_{L_\infty(D)} \\
  <&\frac{9}{1+p}\left(\frac{p}{1+p}\right)^j\|\hat{f}\|_{L_1\cap L_2}
  +9\left(\frac{p}{1+p}\right)^{j+1}\big\|\hat{f}\big\|_{L_1\cap L_2} \\
  \leqslant&9\left(\frac{p}{1+p}\right)^j\big\|\hat{f}\big\|_{L_1\cap L_2},
\end{align*}
and the proof is complete.
\end{proof}

\subsection{Tempered functions}
\label{CM:s3:2}

Now we introduce the tempered and weak tempered conditions. And we will further show that, under Assumption \ref{CM:ass:HLFD}, any one of these two conditions can guarantee that there exists a Gaussian kernel interpolant $\mathcal{I}_\chi f$ w.r.t. a suitable sample set $\chi^{(k)}$ on $D^{(k)}$ such that both the error bound condition and the strong convergence condition are satisfied (Lemmas \ref{CM:lem:TF} and \ref{CM:lem:TFa}). Moreover, it is worth noting that the weak condition, as well as Lemma \ref{CM:lem:TFa}, always holds with high probability for all continuous functions (Proposition \ref{CM:prop:TF}).

The tempered condition can be stated as follows.
\begin{assum}[Tempered function]
\label{CM:ass:T}
The function $f$ is a $(p,q)$-type tempered function, that is, for all $k\geqslant0$ and a certain $p>0$, there exists $q>0$ such that
\begin{equation}\label{CM:eq:TF}
\frac{p}{1+p}\left(M_{\Omega,f}^{(k)}-f^*\right)<M_{\Omega,f}^{(k+1)}-f^*
<\frac{1}{1+q}\left(M_{\Omega,f}^{(k)}-f^*\right),
\end{equation}
where $M_{\Omega,f}^{(k)}=\textrm{prctile}\left(f(\xi),2^{-k}\cdot100\%\right)$ and $\xi$ is a uniformly distributed random variable on $\Omega$.
\end{assum}

Any function that satisfies this condition requires its percentiles to decrease steadily, neither too fast nor too slow. Further, we have the following result.
\begin{lem}\label{CM:lem:TF}
Under Assumptions \ref{CM:ass:HLFD} and \ref{CM:ass:T}, suppose $\{D^{(k)}\}$ is defined by Definition \ref{CM:defn:D} with \eqref{CM:eq:SDMed}, in which, $\chi^{(k)}$ is quasi-uniformly distributed over $D^{(k)}\subset\Omega$ with
\begin{equation*}
  q_{\chi^{(k)}}=\tau'\cdot C^{-\frac{1}{n}} 2^{-\frac{k+s}{n}}\pi/\rho
  ~~\textrm{and}~~
  h_{D,\chi^{(k)}}=\tau''\cdot C^{-\frac{1}{n}}2^{-\frac{k+s}{n}}\pi/\rho
  ~~\textrm{for}~~\tau'\leqslant1\leqslant\tau'',
\end{equation*}
where the constant $C$ is as in Lemma \ref{CM:lem:HLFDF&KI}. Then, for all $k\in\mathbb{N}_0$ and $\omega<q$, there exist a unique integer $s\geqslant1$ and a kernel parameter $\sigma>2^{\frac{k+s+1}{n}}\rho$ such that Gaussian kernel interpolant $\mathcal{I}_{\chi^{(k)}} f$ satisfies the error bound condition
\begin{equation*}
 \|\mathcal{I}_{\chi^{(k)}} f-f\|_{L_\infty(D^{(k)})}
 <\omega\left(M_{\Omega,f}^{(k)}-f^*\right)
\end{equation*}
with the strong convergence condition
\begin{equation*}
  M_{\Omega,f}^{(k)}-f^*\leqslant\frac{1}{1+q}\max_{x\in D^{(k)}}[f(x)-f^*].
\end{equation*}
\end{lem}
\begin{proof}
Notice that $\max_{x\in D^{(k)}}f(x)=M_{\Omega,f}^{(k-1)}$, the second inequality obviously holds from the Assumption \ref{CM:ass:T}. We now prove the first inequality.

For a fixed $\omega$, there is the unique integer $s\geqslant1$ such that
\begin{equation}\label{CM:eq:TFC1}
 9\left(\frac{p}{1+p}\right)^s\|\hat{f}\|_{L_1\cap L_2}
 <\omega\left(M_{\Omega,f}^{(0)}-f^*\right)
 \leqslant9\left(\frac{p}{1+p}\right)^{s-1}\|\hat{f}\|_{L_1\cap L_2}.
\end{equation}
Moreover, under Assumption \ref{CM:ass:T}, $f$ is a $(p,q)$-type tempered function on $\Omega$, then
\begin{equation}\label{CM:eq:TFC2}
 \left(\frac{p}{1+p}\right)^k\left(M_{\Omega,f}^{(0)}-f^*\right) <M_{\Omega,f}^{(k)}-f^*;
\end{equation}
meanwhile, under Assumption \ref{CM:ass:HLFD}, $f$ is a $(\rho,p)$-type HLFDF, so it follows from Lemma \ref{CM:lem:HLFDF&KI} that for every $k\geqslant0$ and kernel parameters as in Lemma \ref{CM:lem:HLFDFinterpolating}, the Gaussian kernel interpolant $\mathcal{I}_{\chi^{(k)}} f$ satisfies
\begin{equation}\label{CM:eq:TFC3}
 \|\mathcal{I}_{\chi^{(k)}} f-f\|_{L_\infty(D^{(k)})}
 <9\left(\frac{p}{1+p}\right)^{k+s}\|\hat{f}\|_{L_1\cap L_2}.
\end{equation}
Hence, it follows from \eqref{CM:eq:TFC1}-\eqref{CM:eq:TFC3} that
\begin{equation*}
 \|\mathcal{I}_{\chi^{(k)}} f-f\|_{L_\infty(D^{(k)})}
 <\omega\left(M_{\Omega,f}^{(k)}-f^*\right),
\end{equation*}
as desired.
\end{proof}

In the next section, Lemma \ref{CM:lem:TF} will be used to guarantee the linear convergence and control the bound of the number of function evaluations per contraction as a constant value.

We now introduce the weak tempered condition.
\renewcommand{\theassuma}{A2a}
\begin{assuma}[Weak tempered function]
\label{CM:ass:Ta}
The function $f$ is a $q$-type tempered function, that is, for all $k\geqslant0$, there exists $q>0$ such that
\begin{equation}\label{CM:eq:TFa}
\frac{q}{1+q}\left(M_{\Omega,f}^{(k)}-f^*\right)<M_{\Omega,f}^{(k+1)}-f^*
<\frac{1}{1+q}\left(M_{\Omega,f}^{(k)}-f^*\right).
\end{equation}
\end{assuma}
\begin{rem}\label{CM:rem:qp}
For a fixed $p>0$ from Assumption \ref{CM:ass:HLFD}, when $q\geqslant p$, we have $\frac{p}{1+p}\leqslant\frac{q}{1+q}$, then
\begin{equation*}
\frac{p}{1+p}\left(M_{\Omega,f}^{(k)}-f^*\right)\leqslant
\frac{q}{1+q}\left(M_{\Omega,f}^{(k)}-f^*\right)<M_{\Omega,f}^{(k+1)}-f^*
<\frac{1}{1+q}\left(M_{\Omega,f}^{(k)}-f^*\right),
\end{equation*}
which is actually Assumption \ref{CM:ass:T}, so this weak version is used to treat the case for $q<p$.
\end{rem}

\renewcommand{\thelema}{6a}
\begin{lema}\label{CM:lem:TFa}
Under Assumptions \ref{CM:ass:HLFD} and \ref{CM:ass:Ta}, suppose $\{D^{(k)}\}$ is defined by Definition \ref{CM:defn:D} with \eqref{CM:eq:SDMed}, in which, $\chi^{(k)}$ is quasi-uniformly distributed over $D^{(k)}\subset\Omega$ with
\begin{equation*}
  q_{\chi^{(k)}}=\tau'\cdot C^{-\frac{1}{n}} 2^{-\frac{kl+s}{n}}\pi/\rho
  ~~\textrm{and}~~
  h_{D,\chi^{(k)}}=\tau''\cdot C^{-\frac{1}{n}}2^{-\frac{kl+s}{n}}\pi/\rho
  ~~\textrm{for}~~\tau'\leqslant1\leqslant\tau'',
\end{equation*}
where the constant $C$ is as in Lemma \ref{CM:lem:HLFDF&KI}. Then, for all $k\in\mathbb{N}_0$ and $\omega<q<p$, there are unique natural numbers $s$, $l>1$ and kernel parameter $\sigma>2^{\frac{kl+s+1}{n}}\rho$ such that Gaussian kernel interpolant $\mathcal{I}_{\chi^{(k)}} f$ satisfies the error bound condition
\begin{equation*}
 \|\mathcal{I}_{\chi^{(k)}} f-f\|_{L_\infty(D^{(k)})}
 <\omega\left(M_{\Omega,f}^{(k)}-f^*\right)
\end{equation*}
with the strong convergence condition
\begin{equation*}
  M_{\Omega,f}^{(k)}-f^*\leqslant\frac{1}{1+q}\max_{x\in D^{(k)}}[f(x)-f^*].
\end{equation*}
\end{lema}
\begin{proof}
The second inequality holds from Assumption \ref{CM:ass:Ta} and $\max_{x\in D^{(k)}}f(x)=M_{\Omega,f}^{(k-1)}$. We now prove the first inequality. For a fixed $\omega$, there is the unique integer $s\geqslant1$ such that
\begin{equation}\label{CM:eq:TFCa1}
 9\left(\frac{p}{1+p}\right)^s\|\hat{f}\|_1
 <\omega\left(M_{\Omega,f}^{(0)}-f^*\right)
 \leqslant9\left(\frac{p}{1+p}\right)^{s-1}\|\hat{f}\|_{L_1\cap L_2}.
\end{equation}
Moreover, under Assumption \ref{CM:ass:Ta}, $f$ is a $q$-type tempered function on $\Omega$ with $q<p$ (as mentioned in Remark \ref{CM:rem:qp}), there is a unique integer $l>1$ such that
\begin{equation*}
 \left(\frac{p}{1+p}\right)^l\leqslant\frac{q}{1+q}
 <\left(\frac{p}{1+p}\right)^{l-1},
\end{equation*}
then
\begin{equation}\label{CM:eq:TFCa2}
 \left(\frac{p}{1+p}\right)^{kl}\left(M_{\Omega,f}^{(0)}-f^*\right)\leqslant
 \left(\frac{q}{1+q}\right)^k\left(M_{\Omega,f}^{(0)}-f^*\right)
 <M_{\Omega,f}^{(k)}-f^*;
\end{equation}
meanwhile, under Assumption \ref{CM:ass:HLFD}, $f$ is a $(\rho,p)$-type HLFDF, so it follows from Lemma \ref{CM:lem:HLFDF&KI} that for every $k\geqslant0$ and kernel parameters as in Lemma \ref{CM:lem:HLFDFinterpolating}, the interpolant $\mathcal{I}_{\chi^{(k)}} f$ satisfies
\begin{equation}\label{CM:eq:TFCa3}
 \|\mathcal{I}_{\chi^{(k)}} f-f\|_{L_\infty(D^{(k)})}
 <9\left(\frac{p}{1+p}\right)^{kl+s}\|\hat{f}\|_{L_1\cap L_2}.
\end{equation}
Hence, from \eqref{CM:eq:TFCa1}-\eqref{CM:eq:TFCa3}, the error bound condition $\|\mathcal{I}_{\chi^{(k)}} f-f\|_{L_\infty (D^{(k)})}<\omega(M_{\Omega,f}^{(k)}-f^*)$ holds as desired.
\end{proof}

Lemma \ref{CM:lem:TFa} will be used to guarantee the linear convergence and control the growth of the bound of the number of function evaluations per contraction not exceeding the exponential order. In the following, we will show that for all continuous functions, the weak tempered condition holds with high probability.

\begin{prop}\label{CM:prop:TF}
If $\Omega$ is a compact set and $f$ is continuous and not a constant on $\Omega$, then for any $\epsilon>0$, there must exists a $q=q(\epsilon)\in(0,1)$ such that
\begin{equation}\label{CM:eq:TFW}
 \frac{q}{1+q}\left(M_{\Omega,f}^{(k)}-f^*\right)<M_{\Omega,f}^{(k+1)}-f^*
<\frac{1}{1+q}\left(M_{\Omega,f}^{(k)}-f^*\right)
\end{equation}
holds for every $k\geqslant0$ with probability at least $1-\epsilon$, where $M_{\Omega,f}^{(k)}=\textrm{prctile} \left(f(\xi),2^{-k}\cdot100\%\right)$ and $\xi$ is a uniformly distributed random variable on $\Omega$.
\end{prop}
\begin{proof}
Consider the sublevel set sequence $\{E^{(k)}\}$ defined by \eqref{CM:eq:SLMed}, let $\xi^{(k)}$ be a uniformly distributed random variable on $E^{(k)}$, then the desired inequality can be rewritten as
\begin{equation*}
  \frac{q}{1+q}\left(\textrm{Median}f(\xi^{(k)})-f^*\right)<
  \textrm{Median}f(\xi^{(k+1)})-f^*
  <\frac{1}{1+q}\left(\textrm{Median}f(\xi^{(k)})-f^*\right).
\end{equation*}
Let $m^{(k)}=\textrm{Median}[f(\xi^{(k)})-f^*]$, $a^{(k)}=\textrm{Mean}[f(\xi^{(k)})-f^*]$ and $d^{(k)}=\sqrt{\textrm{Var}[f(\xi^{(k)})-f^*]}$ for convenience. Proving the first inequality is equivalent to showing that on each $E^{(k+1)}$, the upper bound of $f$, i.e., $m^{(k)}$, is less than $\frac{1+q}{q}$ of the median of $f$, i.e., $m^{(k+1)}$, with probability at least $1-\epsilon$.
	
Since $f$ is continuous and not a constant on $D^{(k+1)}\subseteq\Omega$, we have $0<m^{(k+1)}<\infty$ and $0<d^{(k+1)}<\infty$ and there exists a $C>0$ such that $d^{(k+1)}=Cm^{(k+1)}$. First, the distance between the median and the mean is bounded by standard deviation \citep{MallowsC1991_Median&Mean}, i.e.,
\begin{equation*}
 m^{(k+1)}-d^{(k+1)}\leqslant a^{(k+1)}\leqslant m^{(k+1)}+d^{(k+1)};
\end{equation*}
and it follows from the Chebyshev-Cantelli inequality that
\begin{equation*}
 P\left(f(\xi^{(k+1)})-f^*>a^{(k+1)}+\frac{\sqrt{1-\epsilon}}
 {\sqrt{\epsilon}}d^{(k+1)}\right)\leqslant\epsilon.
\end{equation*}
So it holds that
\begin{align*}
 \textrm{prctile}\left(f(\xi^{(k+1)})-f^*,(1-\epsilon)\cdot100\%\right)
 <&a^{(k+1)}+\frac{\sqrt{1-\epsilon}}{\sqrt{\epsilon}}d^{(k+1)} \\
 \leqslant&m^{(k+1)}+\left(1+\frac{\sqrt{1-\epsilon}}
 {\sqrt{\epsilon}}\right)d^{(k+1)} \\
 =&\frac{\sqrt{\epsilon}+C+C\sqrt{1-\epsilon}}{\sqrt{\epsilon}}m^{(k+1)},
\end{align*}
or equivalently, $\frac{q}{1+q}m^{(k)}<m^{(k+1)}$ holds for all $k\geqslant0$ with probability at least $1-\epsilon$, where $q=\sqrt{\epsilon}/(C+C\sqrt{1-\epsilon})$.
	
Similarly, by considering a translation $f(x)-f^*-m^{(k)}$, the second inequality $m^{(k+1)}<\frac{1}{1+q}m^{(k)}$ also holds for every $k\geqslant0$ with probability at least $1-\epsilon$. And it is clear that these two inequalities only make sense simultaneous when $q<1$.
\end{proof}

\subsection{Critical regular functions}
\label{CM:s3:3}

First, we introduce the last assumption on the objective function. This condition is not an essential requirement, but for convenience of theoretical analysis.

\begin{assum}[Critical regular function]
\label{CM:ass:CR}
The function $f$ is a critical regular, that is, the set of all critical points of $f:\Omega\to\mathbb{R}$ has a zero $n$-dimensional Lebesgue measure, where a critical point is a $x\in\Omega$ where the gradient is undefined or is equal to zero.
\end{assum}

This critical regular condition guarantees that the contraction factor is equal to $1/2$ in expectation for the median-type method. Especially, consider a special sublevel set sequence
\begin{equation}\label{CM:eq:SLMed}
 E^{(k+1)}=\left\{x\in\Omega:f(x)\leqslant M_{\Omega,f}^{(k)}\right\},
\end{equation}
where $M_{\Omega,f}^{(k)}$ is as in \eqref{CM:eq:SDMed}; then under Assumption \ref{CM:ass:CR}, we have $\mu(E^{(k+1)})=\frac{1}{2}\mu(E^{(k)})$ and
\begin{equation}\label{CM:eq:SLMedV}
 \mu(E^{(k)})=\frac{1}{2^k}\mu(\Omega).
\end{equation}

A set of critical points with a large non-zero measure may cause the contraction factors to be too large or too small in some iterations. Some functions that do not meet Assumption \ref{CM:ass:CR} might be difficult to solve, however, they can be excluded by Assumption \ref{CM:ass:HLFD}. In contrast, the hierarchical low-frequency dominant condition is an essential requirement for effective contractions.

Finally, we establish an upper bound on the measure of $D^{(k)}$ defined by \eqref{CM:eq:SDMed}. With the help of Assumption \ref{CM:ass:HLFD}, it can limit the computational complexity of each kernel-based approximation on $D^{(k)}$ in the next section.
\begin{lem}\label{CM:lem:muDB}
Under Assumptions \ref{CM:ass:T} (or \ref{CM:ass:Ta}) and \ref{CM:ass:CR}, suppose $\{D^{(k)}\}$ is defined by Definition \ref{CM:defn:D} with \eqref{CM:eq:SDMed} and $\omega<q$. Then, for all $k\in\mathbb{N}$, it follows that
\begin{equation*}
  \mu(D^{(k)})\leqslant2^{-(k-1)}\mu(\Omega).
\end{equation*}
\end{lem}
\begin{proof}
According to Definition \ref{CM:defn:D} and \eqref{CM:eq:SDMed}, for any $x'\in D^{(k)}$, it follows from Lemma \ref{CM:lem:SConv} that
\begin{equation*}
  f(x')-f^*\leqslant(1+\omega)\big(M_{\Omega,f}^{(k-1)}-f^*\big),
\end{equation*}
according to Assumption \ref{CM:ass:T} (or \ref{CM:ass:Ta}), it holds that
\begin{equation*}
  M_{\Omega,f}^{(k-1)}-f^*\leqslant\frac{1}{1+q}\big(M_{\Omega,f}^{(k-2)}-f^*\big),
\end{equation*}
thus, it follows from $\omega<q$ that
\begin{equation*}
  f(x')-f^*\leqslant\frac{1+\omega}{1+q}\big(M_{\Omega,f}^{(k-2)}-f^*\big)
  \leqslant M_{\Omega,f}^{(k-2)}-f^*,
\end{equation*}
that is, $f(x')\leqslant M_{\Omega,f}^{(k-2)}$, therefore, $x'\in E^{(k-1)}$, that is, $D^{(k)}\subset E^{(k-1)}$, where $E^{(k)}$ is defined as \eqref{CM:eq:SLMed}, i.e.,
\begin{equation*}
  E^{(k+1)}=\left\{x\in\Omega:f(x)\leqslant M_{\Omega,f}^{(k)}\right\},
\end{equation*}
by noting that $\mu(E^{(k-1)})=2^{-(k-1)}\mu(\Omega)$, we obtain the desired result.
\end{proof}

\section{Categories of continuous optimization problems}
\label{CM:s4}

According to the contractibility, all the possible continuous optimization problems can be divided into the following three categories: logarithmic time contractible, polynomial time contractible, or noncontractible.

\subsection{Logarithmic time contractible}

\begin{defn}
 The problem \eqref{CM:eq:COP} is said to be logarithmic time contractible if there exist three suitable parameters $\rho,p,q>0$ such that $f$ satisfies Assumptions \ref{CM:ass:HLFD}, \ref{CM:ass:T} and \ref{CM:ass:CR}.
\end{defn}

For such a problem, the key to obtaining logarithmic time efficiency comes from two reasons: (i) there is a certain approximation $\mathcal{A}^{(k)}f$ such that both the error bound condition and the strong convergence condition are satisfied; (ii) the approximation $\mathcal{A}^{(k)}f$ can be fully constructed by some samples on $D^{(k)}$ and the size of these samples does not exceed a certain fixed upper bound. According to these three points, we can prove the following theorem.
\begin{thm}\label{CM:thm:LTC}
Suppose the problem \eqref{CM:eq:COP} is logarithmic time contractible with parameters $(\rho,p,q)$ and $\{D^{(k)}\}$ is defined by Definition \ref{CM:defn:D} with \eqref{CM:eq:SDMed}, in which, $\chi^{(k)}$ is quasi-uniformly distributed over $D^{(k)}\subset\Omega$ with
\begin{equation*}
  q_{\chi^{(k)}}=\tau'\cdot C^{-\frac{1}{n}} 2^{-\frac{k+s}{n}}\pi/\rho
  ~~\textrm{and}~~
  h_{D,\chi^{(k)}}=\tau''\cdot C^{-\frac{1}{n}}2^{-\frac{k+s}{n}}\pi/\rho
  ~~\textrm{for}~~\tau'\leqslant1\leqslant\tau'',
\end{equation*}
and $\mathcal{A}^{(k)}f$ is given by Gaussian kernel interpolant $\mathcal{I}_{\chi^{(k)}}f$ with parameter satisfying \eqref{CM:eq:kpara}. Then, for any $\omega<q$ and $\epsilon>0$, there exist $C>1$ and $K\geqslant1$ such that after $K$ contractions, it holds that the upper bound
\begin{equation*}
 \max_{x\in D^{(K)}}[f(x)-f^*]<\epsilon,
\end{equation*}
with the linear convergence rate
\begin{equation*}
 \max_{x\in D^{(k)}}[f(x)-f^*]<\left(\frac{1+\omega}{1+q}\right)^k
 \max_{x\in\Omega}[f(x)-f^*],
\end{equation*}
the total number of function evaluations $\mathcal{O}\big(N_{\Omega,f} \cdot\log_{\frac{1+\omega}{1+q}}\epsilon\big)$, and the total time complexity $\mathcal{O}\big(N_{\Omega,f}^3\cdot\log_{\frac{1+\omega}{1+q}}\epsilon\big)$, where $N_{\Omega,f}=\mathcal{O}\big(2^{s+1}\rho^n/\pi^n\big)$ and $s$ is the unique integer such that
\begin{equation*}
9\left(\frac{p}{1+p}\right)^s\|\hat{f}\|_{L_1\cap L_2}
<\omega\Big(\max_{x\in\Omega}f(x)-f^*\Big)\leqslant9
\left(\frac{p}{1+p}\right)^{s-1}\|\hat{f}\|_{L_1\cap L_2}.
\end{equation*}
\end{thm}
\begin{proof}
From Lemma \ref{CM:lem:TF}, for any $k\in\mathbb{N}_0$, since $\chi^{(k)}$ is quasi-uniformly distributed w.r.t. a sampling density of $C2^{k+s}\rho^n/\pi^n$, there exist suitable kernel parameters as in \eqref{CM:eq:kpara} such that Gaussian kernel interpolant $\mathcal{I}_{\chi^{(k)}}f$ satisfies the error bound condition
\begin{equation*}
\|\mathcal{I}_{\chi^{(k)}}f-f\|_{L_\infty(D^{(k)})}
<\omega\big(M_{\Omega,f}^{(k)}-f^*\big)
\end{equation*}
with the strong convergence condition
\begin{equation*}
  M_{\Omega,f}^{(k)}-f^*\leqslant\frac{1}{1+q}\max_{x\in D^{(k)}}[f(x)-f^*].
\end{equation*}
So it follows from the strong convergence Theorem \ref{CM:thm:SConv} that
\begin{equation*}
 \max_{x\in D^{(k)}}[f(x)-f^*]\leqslant
 \left(\frac{1+\omega}{1+q}\right)^k\max_{x\in\Omega}[f(x)-f^*]
 =\left(\frac{1+\omega}{1+q}\right)^k(f^{**}-f^*).
\end{equation*}

According to $\mu(\Omega)\leqslant1$ and Lemma \ref{CM:lem:muDB}, i.e.,
\begin{equation*}
\mu(D^{(k)})\leqslant2^{-(k-1)}\mu(\Omega)=2^{-(k-1)},
\end{equation*}
there is $N_{\Omega,f}=\mathcal{O}(2^{s+1}\rho^n/\pi^n)$ such that
\begin{equation*}
N^{(k)}=C\mu(D^{(k)})2^{k+s}\rho^n/\pi^n\leqslant
C2^{s+1}\rho^n/\pi^n\leqslant N_{\Omega,f}.
\end{equation*}
Then, for a fixed accuracy $\epsilon>0$, there exists a $K>0$ such that
\begin{equation*}
 \left(\frac{1+\omega}{1+q}\right)^K(f^{**}-f^*)<\epsilon
 \leqslant\left(\frac{1+\omega}{1+q}\right)^{K-1}(f^{**}-f^*),
\end{equation*}
hence, after $K$ contractions, one gets the approximate solution set $D^{(K)}$ with an error bound
\begin{equation*}
 \max_{x\in D^{(K)}}[f(x)-f^*]\leqslant\epsilon,
\end{equation*}
and the total number of function evaluations is less than
\begin{equation*}
 \sum_{k=0}^{K-1}N^{(k)}\leqslant\sum_{k=0}^{K-1}N_{\Omega,f}
 =\mathcal{O}\left(N_{\Omega,f}\cdot\log_{\frac{1+\omega}{1+q}}\epsilon\right),
\end{equation*}
further, since the Gaussian kernel interpolant $\mathcal{I}_{\chi^{(k)}}f$ can be computed by GMRES \citep{Saad1986A_GMRES} in $\mathcal{O}(N^{(k)})^2$ iterations, even if the model is updated every time a sample is added, the complexity of each contraction still does not exceed $\mathcal{O}(N_{\Omega,f}^3)$, so the total time complexity is less than
\begin{equation*}
 \sum_{k=0}^{K-1}N^3_{\Omega,f}=\mathcal{O}
 \left(N^3_{\Omega,f}\cdot\log_{\frac{1+\omega}{1+q}}\epsilon\right),
\end{equation*}
taking a logarithmic time for any desired accuracy $\epsilon$.
\end{proof}

Notice that Lemma \ref{CM:lem:TF} also hold for any $u^{(k)}$ satisfying
\begin{equation}\label{CM:eq:LTCu}
\frac{p}{1+p}\left(u^{(k)}-f^*\right)<u^{(k+1)}-f^*
<\frac{1}{1+q}\left(u^{(k)}-f^*\right),~~\textrm{where}~~pq<1,
\end{equation}
that is,
\begin{lem}\label{CM:lem:TFr}
Under Assumptions \ref{CM:ass:HLFD}, suppose $\{D^{(k)}\}$ is defined by Definition \ref{CM:defn:D} with \eqref{CM:eq:LTCu}, in which, $\chi^{(k)}$ is quasi-uniformly distributed over $D^{(k)}\subset\Omega$ with
\begin{equation*}
  q_{\chi^{(k)}}=\tau'\cdot C^{-\frac{1}{n}} 2^{-\frac{k+s}{n}}\pi/\rho
  ~~\textrm{and}~~
  h_{D,\chi^{(k)}}=\tau''\cdot C^{-\frac{1}{n}}2^{-\frac{k+s}{n}}\pi/\rho
  ~~\textrm{for}~~\tau'\leqslant1\leqslant\tau'',
\end{equation*}
where the constant $C$ is as in Lemma \ref{CM:lem:HLFDF&KI}. Then, for all $k\in\mathbb{N}_0$ and $\omega<q$, there exist a unique integer $s\geqslant1$ and kernel parameter $\sigma>2^{\frac{k+s+1}{n}}\rho$ such that Gaussian kernel interpolant $\mathcal{I}_{\chi^{(k)}} f$ satisfies the error bound condition
\begin{equation*}
 \|\mathcal{I}_{\chi^{(k)}} f-f\|_{L_\infty(D^{(k)})}
 <\omega\left(u^{(k)}-f^*\right)
\end{equation*}
with the strong convergence condition
\begin{equation*}
  u^{(k)}-f^*\leqslant\frac{1}{1+q}\max_{x\in D^{(k)}}[f(x)-f^*].
\end{equation*}
\end{lem}

Thus, an immediate corollary of Theorem \ref{CM:thm:LTC} is:
\begin{cor}\label{CM:cor:LTC}
Suppose there exist $\rho,p,q>0$ such that the problem \eqref{CM:eq:COP} satisfies Assumption \ref{CM:ass:HLFD} and $\{D^{(k)}\}$ is defined by Definition \ref{CM:defn:D} with \eqref{CM:eq:LTCu} and $\mu(D^{(k+1)})\leqslant\frac{1}{2}\mu(D^{(k)})$, in which, $\chi^{(k)}$ is quasi-uniformly distributed over $D^{(k)}\subset\Omega$ with
\begin{equation*}
  q_{\chi^{(k)}}=\tau'\cdot C^{-\frac{1}{n}} 2^{-\frac{k+s}{n}}\pi/\rho
  ~~\textrm{and}~~
  h_{D,\chi^{(k)}}=\tau''\cdot C^{-\frac{1}{n}}2^{-\frac{k+s}{n}}\pi/\rho
  ~~\textrm{for}~~\tau'\leqslant1\leqslant\tau'',
\end{equation*}
and $\mathcal{A}^{(k)}f$ is given by Gaussian kernel interpolant $\mathcal{I}_{\chi^{(k)}}f$ with parameter satisfying \eqref{CM:eq:kpara}. Then, for any $\omega<q$, all conclusions of Theorem \ref{CM:thm:LTC} also hold, where $N_{\Omega,f}=\mathcal{O}\big(2^s\rho^n/\pi^n\big)$ and $s$ is the unique integer such that
\begin{equation*}
9\left(\frac{p}{1+p}\right)^s\|\hat{f}\|_{L_1\cap L_2}
<\omega\Big(u^{(0)}-f^*\Big)\leqslant9
\left(\frac{p}{1+p}\right)^{s-1}\|\hat{f}\|_{L_1\cap L_2}.
\end{equation*}
\end{cor}

We have already seen a typical logarithmic time contractible case in Figure \ref{CM:fig:A1}, here are some other examples including various types.

First, we show the comparison between the CM and the grid search (GS) for Lipschitz continuous and H\"{o}lder continuous objectives, as illustrated in Figures \ref{CM:fig:LTL} and \ref{CM:fig:LTH}, respectively. For fairness, the positions of these global minimizers are designed to avoid being directly covered by certain equally spaced grids. In the usual sense, they are not very ``good'' functions because they all have many local minima so that a solution is easily trapped in any one of the local minima; however, in the sense of contraction, they are still ``good'' functions since the contraction strategy could be performed very successfully. The GS with gradual refinements might perform better. We did not compare with it here because it has no guarantee of global convergence.

\begin{figure}[tbhp]
  \centering
  \subfigure{\includegraphics[width=0.325\textwidth]{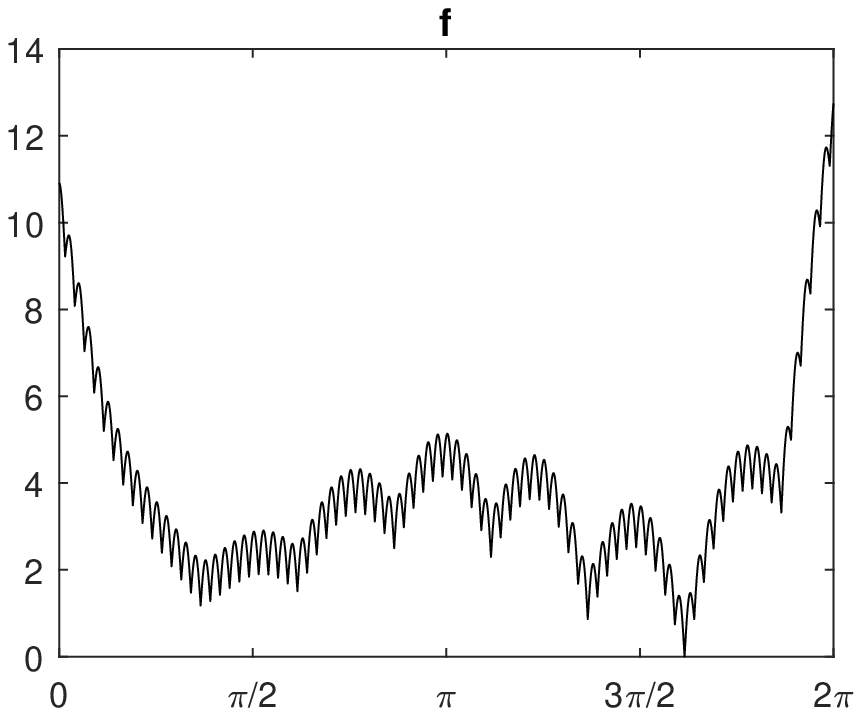}}
  \subfigure{\includegraphics[width=0.325\textwidth]{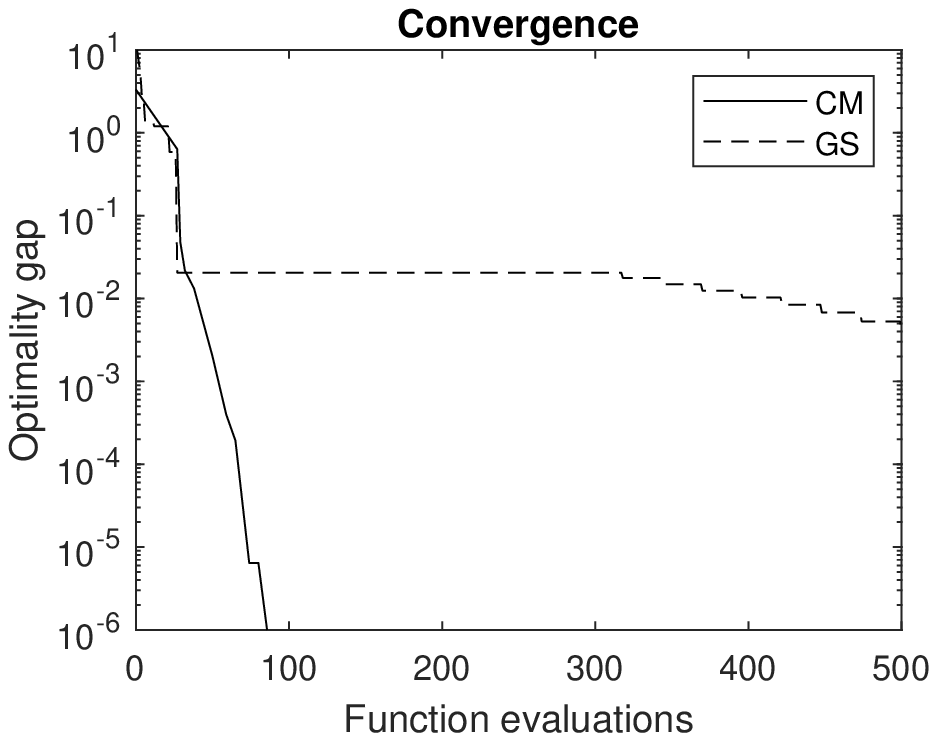}}
  \subfigure{\includegraphics[width=0.325\textwidth]{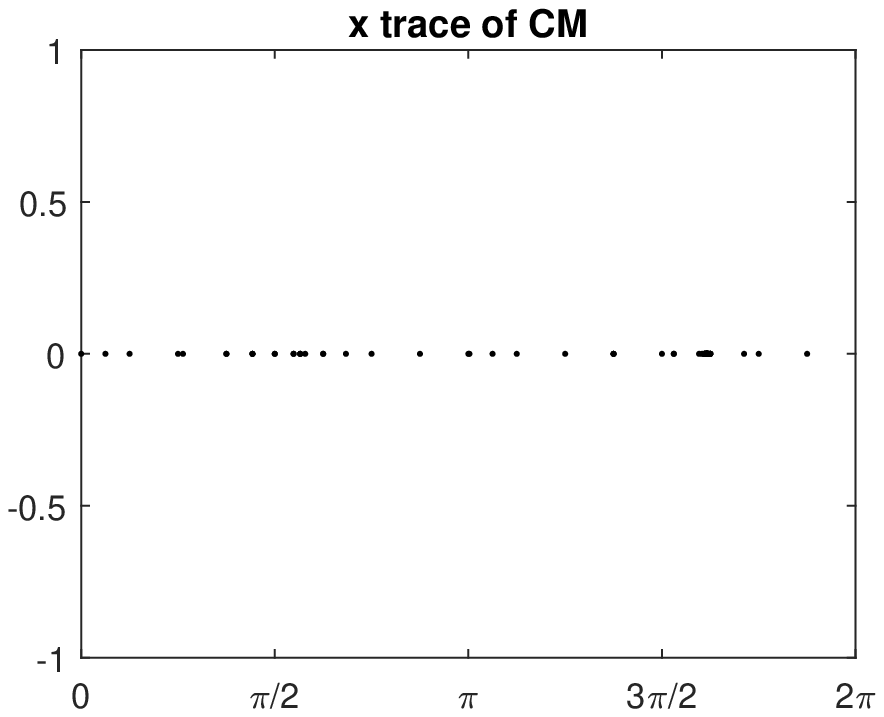}}
  \caption{An example of logarithmic time contractible and Lipschitz continuous.
  Left: the Lipschitz continuous objective $f(x)=0.2(x-1)(x-1.9)(x-4)(x-5.6)+ 0.5x|\sin(4(x-e))|+|\sin(40(x-e))|$ with the unique global minimum $0$ located at $x^*=5.074476$ on the domain $\Omega=[0,2\pi]$. And the position of the minimizer is designed to avoid being directly covered by certain equally spaced grids. Middle: the convergence behaviors of the CM and the GS, where CM (i.e., Algorithm \ref{CM:alg:CM}) is run with parameters $K=15,m=2,\textrm{minIterInner}=1,\omega=1$, $c_k=\bar{c}=50$ and $t_k=\bar{t}=1$ for $k=0,1,\cdots,K-1$; and the GS is based on equally spaced grids with different numbers of nodes on $\Omega$. The optimality gap is defined as $f_{\textrm{best}}^*-f^*$, where $f_{\textrm{best}}^*$ is the current best. Right: x trace of CM after all fifteen contractions.}
  \label{CM:fig:LTL}
\end{figure}

\begin{figure}[tbhp]
  \centering
  \subfigure{\includegraphics[width=0.325\textwidth]{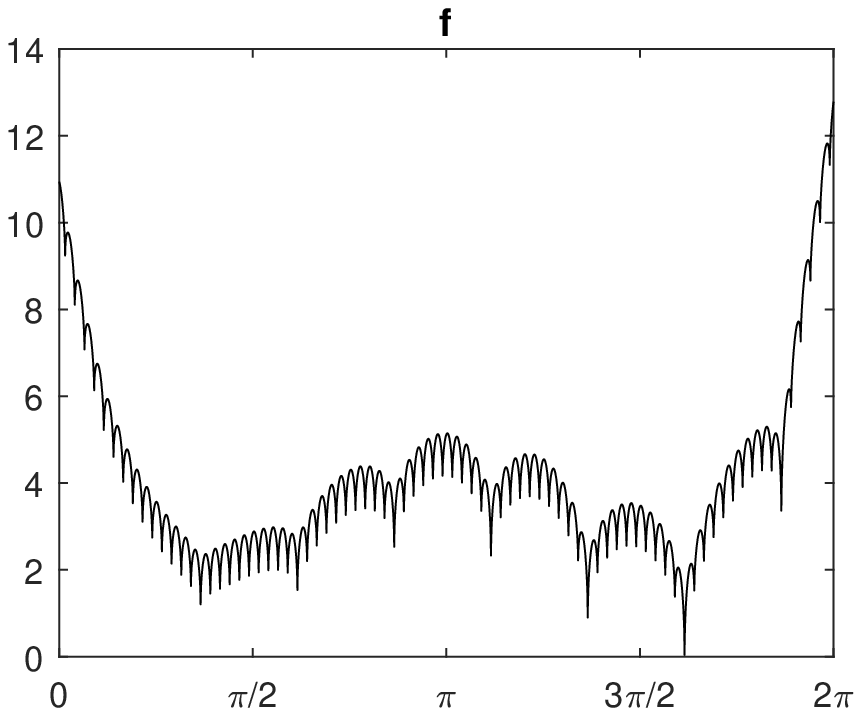}}
  \subfigure{\includegraphics[width=0.325\textwidth]{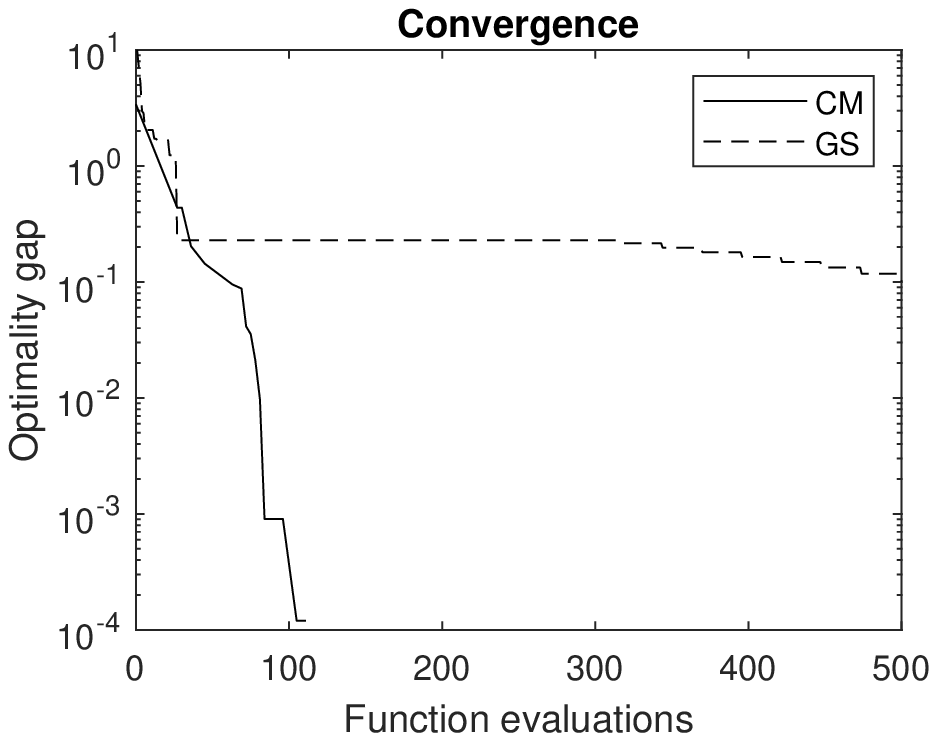}}
  \subfigure{\includegraphics[width=0.325\textwidth]{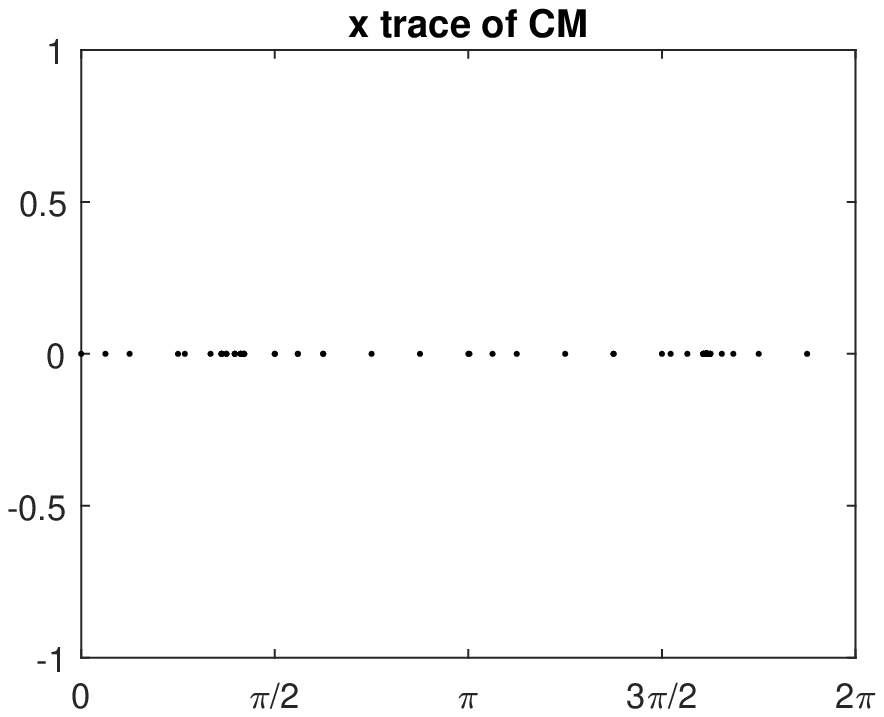}}
  \caption{An example of logarithmic time contractible and H\"{o}lder continuous.
  Left: the H\"{o}lder continuous objective function $f(x)=0.2(x-1)(x-1.9)(x-4)(x-5.6)+0.5x|\sin(4(x-e))|^{0.5}+|\sin(40(x-e))|^{0.5}$ with the unique global minimum $0$ located at $x^*=5.074476$ on the domain $\Omega=[0,2\pi]$. The position of the minimizer is also designed to avoid being directly covered by certain equally spaced grids. Middle: the convergence behaviors of the CM and the GS, where CM (Algorithm \ref{CM:alg:CM}) is run with parameters $K=25,m=2,\textrm{minIterInner}=1,\omega=1$, $c_k=\bar{c}=50$ and $t_k=\bar{t}=2$ for $k=0,1,\cdots,K-1$; and the GS is based on equally spaced grids with different numbers of nodes on $\Omega$. The optimality gap is defined as $f_{\textrm{best}}^*-f^*$, where $f_{\textrm{best}}^*$ is the current best. Right: x trace of CM after all twenty five contractions.}
  \label{CM:fig:LTH}
\end{figure}

Second, we also show the comparison between the CM the BO, and the gradient descent, for a quadratic objective illustrated in Figure \ref{CM:fig:LTQ}. One can see that the CM converges even faster than the gradient descent with the optimal step size, which verifies the established linear convergence and logarithmic complexity w.r.t. the number of function evaluations. Of course, the gradient descent with momentum will have better performance for this strongly convex function, while the CM is obviously designed for nonconvex problems. One may also notice that the CM converges a little bit slower than the BO algorithm in the initial stage. The reason is that the former does more space detection in the current domain to cover the minimizer with the next subdomain, while the latter focuses more on what the current model can predict.

\begin{figure}[tbhp]
  \centering
  \subfigure{\includegraphics[width=0.325\textwidth]{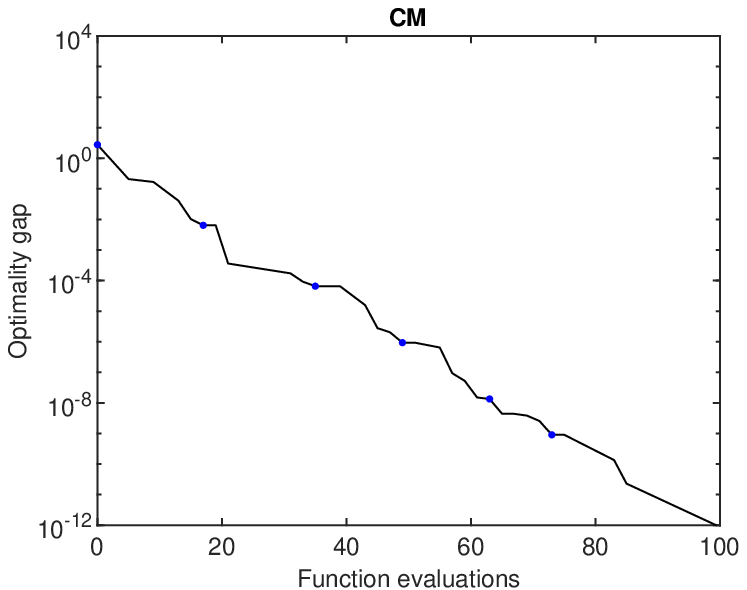}}
  \subfigure{\includegraphics[width=0.325\textwidth]{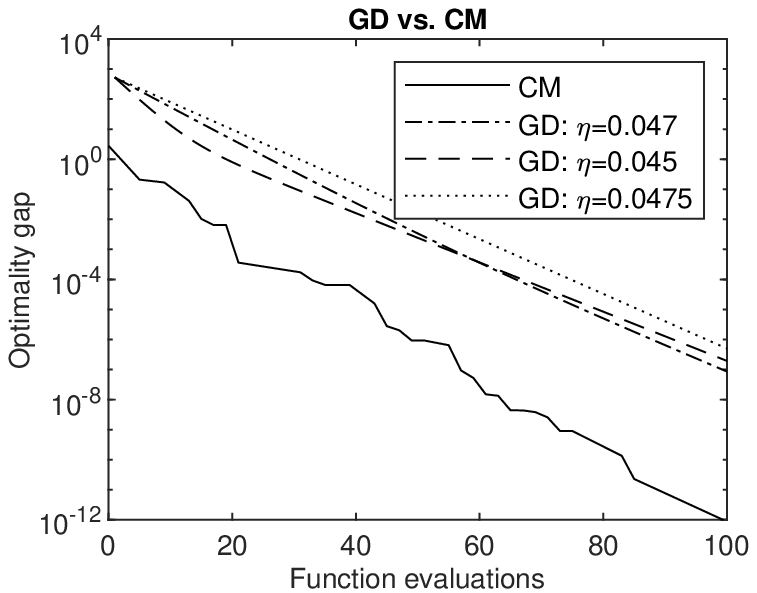}}
  \subfigure{\includegraphics[width=0.325\textwidth]{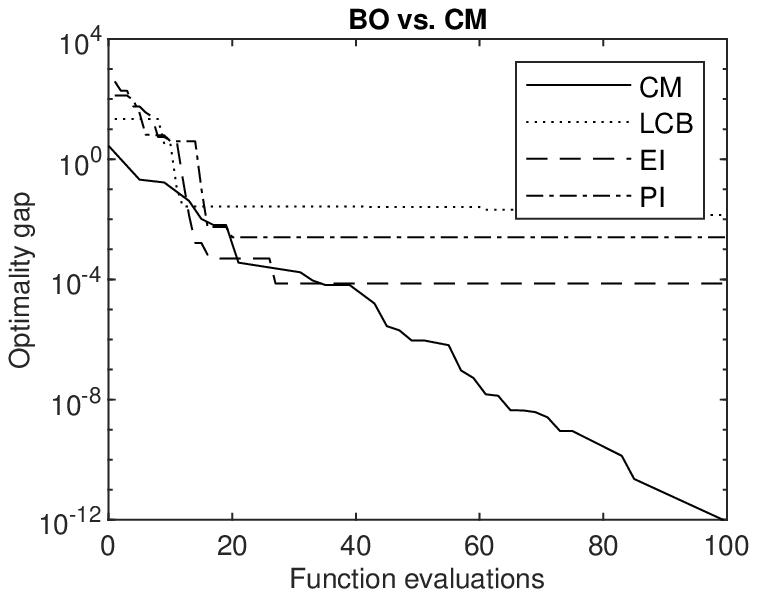}}\\
  \subfigure{\includegraphics[width=0.325\textwidth]{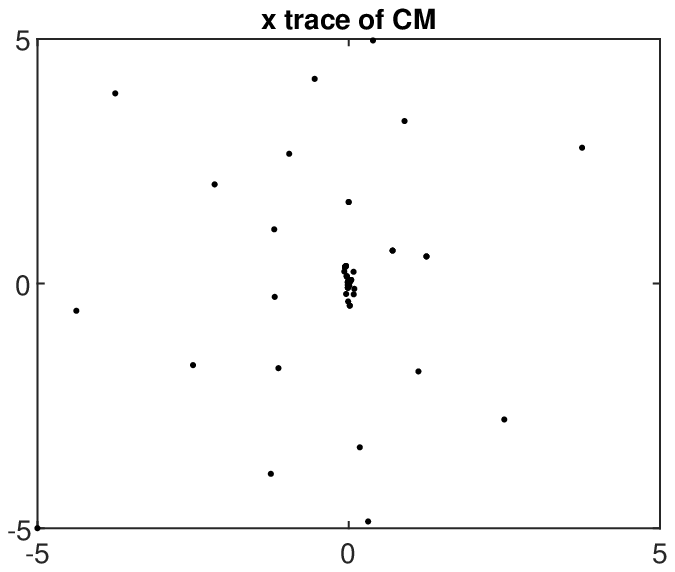}}
  \subfigure{\includegraphics[width=0.325\textwidth]{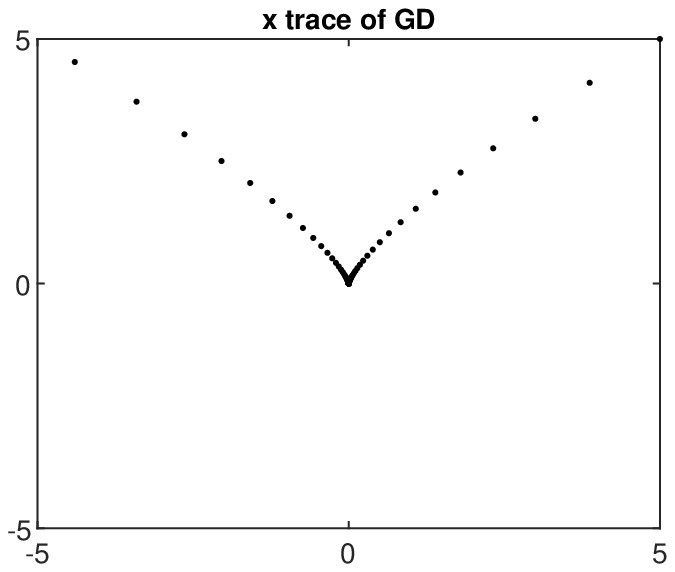}}
  \subfigure{\includegraphics[width=0.325\textwidth]{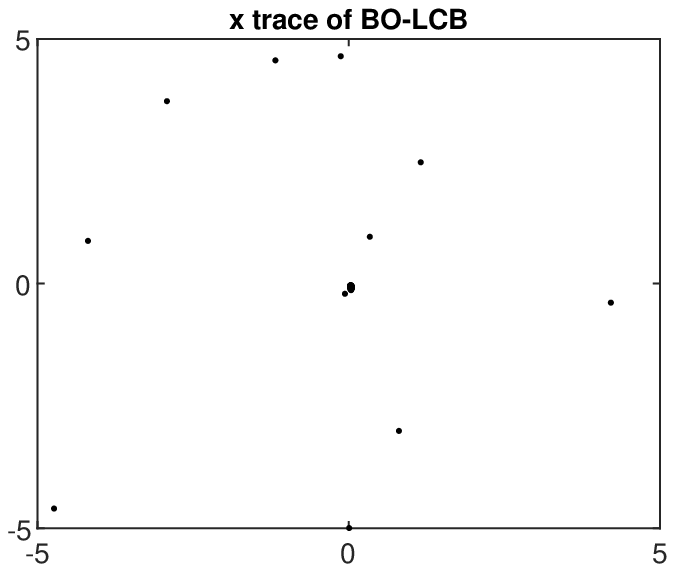}}
  \caption{A quadratic example of logarithmic time contractible. The objective $f(x_1,x_2)=20x_1^2+x_2^2, x_1\in[-5,5], x_2\in[-5,5]$ with the global minimum $0$ located at $(0,0)$. Upper left: the convergence behavior of CM (Algorithm \ref{CM:alg:CM}) with $K=10,m=1,\textrm{minIterInner}=1,\omega=1$, $c_k=\bar{c}=50$ and $t_k=\bar{t}=2$ for every $0\leqslant k\leqslant K-1$. Upper middle: the convergence behaviors of gradient descents with the initial point $(5,5)$ and different stepsizes $\eta=0.045,0.047$ and $0.0475$, where the optimal stepsize is obviously $\eta=0.047$ and the number of function evaluations refers to the gradient. Of course, this function is not well-conditioned because its condition number $\kappa=\frac{L}{l}=20$ with smooth parameter $L=40$ and strongly convex parameter $l=2$, then the gradient descent has a linear rate with contraction constant given by $\frac{\kappa-1}{\kappa+1}\approx0.90$ while the accelerated gradient with momentum has a contraction constant given by $\frac{\sqrt{\kappa}-1}{\sqrt{\kappa}+1}\approx0.63$. Therefore, the accelerated gradient methods will have a better performance, however, this example is only used to illustrate that the contraction algorithm can achieve the expected linear rate. Upper right: the convergence behaviors of Bayesian optimization with three types of acquisition function LCB, EI and PI. For fairness and good reproducibility, we use the MATLAB function \texttt{bayesopt} with default settings to implement these three methods, where the confidence parameter of the BO-LCB is also set as $2$ by default. Lower row: x traces of different algorithms.}
  \label{CM:fig:LTQ}
\end{figure}

Finally, let us look at a slightly more complicated example, which is formed by adding a Gaussian function to the above quadratic function. It has two local minimizers, the one near the origin is deceptive while the other far from the origin is the real global minimizer. As already mentioned, every exploitation on the basis of inadequate information may reduce efficiency or even cause trouble, now Figure \ref{CM:fig:LT2} shows possible convergence failures in some types of Bayesian optimization. So it seems that only when the sample density reaches a certain level, the relevant inference will become reliable, although the aggressive confidence parameter 2 might be the cause of convergence failure for BO-LCB.

\begin{figure}[tbhp]
  \centering
  \subfigure{\includegraphics[scale=0.67]{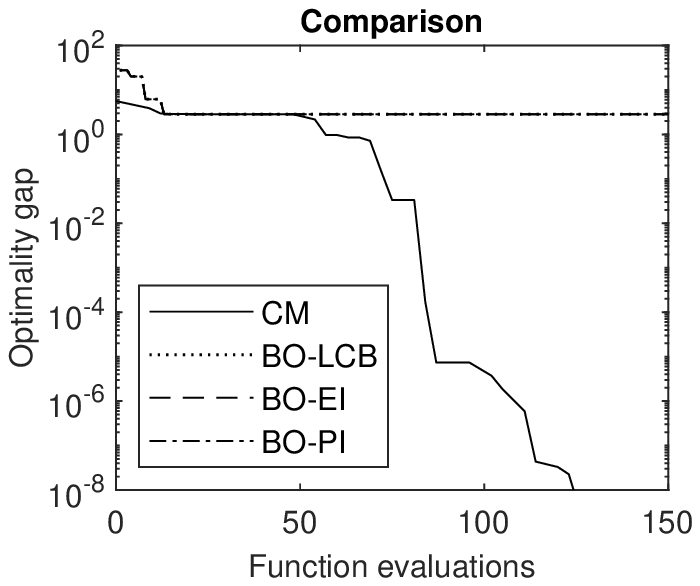}}
  \subfigure{\includegraphics[width=0.32\textwidth]{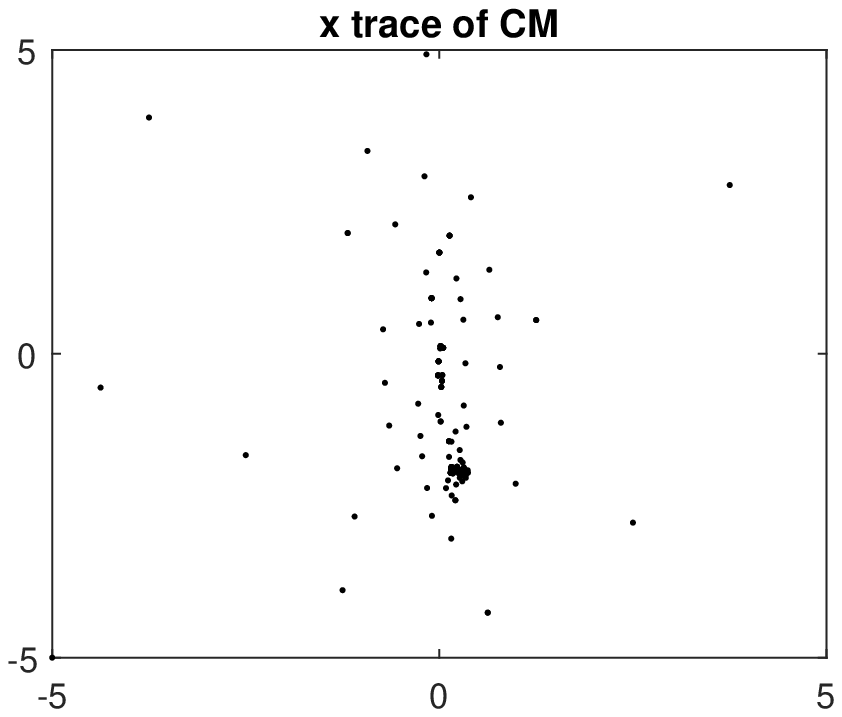}}
  \subfigure{\includegraphics[width=0.32\textwidth]{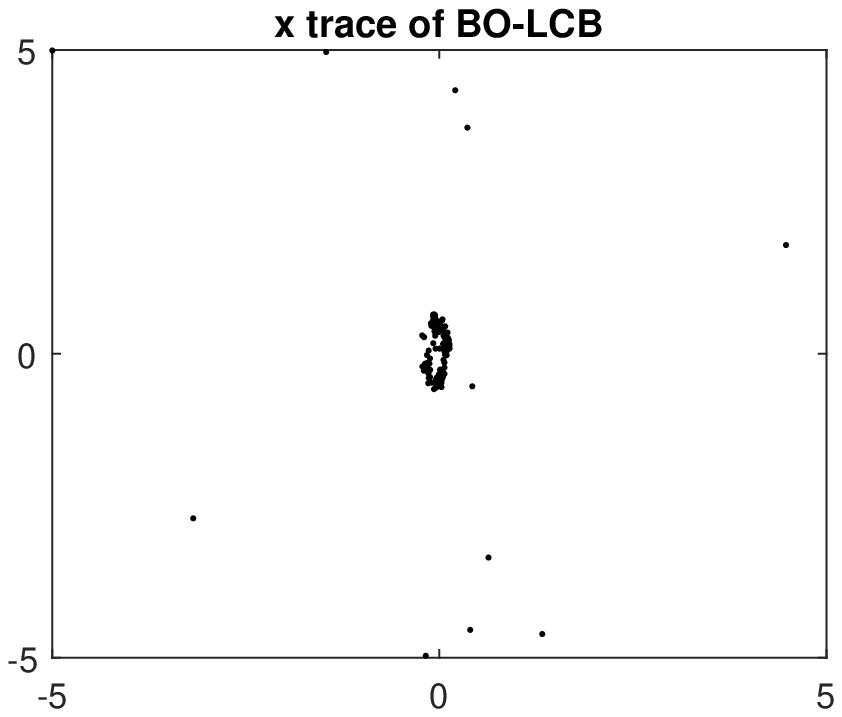}}
  \caption{A logarithmic time contractible example formed by a quadratic function plusing a Gaussian. The objective $f(x_1,x_2)=20x_1^2+x_2^2-10 \exp[-4(x_1-0.5)^2-4(x_2+2)^2]$ with $x_1\in[-5,5], x_2\in[-5,5]$. There are two local minima, one is $4.14\times10^{-7}$ located at $(4.62\times10^{-8},-3.30\times10^{-6})$ and the other, which is the global minima, is $-2.84835122$ located at $(0.31659395,1.9436997)$. Left: convergence comparison between Algorithm \ref{CM:alg:CM} and three different types of Bayesian optimization (BO). Algorithm \ref{CM:alg:CM} is run with $K=18,m=2,\textrm{minIterInner}=1,\omega=1$, $c_k=\bar{c}=50$ and $t_k=\bar{t}=2$ for every $0\leqslant k\leqslant K-1$. BO is implemented by the MATLAB function \texttt{bayesopt}, where the confidence parameter of BO-LCB is also set to $2$ by default. BO converges to the former local minima while the  contraction automatically increases the strength of space detection to ensure a full coverage. For BO-LCB, the aggressive confidence parameter $2$ might be the cause of convergence failure; of course, such an objective function may also be excluded from the convergence of probability. Middle: x trace of CM. Right: x trace of BO-LCB.}
  \label{CM:fig:LT2}
\end{figure}

According to the above examples from Lipschitz continuous, H\"{o}lder continuous, strong convex to smooth nonconvex objectives, it can be seen that one of the main advantages of the CM is to quickly concentrate the detection of the space while ensuring that all global minimizers are always included in the contraction sets, so that the expected logarithmic time efficiency can be observed. This supports that the concept of contractibility could provide a new perspective on what kind of continuous optimization problems can be effectively solved. Of course, since the contraction is performed stably, the computational cost of each modeling is controlled within a fixed amount.

\subsection{Polynomial time contractible}

\begin{defn}
 The problem \eqref{CM:eq:COP} is said to be polynomial time contractible if there exist three suitable parameters $\rho,p,q>0$ such that $f$ satisfies Assumptions \ref{CM:ass:HLFD} and \ref{CM:ass:Ta}.
\end{defn}

Unlike the logarithmic time contractible problems, there is no fixed upper bound for $N^{(k)}$; but it can often be controlled by $C2^{kl}N_{\Omega,f}$. So the CM is still effective because the contraction of $D^{(k)}$ can reduce the computational cost of the model $\mathcal{A}^{(k)}f$.
\begin{thm}\label{CM:thm:PTC}
Suppose the problem \eqref{CM:eq:COP} is polynomial time contractible with parameters $(\rho,p,q)$ and $\{D^{(k)}\}$ is defined by Definition \ref{CM:defn:D} with \eqref{CM:eq:SDMed}, in which, $\chi^{(k)}$ is quasi-uniformly distributed over $D^{(k)}\subset\Omega$ with
\begin{equation*}
  q_{\chi^{(k)}}=\tau'\cdot C^{-\frac{1}{n}} 2^{-\frac{kl+s}{n}}\pi/\rho
  ~~\textrm{and}~~
  h_{D,\chi^{(k)}}=\tau''\cdot C^{-\frac{1}{n}}2^{-\frac{kl+s}{n}}\pi/\rho
  ~~\textrm{for}~~\tau'\leqslant1\leqslant\tau'',
\end{equation*}
and $\mathcal{A}^{(k)}f$ is given by Gaussian kernel interpolant $\mathcal{I}_{\chi^{(k)}}f$ with parameter satisfying \eqref{CM:eq:kpara}. Then, for any $\omega<q$ and $\epsilon>0$, there exist $C>1$ and $K\geqslant1$ such that after $K$ contractions, it holds that the upper bound
\begin{equation*}
 \max_{x\in D^{(K)}}[f(x)-f^*]<\epsilon,
\end{equation*}
with the linear convergence rate
\begin{equation*}
 \max_{x\in D^{(k)}}[f(x)-f^*]<\left(\frac{1+\omega}{1+q}\right)^k
 \max_{x\in\Omega}[f(x)-f^*],
\end{equation*}
the total number of function evaluations $\mathcal{O}\big(N_{\Omega,f} \cdot2^{l\log_{\frac{1+\omega}{1+q}}\epsilon}\big)$, and the total time complexity $\mathcal{O}\big(N_{\Omega,f}^3\cdot8^{l\log_{\frac{1+\omega}{1+q}}\epsilon}\big)$, where $l$ is as in Lemma \ref{CM:lem:TFa}, $N_{\Omega,f}=\mathcal{O} \big(2^s\rho^n/\pi^n\big)$ and $s$ is the unique integer such that
\begin{equation*}
9\left(\frac{p}{1+p}\right)^s\|\hat{f}\|_{L_1\cap L_2}
<\omega\Big(\max_{x\in\Omega}f(x)-f^*\Big)\leqslant9
\left(\frac{p}{1+p}\right)^{s-1}\|\hat{f}\|_{L_1\cap L_2}.
\end{equation*}
\end{thm}
\begin{proof}
From Lemma \ref{CM:lem:TFa}, for any $k\in\mathbb{N}_0$, since $\chi^{(k)}$ is quasi-uniformly distributed w.r.t. a sampling density of $C2^{\frac{kl+s}{n}}\rho^n/\pi^n$, there exist suitable kernel parameters as in \eqref{CM:eq:kpara} such that Gaussian kernel interpolant $\mathcal{I}_{\chi^{(k)}}f$ satisfies the error bound condition with the strong convergence condition. So it follows from the strong convergence Theorem \ref{CM:thm:SConv} that
\begin{equation*}
 \max_{x\in D^{(k)}}[f(x)-f^*]\leqslant
 \left(\frac{1+\omega}{1+q}\right)^k\max_{x\in\Omega}[f(x)-f^*].
\end{equation*}
There is $N_{\Omega,f}=\mathcal{O}(2^s\rho^n/\pi^n)$ such that
\begin{equation*}
N^{(k)}=C\mu(D^{(k)})2^{kl+s}\rho^n/\pi^n\leqslant
C2^{kl}2^s\rho^n/\pi^n\leqslant2^{kl}N_{\Omega,f}.
\end{equation*}
Then, for a fixed accuracy $\epsilon>0$, there exists a $K>0$ such that
\begin{equation*}
 \left(\frac{1+\omega}{1+q}\right)^K(f^{**}-f^*)<\epsilon
 \leqslant\left(\frac{1+\omega}{1+q}\right)^{K-1}(f^{**}-f^*),
\end{equation*}
hence, after $K$ contractions, one gets the approximate solution set $D^{(K)}$ with an error bound
\begin{equation*}
 \max_{x\in D^{(K)}}[f(x)-f^*]\leqslant\epsilon,
\end{equation*}
and the total number of function evaluations is less than
\begin{equation*}
 \sum_{k=0}^{K-1}N^{(k)}\leqslant\sum_{k=0}^{K-1}2^{kl}N_{\Omega,f}
 =\mathcal{O}\left(N_{\Omega,f}\cdot2^{l\log_{\frac{1+\omega}{1+q}}\epsilon}\right),
\end{equation*}
further, since the Gaussian kernel interpolant $\mathcal{I}_{\chi^{(k)}}f$ can be computed by GMRES \citep{Saad1986A_GMRES} in $\mathcal{O}(N^{(k)})^2$ iterations, even if the model is updated every time a sample is added, the complexity of the $k$th contraction still does not exceed $\mathcal{O}(8^{kl}N_{\Omega,f}^3)$, so the total time complexity is less than
\begin{equation*}
 \sum_{k=0}^{K-1}8^{kl}N^3_{\Omega,f}=\mathcal{O}
 \left(N^3_{\Omega,f}\cdot8^{l\log_{\frac{1+\omega}{1+q}}\epsilon}\right),
\end{equation*}
taking a polynomial time for any desired accuracy $\epsilon$.
\end{proof}

Notice that Lemma \ref{CM:lem:TFa} also hold for any $u^{(k)}$ satisfying
\begin{equation}\label{CM:eq:PTCu}
\frac{q}{1+q}\left(u^{(k)}-f^*\right)<u^{(k+1)}-f^*
<\frac{1}{1+q}\left(u^{(k)}-f^*\right),
\end{equation}
where $q<1$ (and at the same time, $q<p$, as mentioned in Remark
\ref{CM:rem:qp}); that is,
\renewcommand{\thelema}{8a}
\begin{lema}\label{CM:lem:TFra}
Under Assumptions \ref{CM:ass:HLFD}, suppose $\{D^{(k)}\}$ is defined by Definition \ref{CM:defn:D} with \eqref{CM:eq:PTCu}, in which, $\chi^{(k)}$ is quasi-uniformly distributed over $D^{(k)}\subset\Omega$ with
\begin{equation*}
  q_{\chi^{(k)}}=\tau'\cdot C^{-\frac{1}{n}} 2^{-\frac{kl+s}{n}}\pi/\rho
  ~~\textrm{and}~~
  h_{D,\chi^{(k)}}=\tau''\cdot C^{-\frac{1}{n}}2^{-\frac{kl+s}{n}}\pi/\rho
  ~~\textrm{for}~~\tau'\leqslant1\leqslant\tau'',
\end{equation*}
where the constant $C$ is as in Lemma \ref{CM:lem:HLFDF&KI}. Then, for all $k\in\mathbb{N}_0$ and $\omega<q<p$, there are unique natural numbers $s$, $l>1$ and kernel parameter $\sigma>2^{\frac{kl+s+1}{n}}\rho$ such that Gaussian kernel interpolant $\mathcal{I}_{\chi^{(k)}} f$ satisfies the error bound condition
\begin{equation*}
 \|\mathcal{I}_{\chi^{(k)}} f-f\|_{L_\infty(D^{(k)})}
 <\omega\left(u^{(k)}-f^*\right)
\end{equation*}
with the strong convergence condition
\begin{equation*}
  u^{(k)}-f^*\leqslant\frac{1}{1+q}\max_{x\in D^{(k)}}[f(x)-f^*].
\end{equation*}
\end{lema}

Thus, an immediate corollary of Theorem \ref{CM:thm:PTC} is:
\begin{cor}\label{CM:cor:PTC}
Suppose there exist $\rho,p,q>0$ such that the problem \eqref{CM:eq:COP} satisfies Assumption \ref{CM:ass:HLFD} and $\{D^{(k)}\}$ is defined by Definition \ref{CM:defn:D} with $u^{(k)}$ satisfying \eqref{CM:eq:PTCu}, in which, $\chi^{(k)}$ is quasi-uniformly distributed over $D^{(k)}\subset\Omega$ with
\begin{equation*}
  q_{\chi^{(k)}}=\tau'\cdot C^{-\frac{1}{n}} 2^{-\frac{kl+s}{n}}\pi/\rho
  ~~\textrm{and}~~
  h_{D,\chi^{(k)}}=\tau''\cdot C^{-\frac{1}{n}}2^{-\frac{kl+s}{n}}\pi/\rho
  ~~\textrm{for}~~\tau'\leqslant1\leqslant\tau'',
\end{equation*}
and $\mathcal{A}^{(k)}f$ is given by Gaussian kernel interpolant $\mathcal{I}_{\chi^{(k)}}f$ with parameter satisfying \eqref{CM:eq:kpara}. Then, for any $\omega<q$, all conclusions of Theorem \ref{CM:thm:PTC} also hold, where $N_{\Omega,f}=\mathcal{O}\big(2^s\rho^n/\pi^n\big)$ and $s$ is the unique integer such that
\begin{equation*}
9\left(\frac{p}{1+p}\right)^s\|\hat{f}\|_{L_1\cap L_2}
<\omega\Big(u^{(0)}-f^*\Big)\leqslant9
\left(\frac{p}{1+p}\right)^{s-1}\|\hat{f}\|_{L_1\cap L_2}.
\end{equation*}
\end{cor}

A typical polynomial time contractible example is illustrated in Figure \ref{CM:fig:PT1}.

\begin{figure}[tbhp]
  \centering
  \subfigure{\includegraphics[width=0.32\textwidth]{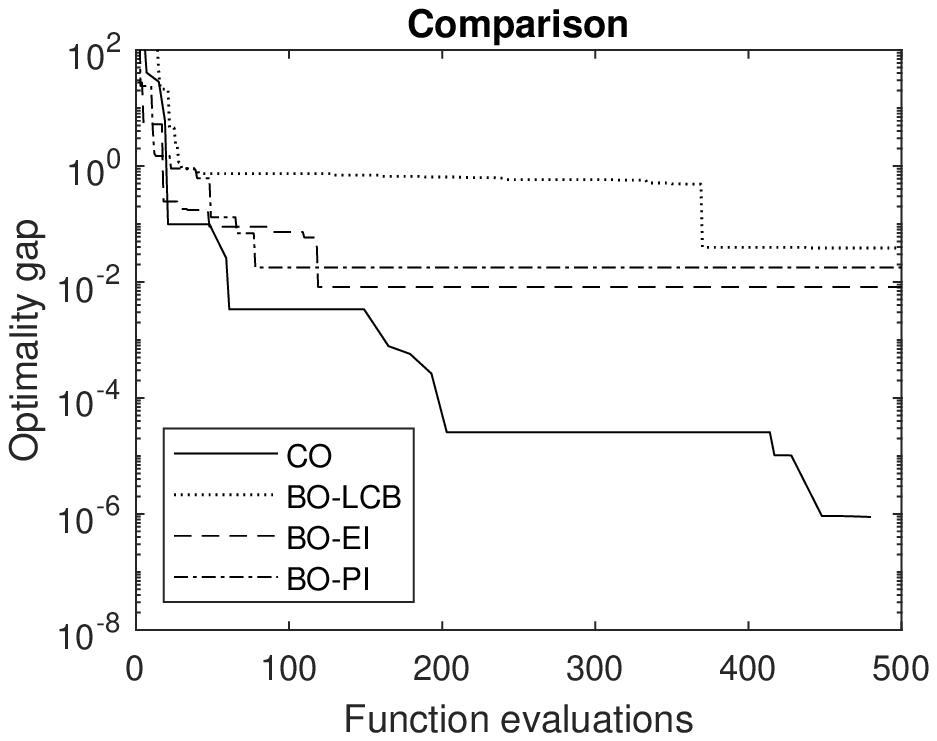}}
  \subfigure{\includegraphics[width=0.32\textwidth]{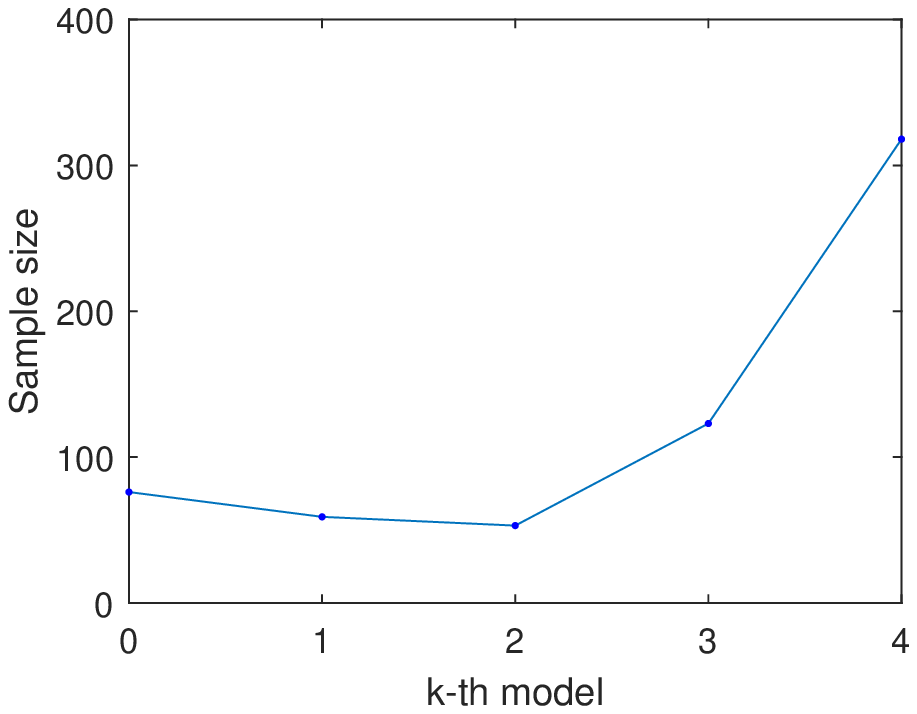}}
  \subfigure{\includegraphics[width=0.32\textwidth]{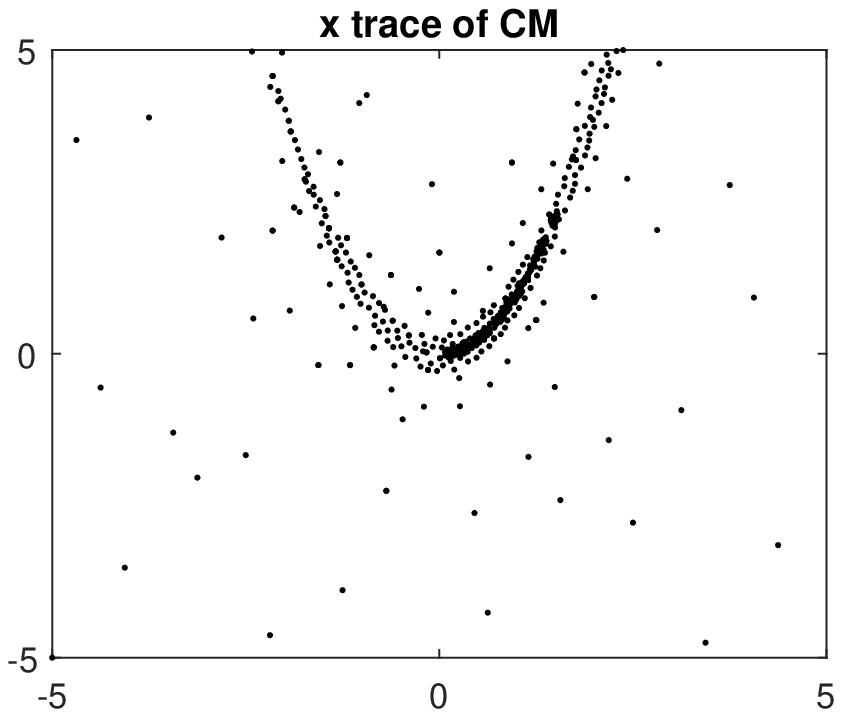}}
  \caption{A polynomial time contractible example: the objective $f(x_1,x_2)=100(x_2-x_1^2)^2+(x_1-1)^2, x_1\in[-5,5], x_2\in[-5,5]$ \citep{RosenbrockH1960M_RosenbrockFunction}. Its global minimum is $0$ located at $(1,1)$, which lies in a long, narrow and parabolic shaped flat valley. For the CM, the number of samples used by each model is increasing, which leads to the cost of polynomial time; however, the contraction strategy can still help reduce the total computational cost. Left: convergence comparison between Algorithm \ref{CM:alg:CM} and three different types of Bayesian optimization (BO). Algorithm \ref{CM:alg:CM} is run with $K=5,m=1,\textrm{minIterInner}=1, \omega=1$, $c_k=\bar{c}=50$ and $t_k=\bar{t}=2$ for every $0\leqslant k\leqslant K-1$. BO is implemented by the MATLAB function \texttt{bayesopt} with default settings. Middle: the sample size used in each model. Right: x trace of CM.}
  \label{CM:fig:PT1}
\end{figure}

\subsection{Noncontractible}

For some functions, the error bound condition cannot be satisfied until the global minima is reached, at this time the CM degenerates into a conventional model-based approach. Although the contractions cannot be applied, the corresponding optimization problem may still be effectively solved if the function is sufficient smooth.

\begin{figure}[tbhp]
  \centering
  \includegraphics[width=0.32\textwidth]{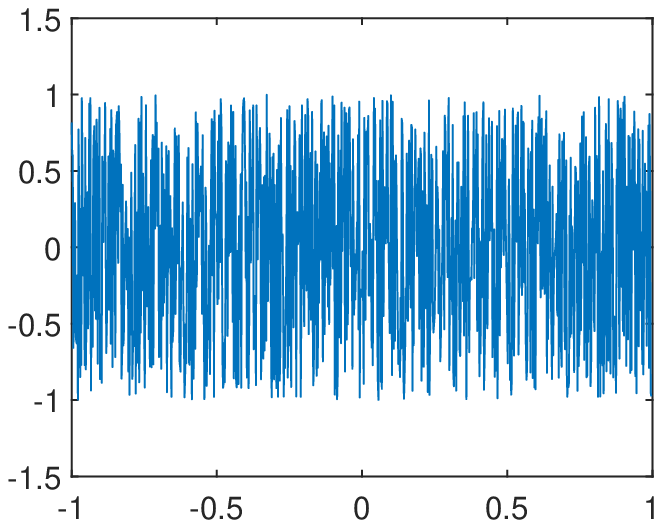}
  \includegraphics[width=0.32\textwidth]{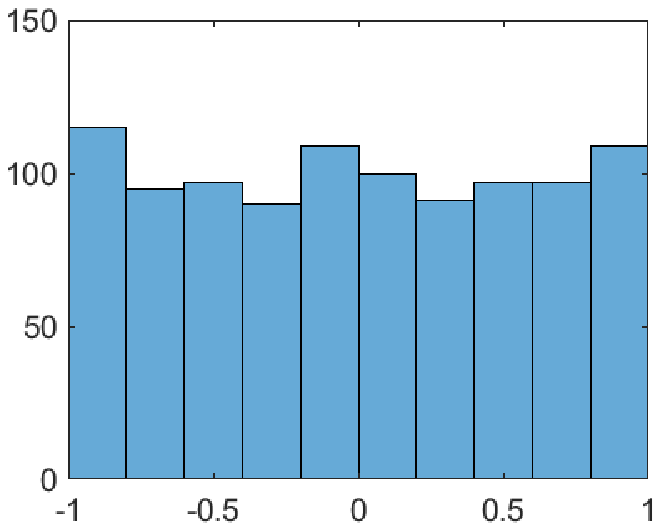}
  \includegraphics[width=0.32\textwidth]{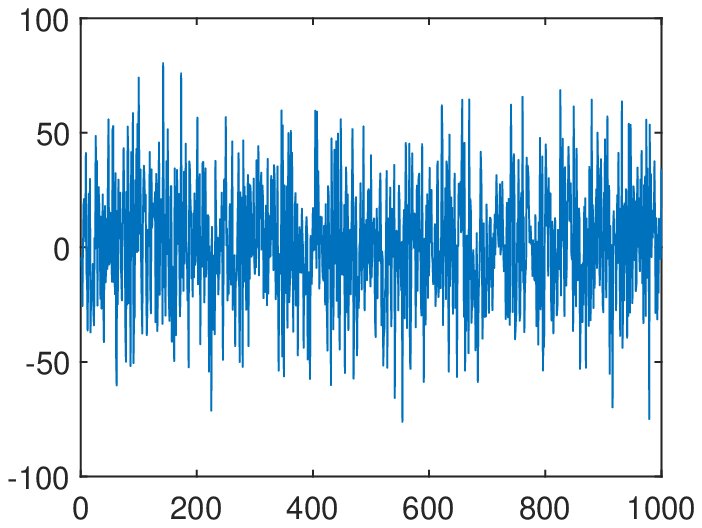}
  \includegraphics[width=0.32\textwidth]{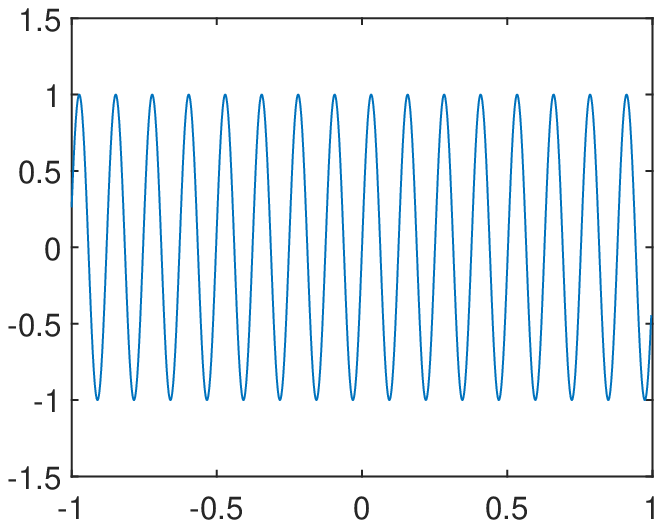}
  \includegraphics[width=0.32\textwidth]{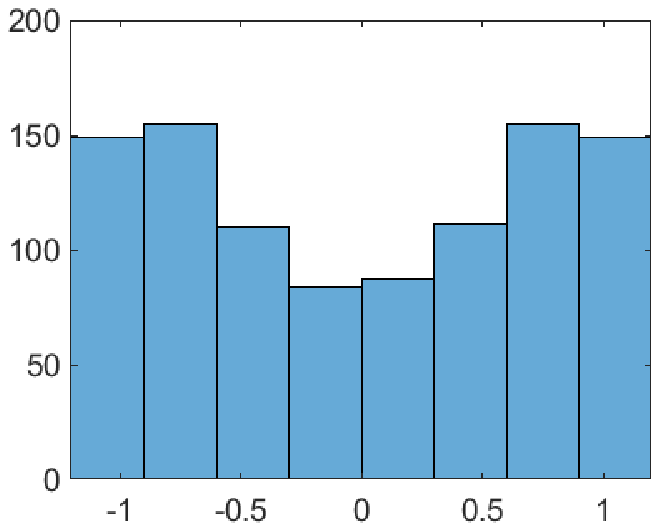}
  \includegraphics[width=0.32\textwidth]{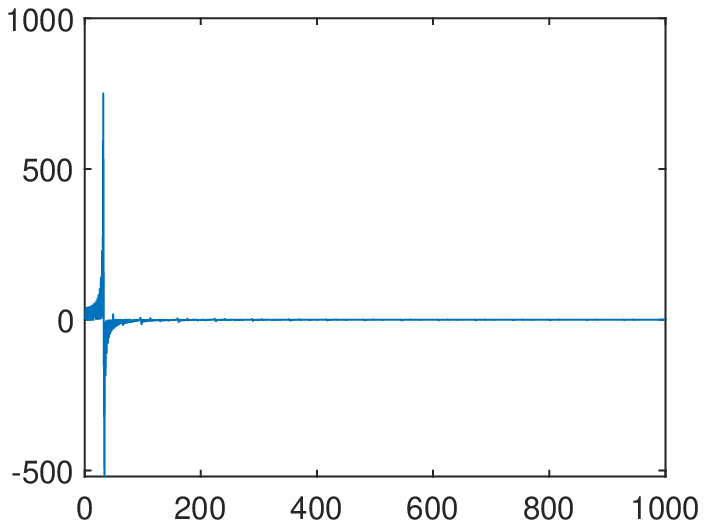}
  \includegraphics[width=0.32\textwidth]{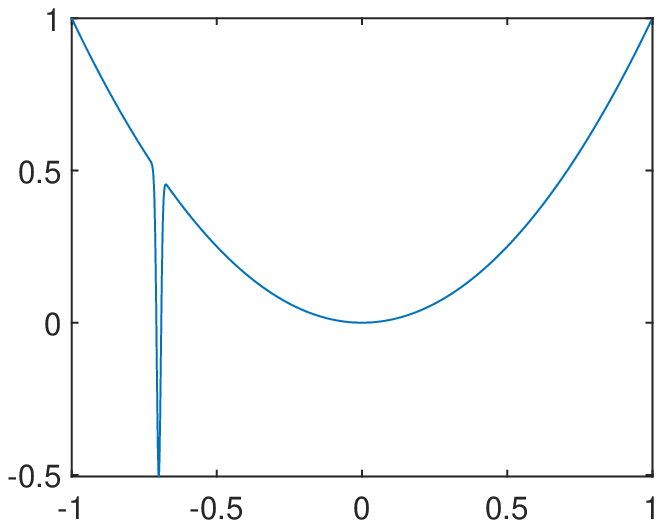}
  \includegraphics[width=0.32\textwidth]{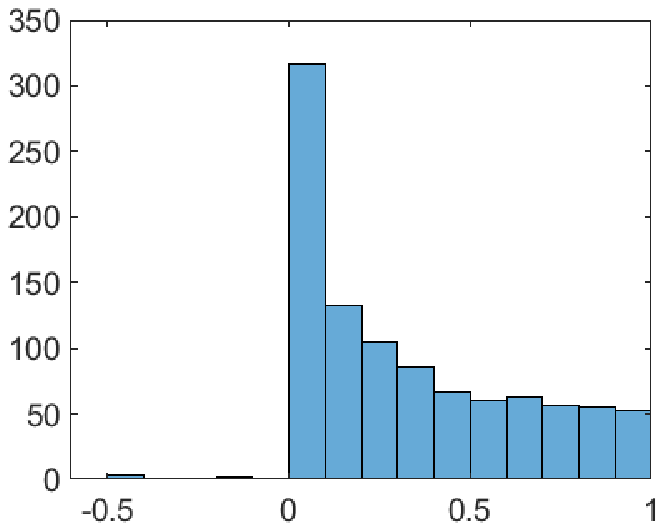}
  \includegraphics[width=0.32\textwidth]{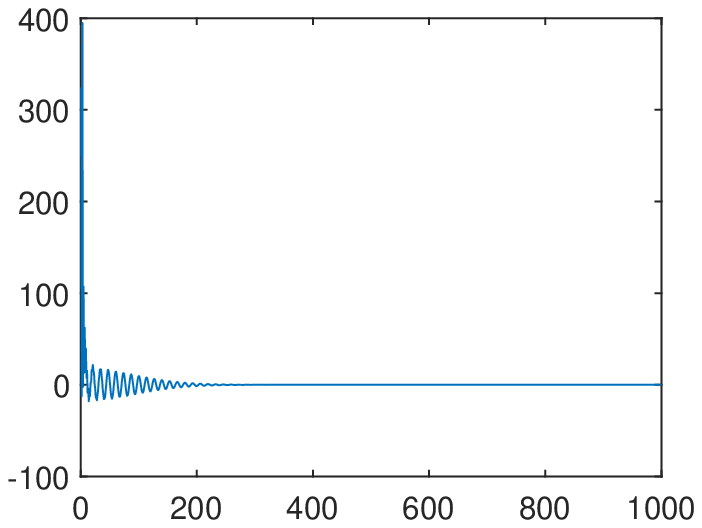}
  \caption{Three examples for noncontractible functions. Left column: a Gaussian noise function $f_1$ that is obviously not a HLFDF and the error bound condition will never be met; a bandlimited function, i.e., $f_2=\sin(50x)$, that does not have the essential multilevel characteristic of HLFDFs, so that the error bound condition cannot be met until the global minima is reached; and a quadratic function with a deep trap, i.e., $f_3(x)=x^2-e^{-10000*(x+0.7)^2}$, that is a HLFDF but does not satisfy the conclusion of Lemma \ref{CM:lem:TF}, at this time, the contractions will still be carried out with a large probability and eventually lead to convergence failure, due to the failure of model error bound estimations. Middle column: histograms of function value distribution for $f_1,f_2$ and $f_3$, respectively. Right column: discrete cosine transforms of $f_1,f_2$ and $f_3$, respectively.}
  \label{CM:fig:non}
\end{figure}

The term ``non-contractible'' is used to describe functions for which the corresponding contraction cannot be consistently guaranteed. This does not always mean that, although it does exist, the algorithms does not perform any contraction throughout the process. See Figure \ref{CM:fig:non} for different non-contractible examples. Some of these functions are not reasonable in the sense of optimization, such as white noise; but it cannot be ruled out that they contain some functions with practical significance.

\begin{defn}
 The problem \eqref{CM:eq:COP} is said to be non-contractible if there are no parameters $\rho,p,q>0$ such that $f$ satisfies Assumptions \ref{CM:ass:HLFD} or \ref{CM:ass:Ta}.
\end{defn}

We will summarize some conclusions based on smoothness for non-contractible problems. Suppose that $\chi$ are quasi-uniformly distributed over $\Omega$ with the sample size $N$, the relevant data values $f_\chi=\{f(\chi_i)\}_{i=1}^N$ and $\mathcal{I}_\chi f$ interpolates $f$ on $\chi$. For a sample set $\chi$ over $\Omega$ and any $\epsilon>0$, there exists $C_\epsilon>0$ such that the fill distance satisfies
\begin{equation}\label{CM:eq:hNG}
P\left(h_{\Omega,\chi}>C_\epsilon(\log N/N)^{1/n}\right)
=\mathcal{O}(N^{-\epsilon}),
\end{equation}
or, $h_{\Omega,\chi}=\mathcal{O}(N^{(\gamma-1)/n})$, where $0<\gamma\ll1$; see Lemma 12 of \citet{BullA2011T_ConvergenceRatesOfBO}.

According to \eqref{CM:eq:hNG} and Lemma 3.9 in \citet{WendlandH2005A_SobolevBoundsOnFunctionsWithScatteredZeros}, there is a bandlimited interpolant $\mathcal{I}_\chi f$ such that
\begin{equation}\label{CM:eq:sSmooth}
\|f-\mathcal{I}_\chi f\|_{L_\infty(\Omega)}\leqslant Ch_{\Omega,\chi}^{s-n/2}
\|f\|_{C^s(\Omega)}=\mathcal{O}\left(N^{-(s-n/2)(1-\gamma)/n}\right),
\end{equation}
which implies the following conclusion:
\begin{thm}\label{CM:thm:sSmooth}
If $f\in C^s(\Omega)$ with $s>n/2$ on a bounded domain $\Omega\subset\mathbb{R}^n$, then there exists a bandlimited interpolant $\mathcal{I}_\chi f$ such that for any $\epsilon>0$, it holds that
\begin{equation*}
 \Big|\min_{x\in\Omega}\mathcal{I}_\chi f(x)-f^*\Big|<\epsilon
\end{equation*}
and the time complexity of $\mathcal{I}_\chi f$ is
\begin{equation*}
 \mathcal{O}(N^2)=\mathcal{O}\left(2^{\frac{2n}{(s-n/2)(1-\gamma)}
 \log_{\frac{1}{2}}\epsilon}\right),
\end{equation*}
where $\chi$, $\gamma$ and $N$ are consistent with \eqref{CM:eq:hNG}, and $\|f\|_{C^s(\Omega)}=\sum_{|\alpha| \leqslant s}\sup_{x\in\Omega}|D^\alpha f|$.
\end{thm}
\begin{rem}
It is worth noting that, compared with the above conclusion, for the complexity given in Theorem \ref{CM:thm:PTC} and Corollary \ref{CM:cor:PTC}, i.e., $\mathcal{O}\big(N_{\Omega,f}^3\cdot 8^{l\log_{\frac{1+\omega}{1+q}}\epsilon}\big)$, the dimension-related part $N_{\Omega,f}=\mathcal{O} \big(2^s\rho^n/\pi^n\big)$ is actually independent of the polynomial growth term $8^{l\log_{\frac{1+\omega}{1+q}}\epsilon}$, where $l$ is as in Lemma 6a and independent of the dimension $n$.
\end{rem}
\begin{proof}
Assume that $s^*=\arg\min_{x\in\Omega}\mathcal{I}_\chi f(x)$ and $x^*=\arg\min_{x\in\Omega}f(x)$. For any
 $x\in\Omega$, if $|\mathcal{I}_\chi f(x)-f(x)|<\epsilon$, then
\begin{equation*}
  \mathcal{I}_\chi f(s^*)>f(s^*)-\epsilon>f(x^*)-\epsilon,
\end{equation*}
and
\begin{equation*}
  \mathcal{I}_\chi f(s^*)<\mathcal{I}_\chi f(x^*)<f(x^*)+\epsilon,
\end{equation*}
that is, $|\min_{x\in\Omega}\mathcal{I}_\chi f(x)-f^*|<\epsilon$, together with \eqref{CM:eq:sSmooth}, the desired result is obtained.
\end{proof}

This is similar to but slightly weaker than that obtained for Bayesian Optimization in \citet{BullA2011T_ConvergenceRatesOfBO}. For such a function, the problem \eqref{CM:eq:COP} can be solved in polynomial time; in fact, if $f$ is H\"{o}lder continuous, the problem \eqref{CM:eq:COP} can also be completed in polynomial time.

$f$ is called $\alpha$-H\"{o}lder continuous function if there exist $C>0$ and $0<\alpha\leqslant1$ such that
\begin{equation}\label{CM:eq:HC}
|f(x')-f(x'')|<C\|x'-x''\|_2^\alpha,~~\forall x',x''\in\Omega.
\end{equation}
The Lipschitz continuous function is obviously a special case of H\"{o}lder continuous function. For a $\alpha$-H\"{o}lder continuous function, then there exists a nearest-neighbor interpolant $\mathcal{I}_\chi f$, which is closely related to the Voronoi diagram of $\chi$ \citep{AurenhammerF1991R_VoronoiDiagrams}, such that for any $x\in\Omega$, it holds that
\begin{equation*}
|f(x)-\mathcal{I}_\chi f(x)|\leqslant Ch_{\Omega,\chi}^\alpha
=\mathcal{O}\left(N^{-\alpha(1-\gamma)/n}\right).
\end{equation*}
Similar to Theorem \ref{CM:thm:sSmooth}, we have the following theorem:
\begin{thm}\label{CM:thm:HC}
If $f$ satisfies a $\alpha$-H\"{o}lder condition on a bounded domain $\Omega\subset\mathbb{R}^n$, then there exists a nearest-neighbor interpolant $\mathcal{I}_\chi f$ such that for any $\epsilon>0$, it holds that
\begin{equation*}
 \Big|\min_{x\in\Omega}\mathcal{I}_\chi f(x)-f^*\Big|<\epsilon
\end{equation*}
and the time complexity of $\mathcal{I}_\chi f$ is
\begin{equation*}
 \mathcal{O}(N^a)=\mathcal{O}
 \left(2^{\frac{na}{\alpha(1-\gamma)}\log_{\frac{1}{2}}\epsilon}\right),
\end{equation*}
where $\chi,\gamma$ and $\alpha$ are consistent with \eqref{CM:eq:HC}, the time cost of $\mathcal{I}_\chi f$ is assumed to be $\mathcal{O}(N^a)$.
\end{thm}

Actually, the accuracy bound of GS on an $n$-cube is reduced by a factor of $1/2$ when the number of function evaluations is increased by $2^n$ times, thus, the complexity bound of GS is
\begin{equation*}
  \mathcal{O}\left(2^{\frac{n}{\alpha}\log_{\frac{1}{2}}\epsilon}\right),
\end{equation*}
which also implies the piecewise constant interpolation. However, GS has no advantage in efficiency in most cases, because it does not use known information at all.

Based on the above conclusions, when the objective function is sufficiently smooth, its total computational complexity is much better than the GS and random search methods. In addition, as we emphasized, even if the objective is nonsmooth, the CM also enjoys good efficiency for any contractible case.

Further, if $f$ does not satisfy any H\"{o}lder condition, then the algorithm may not be done in polynomial time, for example, if $|f(x')-f(x'')|<C\log(1+\|x'-x''\|)$, then the time complexity of a nearest-neighbor interpolant $\mathcal{I}_\chi f$ will be
\begin{equation*}
\mathcal{O}\left((e^\epsilon-1)^{-\frac{na}{1-\gamma}}\right),
\end{equation*}
when the absolute error bound $\sup_{x\in\Omega}|\mathcal{I}_\chi f(x)-f(x)|$ is less than any given $\epsilon>0$.

\section{Quasi-uniform samples and spatial discretization}
\label{CM:s5}

The CM needs to expand the original quasi-uniform sample set into a larger one in a certain bounded domain. In this section, we will discuss the theoretical basis for a method of expansion, and it is also related to the spatial discretization for approximating the H\"{o}lder functions as well as the functions in a certain RKHS.

\subsection{Generation of quasi-uniform sequence}

The basic idea is selecting the point farthest from the known point set in a given domain.
\begin{defn}\label{CM:defn:QU}
Suppose $D\subset\mathbb{R}^n$ is a bounded domain and $\chi=\{\chi_1\}$ is a one point set in $D$. A point sequence $\chi=\{\chi_i\}_{i=1}^\infty$ is defined recursively by $\chi=\chi\cup\chi_{new}$ and
\begin{equation}\label{CM:eq:QU}
\chi_{new}=\arg\max_{x\in D}\left(\min_{\chi_i\in\chi}\|x-\chi_i\|_2\right).
\end{equation}
\end{defn}

\begin{thm}\label{CM:thm:QU}
If $\chi$ is constructed as Definition \ref{CM:defn:QU}, then there is $N\in\mathbb{N}$ such that when the size of $\chi$ is larger than $N$, $\chi$ is quasi-uniform over $D$ with uniformity constant $4$, i.e.,
\begin{equation*}
  \frac{1}{4}q_\chi\leqslant h_{D,\chi}\leqslant4q_\chi,
\end{equation*}
where $h_{D,\chi}$ is defined as \eqref{CM:eq:h} and $q_\chi$ is defined as \eqref{CM:eq:q}.
\end{thm}
\begin{proof}
For a given radius $r>0$, define the union of open balls associated with $\chi$ as
\begin{equation*}
  \mathcal{B}(\chi,r):=\bigcup_{\chi_i\in\chi}B(\chi_i,r),
\end{equation*}
where $B(\chi_i,r)=\{x\in\mathbb{R}^n:\|x-\chi_i\|_2<r\}$ is an open ball of radius $r$ centered at $\chi_i$.

Assume, without loss of generality, that $\max_{x,y\in D}\|x-y\|_2=1$. First, we take $r=\frac{1}{2}$. Originally, $\chi$ is a single point set in $D$ and $\max_{x,y\in D}\|x-y\|_2=1$, then $\mathcal{B}(\chi,\frac{1}{2})$ cannot cover $D$, therefore, according to the rules for adding points \eqref{CM:eq:QU}, before $\mathcal{B}(\chi,\frac{1}{2})$ covers $D$, there will be no points falling into $\mathcal{B}(\chi,\frac{1}{2})$, and at this time, we have
\begin{equation*}
  \frac{1}{4}\leqslant q_\chi\leqslant1~~\textrm{and}~~
  \frac{1}{2}\leqslant h_{D,\chi}\leqslant1,
\end{equation*}
that is, $\frac{1}{8}\leqslant\frac{1}{2}q_\chi\leqslant h_{D,\chi}\leqslant4q_\chi\leqslant4$.
Assume that after a new point is added, $\mathcal{B}(\chi,\frac{1}{2})$ just covers $D$, i.e., at this time there is no point $t\in D$ such that $\min_{x\in\chi}\|t-x\|_2>\frac{1}{2}$, so from now on, we have $h_{D,\chi}\leqslant\frac{1}{2}$, and further, $q_\chi\leqslant\frac{1}{2}$ if one more point is added.

Now we take $r=\frac{1}{4}$ and add one point making $q_\chi\leqslant\frac{1}{2}$. Before $\mathcal{B}(\chi,\frac{1}{4})$ covers $D$, there will be no points falling into $\mathcal{B}(\chi,\frac{1}{4})$, and at this time, we have
\begin{equation*}
  \frac{1}{8}\leqslant q_\chi\leqslant\frac{1}{2}~~\textrm{and}~~
  \frac{1}{8}\leqslant h_{D,\chi}\leqslant\frac{1}{2},
\end{equation*}
that is, $\frac{1}{32}\leqslant\frac{1}{4}q_\chi\leqslant h_{D,\chi}\leqslant4q_\chi\leqslant2$.

Generally, assume that after a certain new point is added, $\mathcal{B}(\chi,\frac{1}{2^k})$ just covers $D$, i.e., there is no point $t\in D$ such that $\min_{x\in\chi}\|t-x\|_2>\frac{1}{2^k}$, so from now on, we have $h_{D,\chi}\leqslant\frac{1}{2^k}$, and further, $q_\chi\leqslant\frac{1}{2^k}$ if one more point is added. Similarly, we take $r=\frac{1}{2^{k+1}}$ and add one point making $q_\chi\leqslant\frac{1}{2^k}$. Before $\mathcal{B}(\chi,\frac{1}{2^{k+1}})$ covers $D$, it follows that
\begin{equation*}
  \frac{1}{2^{k+2}}\leqslant q_\chi\leqslant\frac{1}{2^k}~~\textrm{and}~~
  \frac{1}{2^{k+2}}\leqslant h_{D,\chi}\leqslant\frac{1}{2^k},
\end{equation*}
that is, $\frac{1}{2^{k+4}}\leqslant\frac{1}{4}q_\chi\leqslant h_{D,\chi}\leqslant4q_\chi\leqslant\frac{1}{2^{k-2}}$ for all $k\in\mathbb{N}$, and the proof is complete.
\end{proof}

Our proof relies on the open balls whose radius is gradually reduced by a factor of $\frac{1}{2}$. And it can be also seen from the proof that both $q_\chi$ and $h_{D,\chi}$ will not change too much before each coverage is completed. Obviously, the above conclusion still holds for any factor $0<\frac{1}{\kappa}<1$, or equivalently, $1<\kappa<\infty$,. So an immediate corollary of Theorem \ref{CM:thm:QU} is:
\begin{cor}\label{CM:cor:QU}
Under the assumptions of Theorem \ref{CM:thm:QU}, for any $1<\kappa<\infty$, there is an integer $N_\kappa$ such that when the size of $\chi$ is larger than $N_\kappa$, $\chi$ is quasi-uniform over $\Omega$ with uniformity constant $2\kappa$, i.e.,
\begin{equation*}
  \frac{1}{2\kappa}q_\chi\leqslant h_{D,\chi}\leqslant2\kappa q_\chi,
\end{equation*}
where $h_{D,\chi}$ and $q_\chi$ are defined as \eqref{CM:eq:h} and \eqref{CM:eq:q}, respectively.
\end{cor}

Thus, when $\kappa$ tends to $1$, it follows that $\frac{1}{2}q_\chi\leqslant h_{D,\chi}\leqslant2q_\chi$, asymptotically. This shows that every sequence constructed by Definition \ref{CM:defn:QU} has an excellent uniformity, especially, independent of the dimensionality of the space in which the domain is located. Of course, it is difficult to solve \eqref{CM:eq:QU} in an continuous domain, hence, a practical algorithm will be further considered in Subsection \ref{CM:sec:sampling}.

\subsection{Spatial discretization for H\"{o}lder continuous functions}

For a $\alpha$-H\"{o}lder function $f$, if $\mathcal{I}_\chi f$ is the nearest-neighbor interpolant w.r.t. a sample set $\chi\subset D$, then we have
\begin{equation*}
|f(x)-\mathcal{I}_\chi f(x)|\leqslant Ch_{D,\chi}^\alpha.
\end{equation*}
As we mentioned above, when the size of $\chi$ is fixed, the smaller uniformity constant $\tau$ means a better fill distance $h_{D,\chi}$, and also means a better approximation accuracy. Notice that for $n$-dimensional grid nodes, the uniformity constant $\tau=\sqrt{n}$ depending on $n$, therefore, together with Corollary \ref{CM:cor:QU}, if $n>4$ and the sample size is the same, then the interpolation error corresponding to the quasi-uniform nodes constructed by Definition \ref{CM:defn:QU} will be better than that of the grid nodes.

Correspondingly, the GS method for global optimization is suitable for H\"{o}lder continuous functions in space less than $4$ dimensions. When the dimension of the space is greater than $4$, even the direct search method based on a quasi-uniform sequence constructed by Definition \ref{CM:defn:QU} will be better than the GS method.

It is also worth noting that, for the $(\alpha,\infty)$-H\"{o}lder continuous functions, that is, there exist $C>0$ and $0<\alpha\leqslant1$ such that
\begin{equation}\label{CM:eq:Hinf}
|f(x')-f(x'')|<C\|x'-x''\|_\infty^\alpha,~~\forall x',x''\in D,
\end{equation}
the standard grid nodes are appropriate for their approximation or global search. Because in the sense of infinite norm, the fill distance $h_{D,\chi,\infty}$ and the separation distance $q_{\chi,\infty}$ are often equal for a standard grid nodes $\chi$ over a cube $D$, where
\begin{equation*}
  h_{D,\chi,\infty}:=\max_{x\in D}\min_{\chi_i\in\chi}\|x-\chi_i\|_\infty
  ~~\textrm{and}~~
  q_{\chi,\infty}:=\frac{1}{2}\min_{\chi_i\neq\chi_j}\|\chi_i-\chi_j\|_\infty.
\end{equation*}
However, since the infinite norm is equal to the largest absolute value of the components, the above condition \eqref{CM:eq:Hinf} even excludes the superposition effect between different variables of a multivariate function, therefore, such functions are generally rare in practice, especially in high-dimensional spaces.

\section{Practical algorithms}
\label{CM:s6}

On the basis of the proposed framework, three essential parts for establishing a practical algorithm are (i) the method for modeling a given data pairs $(\chi^{(k)},f_{\chi^{(k)}})$ on $D^{(k)}$, (ii) the sampling strategy for further generating quasi-uniform samples $\chi^{(k)}$ over $D^{(k)}$ according to some known interior points $\chi^{(k-1)}\cap D^{(k)}$, and (iii) the method for estimating the error bounds for the model, i.e., $\max_{x\in D^{(k)}}|\mathcal{A}^{(k)}f(x)-f(x)|$. In the following, we shall first discuss these parts separately, then combine them into a whole and provide insights into its behaviors by establishing two high probability bounds for convergence rate and complexity.

\subsection{Kernel-based modeling}

According to Lemmas \ref{CM:lem:TF} and \ref{CM:lem:TFa}, there exists a Gaussian kernel interpolant $\mathcal{I}_{\chi^{(k)}}f$ on $D^{(k)}$ satisfying the expected error condition, so it is naturally to adopt kernel interpolantions for modeling; however, it is often recommended to use a more stable regression model in practice. As an extension of Gaussian kernel interpolation, we mainly consider the Gaussian process (GP) regression \citep{RasmussenC2006M_GaussianProcesses} with a covariance function
\begin{equation*}
  \psi(x,x')=\phi(x,x')+\sigma_n^2\omega_{x,x'}
  =\sigma_f^2e^{-\sigma^2\|x-x'\|_2^2}+\sigma_n^2\omega_{x,x'}
\end{equation*}
and a prior $\mathcal{GP}(0,\psi(x,x'))$ over a \emph{bandlimited function} $g$, where $g$ interpolates $f$ on $\chi^{(k)}$, $\theta=(\sigma_f^2,\sigma^2,\sigma_n^2)$ is a vector containing all the hyperparameters, and $\omega_{x,x'}$ is the Kronecker delta which is one if and only if $x=x'$ and zero otherwise. In other words, we actually regard this bandlimited $g$ as a function composed of an element from a Gaussian reproducing-kernel Hilbert space (RKHS) plus independent and identically distributed Gaussian noises. Notice that Lemma \ref{CM:lem:HLFDFinterpolating} guarantees the existence of $g$ if $f$ is an HLFDF. Moreover, we do not need to assume that $f$ has a bounded RKHS norm, but consider $f$ as the sum of a bandlimited function and an acceptable (possibly non-differentiable) residual on each $D^{(k)}$.

Then, for a given dataset $(\chi,f_\chi)=\big\{(\chi_i,f_{\chi_i})\big\}_{i=1}^N$ over $D$, the predictive distribution for any $x\in D$ becomes
\begin{equation}\label{CM:eq:GP}
  g(x)|x,\chi,f_\chi~\thicksim~\mathcal{N}
  \Big(\mathbb{E}[g(x)],\mathbb{V}[g(x)]\Big),
\end{equation}
where
\begin{align*}
  \mathbb{E}[g(x)]=&~\Phi(x,\chi)\Psi^{-1}(\chi,\chi)f_\chi, \\
  \mathbb{V}[g(x)]=&~\phi(x,x)-\Phi(x,\chi)\Psi^{-1}(\chi,\chi)\Phi(\chi,x),
\end{align*}
here, $\Phi(x,\chi)=(\phi(x,\chi_1),\cdots,\phi(x,\chi_N))$, $\Phi(\chi,x)=(\Phi(x,\chi))^\mathrm{T}$, $f_\chi=(f(\chi_1),\cdots,f(\chi_N))^\mathrm{T}$, and $\Psi(\chi,\chi)= [\psi(\chi_i,\chi_j)]_{1\leqslant i\leqslant N,1\leqslant j\leqslant N}$. Therefore, the regression model $\mathcal{A}_\chi f$ is given by the mean prediction $\mathbb{E}[g(x)]$ and the logarithmic marginal likelihood can be expressed as
\begin{equation*}
  \log p(f_\chi|\chi,\theta)=
  -\frac{1}{2}f_\chi^\mathrm{T}\Psi^{-1}(\chi,\chi)f_\chi -\frac{1}{2}\log|\Psi(\chi,\chi)|-\frac{N}{2}\log 2\pi,
\end{equation*}
then the hyperparameters can be automatically set by maximizing the marginal likelihood in time $\mathcal{O}(N^3)$, see \citep{RasmussenC2006M_GaussianProcesses} for details. So we have
\begin{lem}\label{CM:lem:GP}
For a dataset $(\chi,f_\chi)=\big\{(\chi_i,f_{\chi_i})\big\}_{i=1}^N$ over $D$, the computational complexity of modeling $f$ by the GP regression, i.e., \eqref{CM:eq:GP}, can be bounded by $\mathcal{O}(N^3)$ and the computational cost of calling this model is $\mathcal{O}(N)$.
\end{lem}

Since the adaptation of hyperparameters is adopted, $\mathcal{A}_\chi f$ degenerates to the Gaussian kernel interpolant $\mathcal{I}_{\chi}f$ under the conditions of Lemmas \ref{CM:lem:TF} and \ref{CM:lem:TFa} because of the adequacy of samples and the existence of bandlimited interpolant. On the other hand, if the sample set $\chi$ does not meet the assumption of sampling density in Lemmas \ref{CM:lem:TF} and \ref{CM:lem:TFa}, $\mathcal{A}_\chi f$ is more stable than $\mathcal{I}_{\chi}f$ due to the smoothing effect.

\subsection{Quasi-uniform sampling based on existing interior points}
\label{CM:sec:sampling}

Our sampling strategy is divided into two steps. First, according to the existing interior point set $\chi$, we use a reflected random walk to generate the candidate point set $\mathcal{T}$ in $D\subset\Omega$. Then, by using a method similar to \eqref{CM:eq:QU}, we recursively select a supplementary sample set $\chi'$ from $\mathcal{T}$, where the size of $\chi'$ should be much smaller than $\mathcal{T}$.

Suppose that $D:=\{x\in D':c(x)\leqslant0\}$, where $c$ is a continuous function on $\Omega$ and, clearly, $c(x)=\mathcal{A}_{\chi^{(k)}}f(x)-u^{(k)}$ when $D'=D^{(k)}$. We begin with the reflected random walk (RRW) in $D$ from a point $\chi_i$. For a given variance $\sigma^2$ and a fixed reduction factor $a\in(0,1)$, the RRW $\{Z_t\}_{t\in\mathbb{N}_0}$ is defined as $Z_0=\chi_i$ and
\begin{align}\label{CM:eq:RRW}
  Z_{t+1}=\left\{\begin{array}{cl}
    Z_t+\sigma W_t, & \textrm{if}~~c(Z_t+\sigma W_t)\leqslant0, \\
    Z_t-(-a)^j\sigma W_t+a^{j+1}\sigma W'_t, & \textrm{if}~~c(Z_t+\sigma W_t)>0;
  \end{array}\right.
\end{align}
where $j$ is the smallest integer such that $c(Z_t-(-a)^j\sigma W_t+a^{j+1}\sigma W'_t)\leqslant0$, $W_t$ and $W'_t$ are $n$-dimensional standard normal random variables. When $D=\mathbb{R}^n$, $Z_t$ has a normal distribution with mean $\chi_i$ and variance $\sigma^2t$; when $D$ is a closed domain, the reflection ensures that $Z_t\in D$ for all $t\in\mathbb{N}_0$. Moreover, our reflection is also suitable for general nonconvex domains, even domains with a H\"{o}lder boundary, see Figure \ref{CM:fig:S1a} for an example.

Suppose that $\chi=\{\chi_i\}_{i=1}^N$ and define
\begin{equation}\label{CM:eq:d}
  d_\chi:=\max_{1\leqslant i\leqslant N}\min_{j\neq i}\|\chi_i-\chi_j\|_2,
\end{equation}
then we describe the algorithm for generating a candidate point set $\mathcal{T}$ from $\chi$ as follows:
\begin{algorithm}
\caption{Generation of a candidate point set with size $n_cN$}
\label{CM:alg:CPS}
\begin{algorithmic}[1]
\STATE{Preset parameters $n_c>0$ and $T\in\mathbb{N}$.}
\STATE{$\mathcal{T}:=\varnothing$.}
\FOR{$i=1,2,\cdots,N$}
\STATE{Generate $n_c$ points by the RRW from $\chi_i$ with $\sigma=\frac{d_\chi}{\sqrt{T}}$ and $T$ steps, denoted by $\mathcal{T}_i$.}
\STATE{Set $\mathcal{T}=\mathcal{T}\cup\mathcal{T}_i$.}
\ENDFOR
\end{algorithmic}
\end{algorithm}

Note that the candidate set $\mathcal{T}(n_c,T)$ contains $n_cN$ points and each $\mathcal{T}_i$ contains $n_c$ points that are independent and identically distributed with mean $\chi_i$ and variance $d_\chi^2$. And it is clear that $d_\chi\geqslant2q_\chi$, so according to the three-sigma rule of thumb, if $\chi$ is quasi-uniform over $D$ with uniformity constant less than $6$, $\mathcal{T}(n_c,T)$ basically covers the entire domain $D$.

On the basis of $\mathcal{T}$, new points can be recursively added as follows:
\begin{align}\label{CM:eq:SP}
 \left\{\begin{array}{ll}
  \chi_{new}&\!\!=~\arg\max_{x\in\mathcal{T}}
  \Big(\min_{\chi_i\in\chi}\|x-\chi_i\|_2\Big),\\
  \mathcal{T}&\!\!=~\mathcal{T}-\chi_{new},
 \end{array}\right.
\end{align}
where $\mathcal{T}$ is a discretization for $D$ and the number of supplementary samples should be much smaller than $n_cN$ to ensure uniformity. The difference between the ideal strategy \eqref{CM:eq:QU} and the actual strategy \eqref{CM:eq:SP} is that the continuous domain $D$ is replaced by a discretization $\mathcal{T}$. Therefore, the corresponding quasi-uniformity can be guaranteed when the density of $\mathcal{T}$ is significantly greater than that of $\chi$.

More specifically, during the process where $\mathcal{B}(\chi,\frac{1}{2^k})$ covers $D$ but $\mathcal{B}(\chi,\frac{1}{2^{k+1}})$ has not yet covered $D$, if there are always points of $\mathcal{T}$ in $D-\mathcal{B}(\chi,\frac{1}{2^{k+1}})$, then before $\mathcal{B}(\chi,\frac{1}{2^{k+1}})$ covers $D$, it follows that
\begin{equation*}
  \frac{1}{2^{k+2}}\leqslant q_\chi\leqslant\frac{1}{2^k}~~\textrm{and}~~
  \frac{1}{2^{k+2}}\leqslant h_{D,\chi}\leqslant\frac{1}{2^k},
\end{equation*}
that is, the results of Theorem \ref{CM:thm:QU} and Corollary \ref{CM:cor:QU} also hold for the strategy \eqref{CM:eq:SP}. In other words, as long as there is any point of $\mathcal{T}$ in $D-\mathcal{B}(\chi,q_\chi)$, one can use this strategy to add at least one point without changing the uniformity constant. Hence, by using the above step recursively, one can generate quasi-uniform samples in any nonconvex domain $D$. See Figures \ref{CM:fig:S1} and \ref{CM:fig:S1a} for examples to illustrate how the strategy is performed.

At a first glance, the computational complexity of generating $s(\ll n_cN)$ new points by the strategy \eqref{CM:eq:SP} is $\mathcal{O}(sn_cN^2)$; actually, this process can also be done in $\mathcal{O}(sn_cN)$ time with $\mathcal{O}(n_cN)$ extra storage. For an implementation, one can refer to our MATLAB code, and the corresponding time cost is given as follows:
\begin{lem}\label{CM:lem:SP}
For a fixed existing interior set $\chi\in D$ with size $N$, parameters $0\leqslant n_c<\infty$ and $0\leqslant T<\infty$, the computational complexity of generating $s(\ll n_cN)$ new points by the strategy \eqref{CM:eq:SP} can be bounded by $\mathcal{O}(N^2)$.
\end{lem}

\begin{figure}[tbhp]
  \centering
  \subfigure{\includegraphics[width=0.32\textwidth]{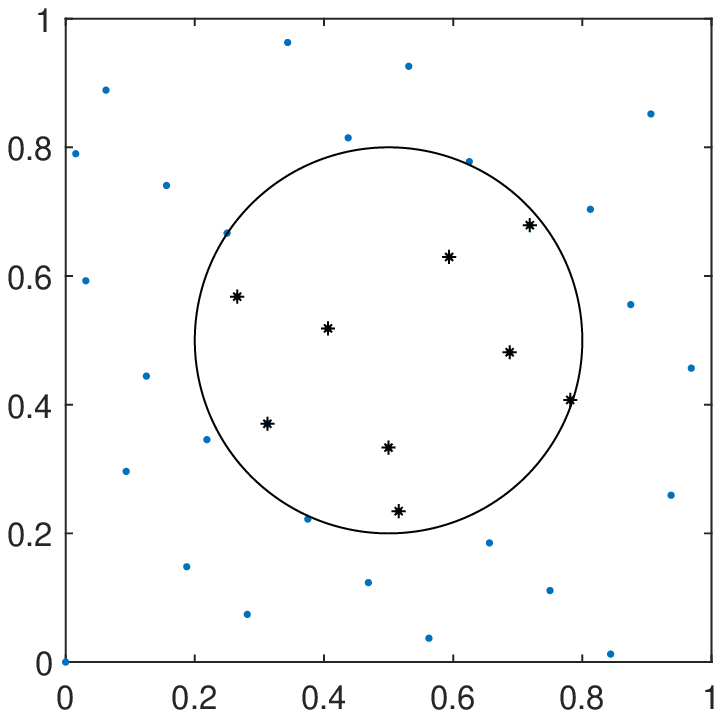}}
  \subfigure{\includegraphics[width=0.32\textwidth]{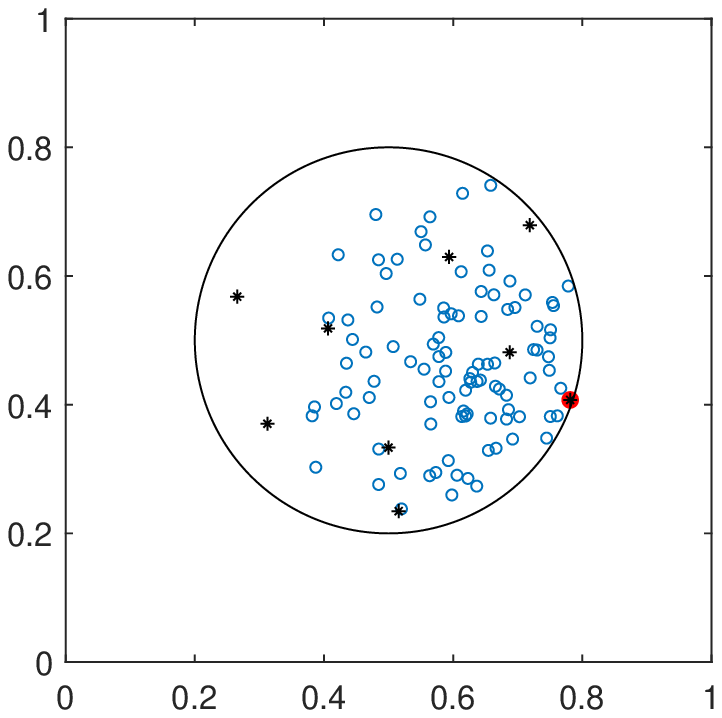}}
  \subfigure{\includegraphics[width=0.32\textwidth]{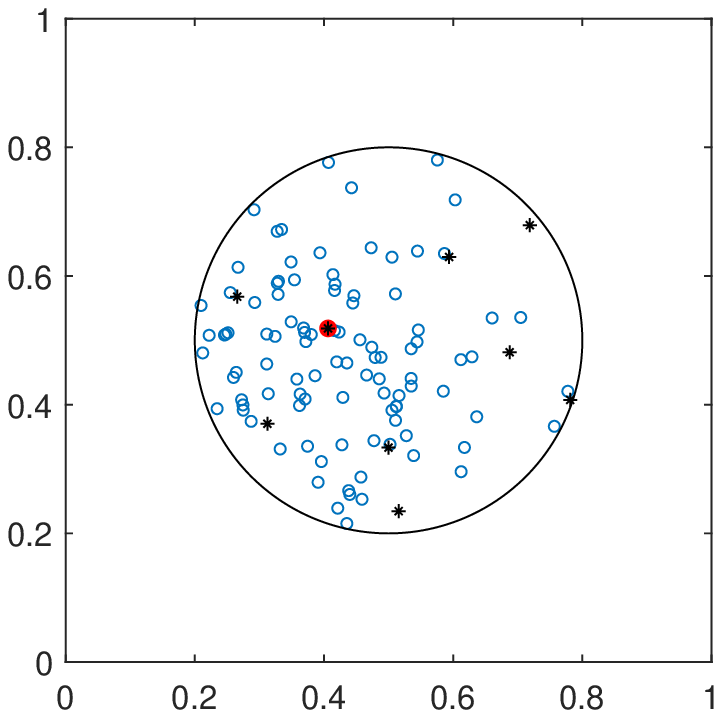}}\\
  \subfigure{\includegraphics[width=0.32\textwidth]{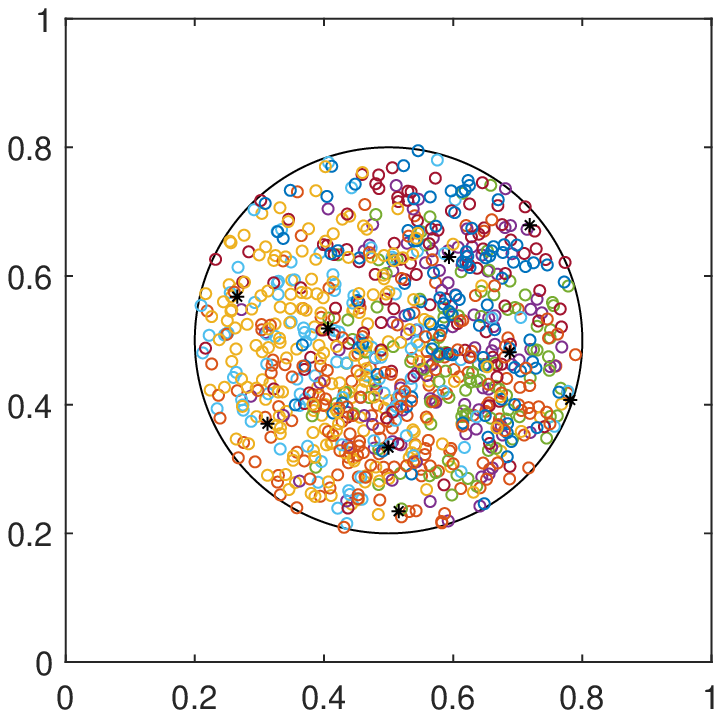}}
  \subfigure{\includegraphics[width=0.32\textwidth]{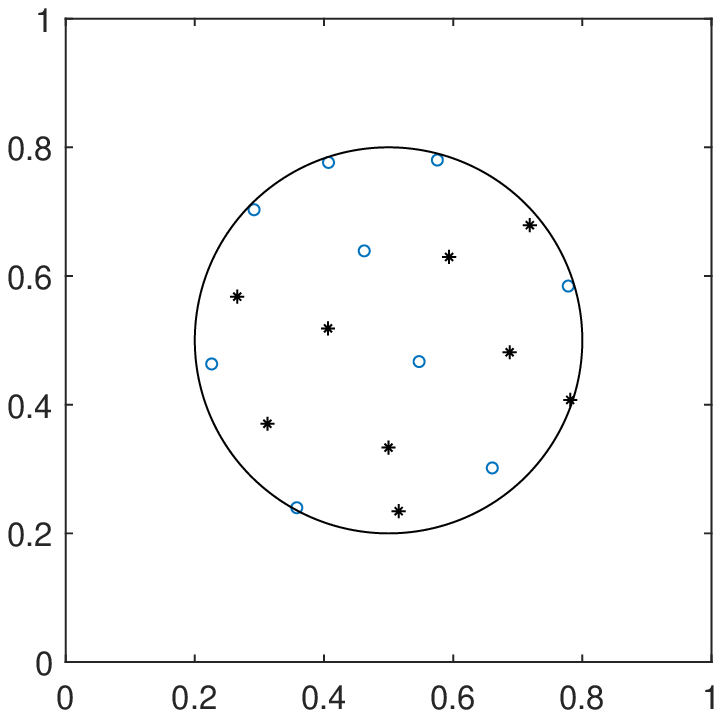}}
  \subfigure{\includegraphics[width=0.32\textwidth]{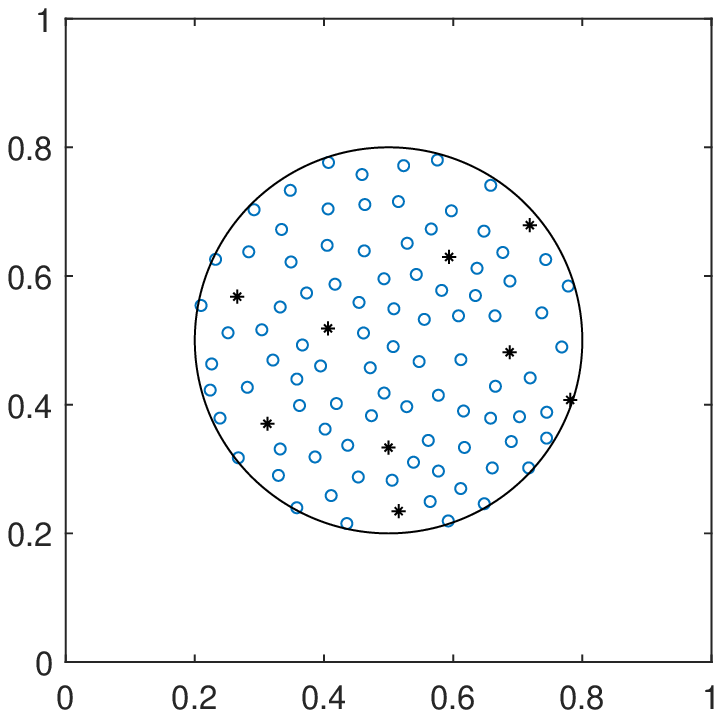}}
  \caption{A $2$-dimensional illustration of the performance of the recursive algorithm given in \eqref{CM:eq:SP}. Upper left: the domain $D=\{(x_1,x_2)\in [0,1]^2:g(x_1,x_2):=(x_1-0.5)^2+(x_2-0.5)^2-0.3^2\leqslant0\}$ is shown as the interior of the circle, the quasi-uniform samples in $[0,1]^2$, that is the first $35$ points of the $2$-dimensional halton sequence, are visible as dots in blue, and $9$ points falling into $D$, denoted by $\chi$, are shown as asterisks in black. Upper middle: $100$ points generated by the RRW are visible as circledots in blue, where the starting point $\chi_4$ is shown as dot in red, the step number $T=10$ and $\sigma=\frac{d_\chi}{\sqrt{10}}$. Upper right: $100$ points generated by the RRW are visible as circledots in blue, where the starting point $\chi_5$ is shown as dot in red, $T=10$ and $\sigma=\frac{d_\chi}{\sqrt{10}}$. Lower left: the candidate set $\mathcal{T}(100,10)$, which is a union of sample sets generated by the RRW from each $\chi_i$, is visible as circledots. Lower middle: $9$ samples added recursively from $\mathcal{T}(100,10)$ is visible as circledots in blue. Lower right: $90$ samples added recursively from $\mathcal{T}(100,10)$ is visible as circledots in blue and the expected quasi-uniformity can be easily observed. Here $n_c$ is selected as $100$ to generate $90$ new points.}
  \label{CM:fig:S1}
\end{figure}

\begin{figure}[tbhp]
  \centering
  \subfigure{\includegraphics[width=0.32\textwidth]{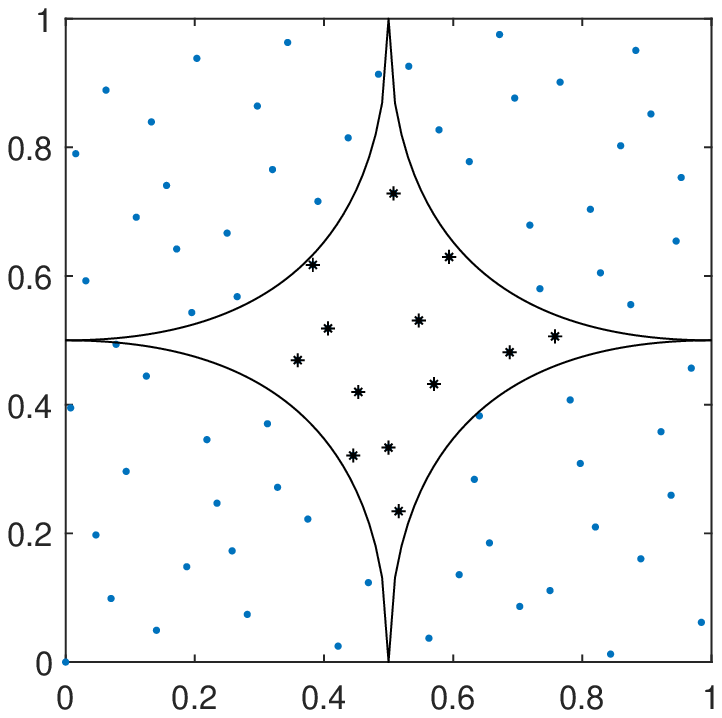}}
  \subfigure{\includegraphics[width=0.32\textwidth]{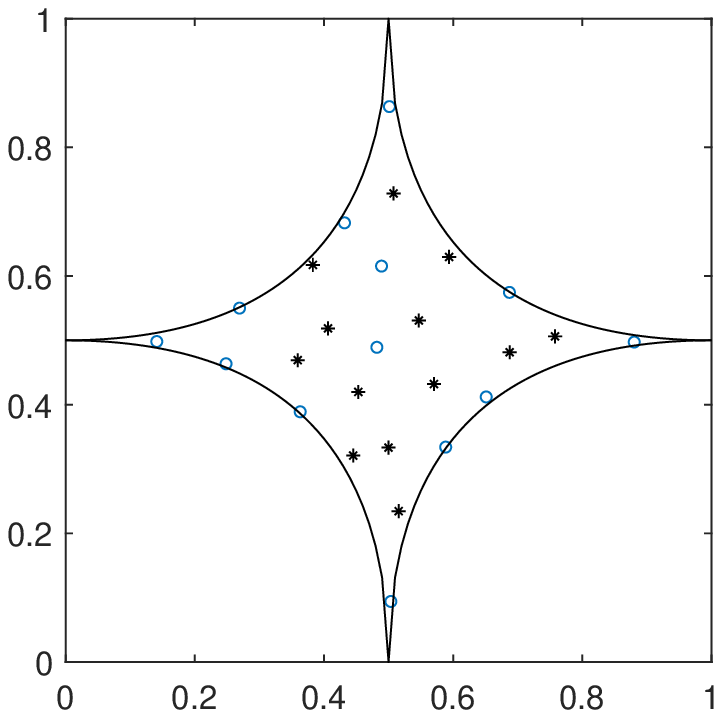}}
  \subfigure{\includegraphics[width=0.32\textwidth]{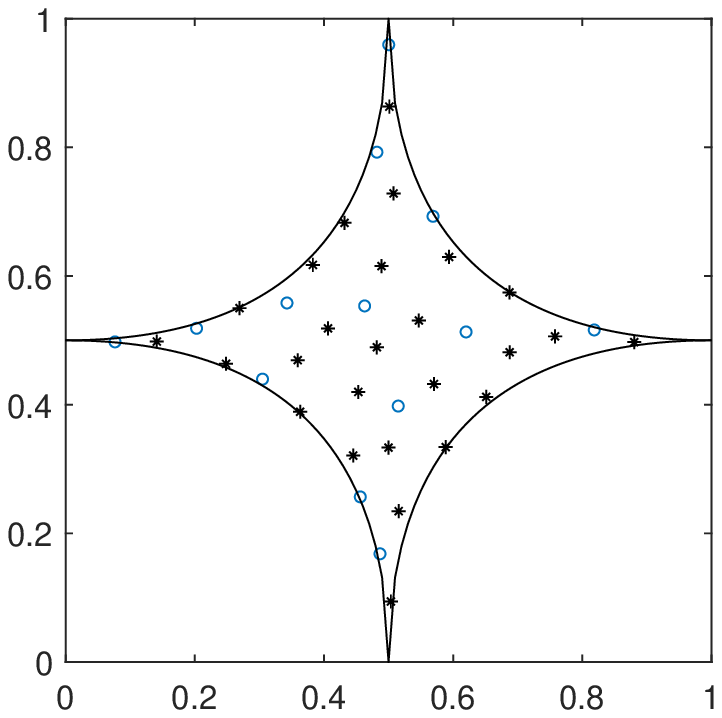}}
  \caption{A $2$-dimensional illustration of the performance of the recursive algorithm given in \eqref{CM:eq:SP} for a nonconvex four pointed star domain with a H\"{o}lder boundary. Left: the domain $D=\{(x_1,x_2)\in [0,1]^2:g(x_1,x_2):=|x_1-0.5|^{0.5}+|x_2-0.5|^{0.5}-0.5^{0.5}\leqslant0\}$ is shown as the interior of the four pointed star, the quasi-uniform samples in $[0,1]^2$, that is the first $80$ points of the $2$-dimensional halton sequence, are visible as dots in blue, and $13$ samples falling into $D$, denoted by $\chi$, are shown as asterisks in black. Middle: $13$ new points are generated based on the $13$ existing interior points given in the left plot, where the existing points are shown as asterisks in black, the new points are visible as circledots in blue, and the candidate set is $\mathcal{T}(100,10)$. Right: $13$ new points are generated based on the $26$ existing interior points given in the middle plot, where the existing points are shown as asterisks in black, the new points are visible as circledots in blue, and the candidate set is $\mathcal{T}(100,10)$. The expected quasi-uniformity can also be observed although the given domain is really tough even in the nonconvex cases, especially on the sharp corners of the four pointed star. In fact, notice that each model we established is smooth due to the Gaussian kernel, so the boundary of $D$ must be also smooth inside $D'$. Therefore, even if there is a nonconvex situation, it is often difficult to find a nonsmooth boundary, for example, a Lipschitz continuous boundary or even a H\"{o}lder continuous boundary, unless at the junction of the boundary of $D$ and the boundary of $D'$. This illustrates that our sampling method constituted by the RRW \eqref{CM:eq:RRW} and the strategy \eqref{CM:eq:SP} is sufficient to deal with all possible situations in the running of a contraction algorithm. In addition, it also provides a theoretically guaranteed and effective approach to generate quasi-uniform samples on general nonconvex domains with an appropriate uniformity constant, even in high-dimensional spaces.}
  \label{CM:fig:S1a}
\end{figure}

\subsection{Estimating model error}
\label{CM:s6:EME}

For the contraction algorithms, the sampling procedure on each subdomain $D^{(k)}$ embodies the exploration of unknown information and the approximate model $\mathcal{A}^{(k)}f$ reflects the exploitation of prior information, while the error bound condition is one of the key factors for ensuring a sufficient exploration and issuing a judgment on the conversion of exploration to exploitation. Although some error estimates of the approximate model can be established in the deterministic sense, they are often not satisfactory enough because of the existence of unknown constants. Hence, from a practical point of view, we will use statistical methods to estimate these errors in a sense of probability.

Since the GP regression is used in modeling, a direct idea for estimating the model error bound on each $D^{(k)}$ seems to maximize the variance in \eqref{CM:eq:GP} for establishing a confidence bound. However, our practical experience shows that, such confidence bounds are usually too conservative to effectively reduce the number of function evaluations on each $D^{(k)}$. One of the potential reasons for this issue is that the assumption, that the noise has Gaussian distributions, is not always satisfied in reality.

Actually, we can estimate the model error bound under a mild assumption. For a given sample set $\chi^{(k)}=\{\chi_i^{(k)}\}_{i=1}^N$ on $D^{(k)}$ and a GP regression prediction $\mathcal{A}^{(k)}f$ w.r.t. the data set $(\chi^{(k)},f_{\chi^{(k)}})$ with known hyperparameter vector $\theta^{(k)}$, assume that the error
\begin{equation*}
  \varepsilon^{(k)}=\mathcal{A}^{(k)}f-f
\end{equation*}
are independent, identically distributed random variables with mean $\mu_\varepsilon^{(k)}$ and standard deviation $\sigma_\varepsilon^{(k)}$ on $\mathbb{R}$, then $\varepsilon^{(k)}$ can be estimated by the $s$-fold cross validation (CV) without adding any new function evaluation \citep{GeisserS1975M_CV,ArlotS2010R_CV}. Specifically, in an $s$-fold CV procedure, the data set is randomly partitioned into $s$ folds, that is, $s$ subsets of equal size. For each fold, the regression model is built on the union of the other folds with the known hyperparameter vector $\theta^{(k)}$, then the error of its output is estimated using the fold. Thus, the mean $\mu_\varepsilon^{(k)}$ and standard deviation $\sigma_\varepsilon^{(k)}$ of the true error can be estimated by all these errors. Specially, the case $s=N$ is called leave-one-out (LOO).

Finally, according to the Chebyshev's inequality, a probabilistic bound for the residual $\varepsilon^{(k)}$ on $D^{(k)}$ can be given as
\begin{equation}\label{CM:eq:PBEI}
P\left(|\varepsilon^{(k)}-\mu_\varepsilon^{(k)}|\leqslant t_k
\sigma_\varepsilon^{(k)}\right)
\geqslant1-\frac{1}{t_k^2}\geqslant e^{-1.2t_k^{-2}},
\end{equation}
where the last inequality holds for every confidence parameter $t_k\geqslant2$. And the model error bound can be regarded as
\begin{equation}\label{CM:eq:PBE}
  \max_{x\in D^{(k)}}|\varepsilon^{(k)}(x)|
  \leqslant\varepsilon^{(k)}_{\textrm{CV}}
  =|\mu_\varepsilon^{(k)}|+t_k\sigma_\varepsilon^{(k)}
\end{equation}
with probability at least $1-t_k^{-2}$. Notice that there is no need to make any prior assumption about the distribution of $\varepsilon^{(k)}$. Obviously, if $\varepsilon^{(k)}$ further follows an independent, identically distributed Gaussian distribution, then the bound above can be rewritten as
\begin{equation*}
P\left(|\varepsilon^{(k)}-\mu_\varepsilon^{(k)}|\leqslant t_k
\sigma_\varepsilon^{(k)}\right)=2\varPhi(t_k)-1,
\end{equation*}
where $\varPhi(t)$ is the one-dimensional standard normal density function; in particular, for $t_k=3$ it follows that the three-sigma rule of thumb $P(|\varepsilon^{(k)}- \mu_\varepsilon^{(k)}|\leqslant3\sigma_\varepsilon^{(k)})=0.9973$.

Since the known hyperparameter vector is used in the modeling of the CV procedure, this process can be completed in $\mathcal{O}(sN)^2$ time, and the corresponding time cost for estimating the model error bound is given as follows:
\begin{lem}\label{CM:lem:PBE}
For a fixed sample set $\chi^{(k)}=\{\chi_i^{(k)}\}_{i=1}^N$ on $D^{(k)}$, let $\mathcal{A}^{(k)}f$ is a GP regression prediction w.r.t. the data set $(\chi^{(k)},f_{\chi^{(k)}})$ with known hyperparameter vector $\theta^{(k)}$, then the computational complexity of estimating the model error bound by the $s$-fold CV procedure can be bounded by $\mathcal{O}(sN^2)$.
\end{lem}

Obviously, even when using the LOO procedure, the corresponding computational cost is $\mathcal{O}(N^3)$. In our MATLAB code, a $10$-fold CV is applied for the error estimate.

\subsection{Algorithms with high probability bounds}

Now we describe the CM as Algorithm \ref{CM:alg:CM}, where $K$ is the total number of contractions, $m$ determines the number of samples added in each detection, $\textrm{minIterInner}$ determines the minimum number of detections per contraction, $\omega$ determines the error bound constant, $\{c_k\}_{k=0}^{K-1}$ is the percentage sequence and $\{t_k\}_{k=0}^{K-1}$ is the confidence parameter sequence.

\begin{algorithm}
\caption{Contraction Method}
\label{CM:alg:CM}
\begin{algorithmic}[1]
\STATE{Preset $K,m,\textrm{minIterInner}\in\mathbb{N}$, $\omega\in(0,1]$, $\{c_k\}_{k=1}^K$ and $\{t_k\}_{k=1}^K$.}
\STATE{Initialize $D^{(0)}:=\Omega\subset\mathbb{R}^n$, $\chi=\varnothing$, and $N_\textrm{total}=0$.}
\FOR{$k=0,1,\cdots,K-1$}
\STATE{Set $\textrm{iterInner}=0$.}
\WHILE{$1$}
\STATE{Update $\textrm{iterInner}=\textrm{iterInner}+1$.}
\STATE{Add $m$ points $X$ (i.e., $X\in\mathbb{R}^{m\times n}$) on $D^{(k)}$ from $\chi$ by the RRW based strategy.}
\STATE{Evaluate $Y=f(X)$ and update $(\chi,f_\chi)=(\chi,f_\chi)\cup(X,Y)$.}
\STATE{Build a GP regression model $\mathcal{A}f$ w.r.t. $(\chi,f_\chi)$.}
\STATE{Estimate $\mu_\varepsilon$ and $\sigma_\varepsilon$ of the error of $\mathcal{A}f$ by the CV procedure.}
\STATE{Find $a^*=\arg\min_{x\in D^{(k)}}\mathcal{A}f(x)$ and update $(\chi,f_\chi)=(\chi,f_\chi)\cup(a^*,f(a^*))$.}
\STATE{Update $u^{(k)}=\textrm{prctile}(f_\chi,c_k)$, $N_\textrm{total}=N_\textrm{total}+m+1$ and $f_{\textrm{best}}^*=f^*_\chi$.}
\IF{$|\mu_\varepsilon|+t_k\sigma_\varepsilon\leqslant\omega(u^{(k)}-f^*_\chi)$ and $\textrm{iterInner}\geqslant \textrm{minIterInner}$}
\STATE{Define $D^{(k+1)}:=\{x\in D^{(k)}:\mathcal{A}f(x)\leqslant u^{(k)}\}$.}
\STATE{Update $\chi=\chi\cap D^{(k+1)}$ and the type parameter.}
\STATE{Break the \textbf{while} loop.}
\ENDIF
\ENDWHILE
\ENDFOR
\end{algorithmic}
\end{algorithm}

The following conclusion establishes the conditions for the logarithmic time complexity of Algorithm \ref{CM:alg:CM} with high probability.
\begin{thm}[Logarithmic time contractible, high probability]\label{CM:thm:LTC2}
Suppose there exist three constants $\rho,p,q>0$ such that the problem \eqref{CM:eq:COP} satisfies Assumption \ref{CM:ass:HLFD} and Algorithm \ref{CM:alg:CM} is run with a sequence $\{u^{(k)}\}$ and parameters $\omega<q$, $t_k\geqslant2$ so that $\sum_{k=0}^{K-1}t_k^{-2}\leqslant\frac{5}{6}\log\frac{1}{1-\delta}$, $\mu(D^{(k+1)})\leqslant\frac{1}{2}\mu(D^{(k)})$ and
\begin{equation*}
\frac{p}{1+p}\left(u^{(k)}-f^*\right)<u^{(k+1)}-f^*
<\frac{1}{1+q}\left(u^{(k)}-f^*\right),~~\textrm{where}~~pq<1.
\end{equation*}
Then, after $K$ contractions, the upper bound
\begin{equation*}
 \max_{x\in D^{(K)}}[f(x)-f^*]<\left(\frac{1+\omega}{1+q}\right)^K
 \max_{x\in\Omega}[f(x)-f^*]
\end{equation*}
holds with probability at least $1-\delta$, there is a fixed $N_{\Omega,f}\in\mathbb{N}$ such that the number of function evaluations per contraction does not exceed $N_{\Omega,f}$, the total number of function evaluations does not exceed $\mathcal{O}\big(KN_{\Omega,f}\big)$, and the total time complexity does not exceed $\mathcal{O}\big(KN_{\Omega,f}^4\big)$.
\end{thm}
\begin{rem}\label{CM:rem:N}
It is worth noting that, as shown in Theorem \ref{CM:thm:LTC}, $N_{\Omega,f}= \mathcal{O}\big(2^{s+1}\rho^n/\pi^n\big)$, where $s$ is the unique integer such that
\begin{equation*}
9\left(\frac{p}{1+p}\right)^s\|\hat{f}\|_{L_1\cap L_2}
<\omega\Big(\max_{x\in\Omega}f(x)-f^*\Big)\leqslant9
\left(\frac{p}{1+p}\right)^{s-1}\|\hat{f}\|_{L_1\cap L_2}.
\end{equation*}
Thus, when $\rho\leqslant\pi$, $N_{\Omega,f}$ is independent of the dimension $n$, or in other words, this type of HLFDFs has a good approximation property that does not depend on dimensionality; when $\rho>\pi$, $N_{\Omega,f}$ depends exponentially on the dimension $n$, but this is the same old story since the curse of dimensionality in high-dimensional problems exists in a general sense. Of course, with the aid of conditions such as smoothness, this dependence could be further reduced or even released.
\end{rem}
\begin{proof}
Due to the adaptability of hyperparameters, if $N^{(k)}$ is large enough, then the GP regression $\mathcal{A}_{\chi^{(k)}}f$ will degenerate to the corresponding interpolant $\mathcal{I}_{\chi^{(k)}}f$ because of the existence of Gaussian kernel interpolant, i.e., Lemmas \ref{CM:lem:HLFDFinterpolating} and \ref{CM:lem:BLS&RKHS}. Let $t^{**}=\max_{0\leqslant k\leqslant K-1}t_k$, then there is a $C_t>0$ such that
\begin{equation*}
  \frac{1}{C_t}\max|\varepsilon^{(k)}|\leqslant|\mu_\varepsilon|
  +t^{**}\sigma_\varepsilon\leqslant C_t\max|\varepsilon^{(k)}|,
\end{equation*}
and then, according to Lemma \ref{CM:lem:TFr} and $\mu(\Omega)\leqslant1$, there is $N_{\Omega,f}=\mathcal{O}(2^s\rho^n/\pi^n)$ such that
\begin{equation*}
N^{(k)}=C_tC\mu(D^{(k)})2^{k+s}\rho^n/\pi^n\leqslant
C_tC2^s\rho^n/\pi^n\leqslant N_{\Omega,f}
\end{equation*}
and
\begin{equation*}
  \mathcal{A}_{\chi^{(k)}}f=\mathcal{I}_{\chi^{(k)}}f,
\end{equation*}
where $C$ is as in Theorem \ref{CM:thm:LTC}. From Lemma \ref{CM:lem:TFr}, for any $k=0,1,\cdots,K-1$, since $\chi^{(k)}$ is quasi-uniformly distributed w.r.t. a sampling density of $C_tC2^{k+s}\rho^n/\pi^n$, the model $\mathcal{A}_{\chi^{(k)}}f$ satisfies the error bound condition
\begin{equation*}
\|\mathcal{A}_{\chi^{(k)}}f-f\|_{L_\infty(D^{(k)})}<\omega\big(u^{(k)}-f^*\big)
\end{equation*}
with the strong convergence condition
\begin{equation*}
  u^{(k)}-f^*\leqslant\frac{1}{1+q}\max_{x\in D^{(k)}}[f(x)-f^*].
\end{equation*}
Since the estimation of the model error bound, i.e.,  $|\mu_\varepsilon|+t_k\sigma_\varepsilon$, holds with probability at least
\begin{equation*}
  1-t_k^{-2}\geqslant e^{-1.2t_k^{-2}},~~\textrm{for all}~~k\geqslant2,
\end{equation*}
it follows from Theorem \ref{CM:thm:SConv} that the error upper bound
\begin{equation*}
 \max_{x\in D^{(k+1)}}[f(x)-f^*]\leqslant
 \left(\frac{1+\omega}{1+q}\right)\max_{x\in D^{(k)}}[f(x)-f^*]
\end{equation*}
also holds with probability at least $e^{-1.2t_k^{-2}}$.

Thus, when $\sum_{k=0}^{K-1}t_k^{-2}\leqslant\frac{5}{6}\log\frac{1}{1-\delta}$, after $K$ contractions, the upper bound
\begin{equation*}
 \max_{x\in D^{(K)}}[f(x)-f^*]<\left(\frac{1+\omega}{1+q}\right)^K
 \max_{x\in\Omega}[f(x)-f^*]
\end{equation*}
holds with probability at least
\begin{equation*}
  \prod_{k=0}^{K-1}e^{-1.2t_k^{-2}}
  =e^{-1.2\sum_{k=0}^{K-1}t_k^{-2}}\geqslant1-\delta,
\end{equation*}
as claimed.

Finally, according to Lemmas \ref{CM:lem:GP}, \ref{CM:lem:SP} and \ref{CM:lem:PBE}, even when $m=1$, that is, after adding two points (including the point that minimizes the current model), the model needs to be updated, the complexity required for each contraction also does not exceed $\mathcal{O}\big(N_{\Omega,f}^4\big)$, in other words, the total time complexity of all $K$ contractions does not exceed $\mathcal{O}\big(KN_{\Omega,f}^4\big)$.
\end{proof}

The following conclusion establishes the conditions for the polynomial time complexity of Algorithm \ref{CM:alg:CM} with high probability.
\begin{thm}[Polynomial time contractible, high probability]\label{CM:thm:PTC2}
~Suppose there exist three constants $\rho,p,q>0$ such that (i) the problem \eqref{CM:eq:COP} satisfies Assumption \ref{CM:ass:HLFD} and (ii) Algorithm \ref{CM:alg:CM} is run with a sequence $\{u^{(k)}\}$ and relevant parameters $\omega<q$, $t_k\geqslant2$ so that $\sum_{k=0}^{K-1}t_k^{-2}\leqslant\frac{5}{6}\log\frac{1}{1-\delta}$ and
\begin{equation}\label{CM:eq:PTCu2}
\frac{q}{1+q}\left(u^{(k)}-f^*\right)<u^{(k+1)}-f^*
<\frac{1}{1+q}\left(u^{(k)}-f^*\right),~~\textrm{where}~~p<q<1.
\end{equation}
Then, after $K$ contractions, the upper bound
\begin{equation*}
 \max_{x\in D^{(K)}}[f(x)-f^*]<\left(\frac{1+\omega}{1+q}\right)^K
 \max_{x\in\Omega}[f(x)-f^*]
\end{equation*}
holds with probability at least $1-\delta$, and there exists a fixed $N_{\Omega,f}\in\mathbb{N}$ such that the number of function evaluations used by each model does not exceed $2^{kl}N_{\Omega,f}$, that is, the total number of function evaluations does not exceed $\mathcal{O}\big(\frac{2^{Kl}-1}{2^l-1}N_{\Omega,f}\big)$ and the total time complexity does not exceed $\mathcal{O}\big(\frac{2^{4Kl}-1}{2^{4l}-1} N_{\Omega,f}^4\big)$, where $l$ is the unique integer such that
\begin{equation*}
  \left(\frac{p}{1+p}\right)^l\leqslant\frac{q}{1+q}
  <\left(\frac{p}{1+p}\right)^{l-1}.
\end{equation*}
\end{thm}
\begin{rem}
See Remark \ref{CM:rem:N} for the dependence of $N_{\Omega,f}$ on the dimension $n$.
\end{rem}
\begin{rem}
Regarding the restriction on $u^{(k)}$, i.e., \eqref{CM:eq:PTCu2}, Proposition \ref{CM:prop:TF} actually establishes its generality for all nonconstant continuous functions in the sense of probability.
\end{rem}
\begin{proof}
This proof is similar to Theorem \ref{CM:thm:LTC2}. First, let $t^{**}=\max_{0\leqslant k\leqslant K-1}t_k$, then there is a $C_t>0$ such that
\begin{equation*}
  \frac{1}{C_t}\max|\varepsilon^{(k)}|\leqslant|\mu_\varepsilon|
  +t^{**}\sigma_\varepsilon\leqslant C_t\max|\varepsilon^{(k)}|,
\end{equation*}
and then, according to Lemma \ref{CM:lem:TFra}, $\mu(\Omega)\leqslant1$ and the adaptability of hyperparameters for the GP regression, there is $N_{\Omega,f}=\mathcal{O}(2^s\rho^n/\pi^n)$ such that
\begin{equation*}
N^{(k)}=C_tC\mu(D^{(k)})2^{kl+s}\rho^n/\pi^n\leqslant
C2^{kl}2^s\rho^n/\pi^n\leqslant2^{kl}N_{\Omega,f},
\end{equation*}
and
\begin{equation*}
  \mathcal{A}_{\chi^{(k)}}f=\mathcal{I}_{\chi^{(k)}}f,
\end{equation*}
where $C$ is as in Theorem \ref{CM:thm:LTC}. From Lemma \ref{CM:lem:TFra}, for any $k=0,1,\cdots,K-1$, since $\chi^{(k)}$ is quasi-uniformly distributed w.r.t. a sampling density of $C_tC2^{kl+s}\rho^n/\pi^n$, the model $\mathcal{A}_{\chi^{(k)}}f$ satisfies the error bound condition and strong convergence condition. Since the estimation of the model error bound, i.e.,  $|\mu_\varepsilon|+t_k\sigma_\varepsilon$, holds with probability at least
\begin{equation*}
  1-t_k^{-2}\geqslant e^{-1.2t_k^{-2}},~~\textrm{for all}~~k\geqslant2,
\end{equation*}
it follows from Theorem \ref{CM:thm:SConv} that the error upper bound
\begin{equation*}
 \max_{x\in D^{(k+1)}}[f(x)-f^*]\leqslant
 \left(\frac{1+\omega}{1+q}\right)\max_{x\in D^{(k)}}[f(x)-f^*]
\end{equation*}
also holds with probability at least $e^{-1.2t_k^{-2}}$. Thus, when $\sum_{k=0}^{K-1}t_k^{-2}\leqslant\frac{5}{6}\log\frac{1}{1-\delta}$, after $K$ contractions, the upper bound
\begin{equation*}
 \max_{x\in D^{(K)}}[f(x)-f^*]<\left(\frac{1+\omega}{1+q}\right)^K
 \max_{x\in\Omega}[f(x)-f^*]
\end{equation*}
holds with probability at least
\begin{equation*}
  \prod_{k=0}^{K-1}e^{-1.2t_k^{-2}}= e^{-1.2\sum_{k=0}^{K-1}t_k^{-2}}\geqslant1-\delta.
\end{equation*}
Finally, for all $K$ contractions, the total number of function evaluations is
\begin{equation*}
  \sum_{k=0}^{K-1}N^{(k)}\leqslant\sum_{k=0}^{K-1}2^{kl}N_{\Omega,f}
  =\mathcal{O}\bigg(\frac{2^{Kl}-1}{2^l-1}N_{\Omega,f}\bigg),
\end{equation*}
and similarly, according to Lemmas \ref{CM:lem:GP}, \ref{CM:lem:SP} and \ref{CM:lem:PBE}, even if $m=1$, that is, after adding two points (including the point that minimizes the current model), the model needs to be updated, then, the complexity required for the $k$th contraction does not exceed $\mathcal{O}\big(2^{4kl}N_{\Omega,f}^4\big)$, in other words, the total time complexity of all $K$ contractions does not exceed \begin{equation*}
  \sum_{k=0}^{K-1}2^{4kl}N_{\Omega,f}^4
  =\mathcal{O}\bigg(\frac{2^{4Kl}-1}{2^{4l}-1}N_{\Omega,f}^4\bigg),
\end{equation*}
and the proof is complete.
\end{proof}

\section{Numerical experiments}
\label{CM:s7}

We first compare the proposed algorithm with various global methods such as Bayesian optimization (BO), particle swarm optimization (PSO), genetic algorithm (GA), simulated annealing (SA) and differential evolution (DE) for several typical benchmark functions, then also consider a real world application: Lennard-Jones molecular conformation.

\subsection{Comparison with popular global methods}

Here we chose to include the Branin, SIN2, Ackley, and Rosenbrock functions as listed:
\begin{enumerate}
  \item Branin function: $f(x) = a(x_{2}-b x_{1}^{2}+cx_{1}-r)^2 +s(1-t)\cos(x_{1})+s$ with $x_{1}\in[-5,10]$, $x_2\in[0,15]$, where $a=1$, $b=5.1/(4\pi^2)$, $c=5/\pi$, $r=6$, $s=10$ and  $t=1/(8\pi)$. The Branin function has three global minima located at $x^*=(-\pi, 12.275), (\pi, 2.275)$ and $(3\pi,2.475)$, $f(x^*)=0.397887$.
  \item SIN2 function: $f(x)=1+\sin^2(x_{1})+\sin^2(x_{2}) -0.1\exp(-x_{1}^{2}-x_{2}^{2})$ with $x_{i}\in[-5,5]$ for $i=1,2$. Its global minimum is $f(x^*)=0.9$ at $x^*=(0,0)$.
  \item Ackley function:
  \begin{equation*}
    f(x)=-a\exp\left(-b\sqrt{\frac{1}{n}
         \sum_{i=1}^{n}x_{i}^{2}}\right)-\exp\left(\frac{1}{n}
         \sum_{i=1}^{n}\cos(cx_{i})\right)+a+\exp(1),
  \end{equation*}
  where $a=20$, $b=0.2$, $c=2\pi$ and $x_{i}\in [-32.768,32.768]$ for  $i=1,2,\cdots,n$. And the global minimum is $f(x^*)=0$ at $x^*=(0,\cdots,0)$.
  \item Rosenbrock function:
  \begin{equation*}
    f(x) = \sum_{i=1}^{n-1}\left[100(x_{i+1}-x_{i}^2)^2
         +(x_{i}-1)^2\right],
  \end{equation*}
  where $x_{i}\in [-2.048,2.048]$ for $i=1,2,\cdots,n$. The function is unimodal with the global minimum $f(x^*)=0$ at $x^*=(1,\cdots,1)$, which lies in a narrow, parabolic valley.
\end{enumerate}

\begin{table}[!htb]
\begin{center}
\caption{Properties of each function.}
\begin{tabular}{c|l|l|c}
  \hline
  No. & Name & Dimensionality & Feature \\
  \hline
  1 	& Branin 		& $2$			 & Three global minima\\
  2 	& SIN2 			& $2$			 & Many local minima\\
  3 	& Ackley 		& $4,6,10$		 & Many local minima\\
  4 	& Rosenbrock 	& $4,6,10$		 & Long and narrow valley\\
  \hline
\end{tabular}
\label{CM:tab:1}
\end{center}	
\end{table}

A summary of the properties of each objective function can be found in Table \ref{CM:tab:1} above. We have seen the Branin function as an illustrative example in Figures \ref{CM:fig:A1} and \ref{CM:fig:A2}, as well as the $2$D Rosenbrock function in Figure \ref{CM:fig:PT1}. The newly added functions SIN2 and Ackley both have many local minima. The Ackley and Rosenbrock functions are considered with $n=4,6$, and $10$, as they are defined with arbitrary dimensionality $n$. The empirical comparison with various global methods are shown in Figures \ref{CM:fig:Branin}-\ref{CM:fig:Rosenbrock10D}.

\begin{figure}[tbhp]
  \centering
  \subfigure{\includegraphics[width=0.32\textwidth]{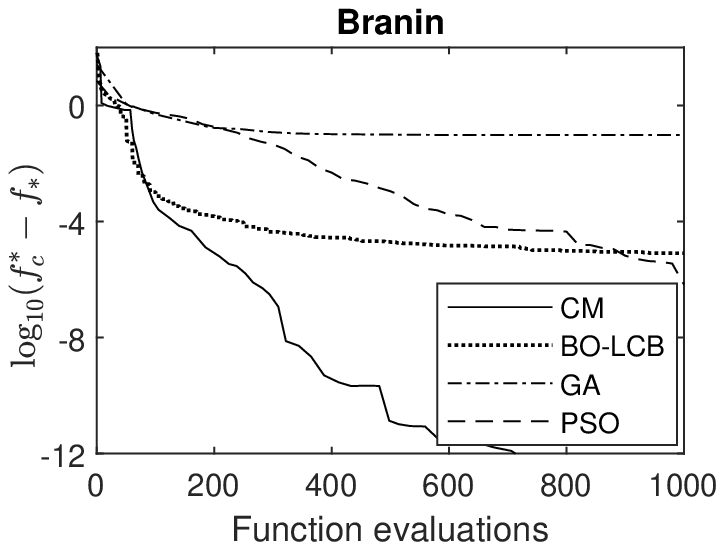}}
  \subfigure{\includegraphics[width=0.32\textwidth]{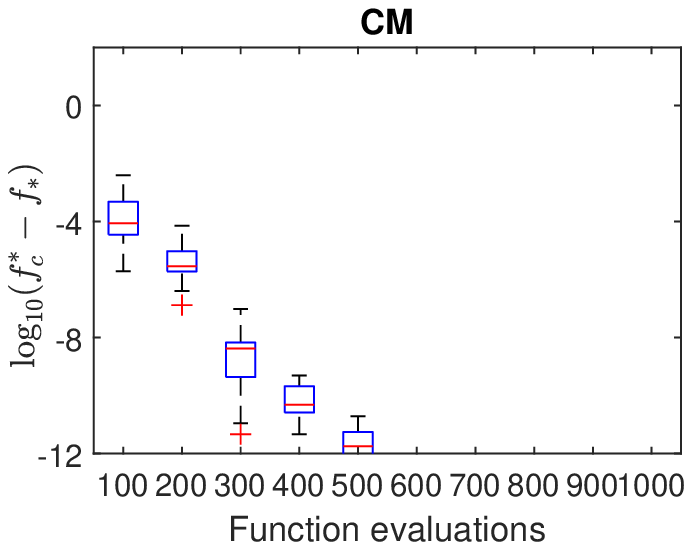}}
  \subfigure{\includegraphics[width=0.32\textwidth]{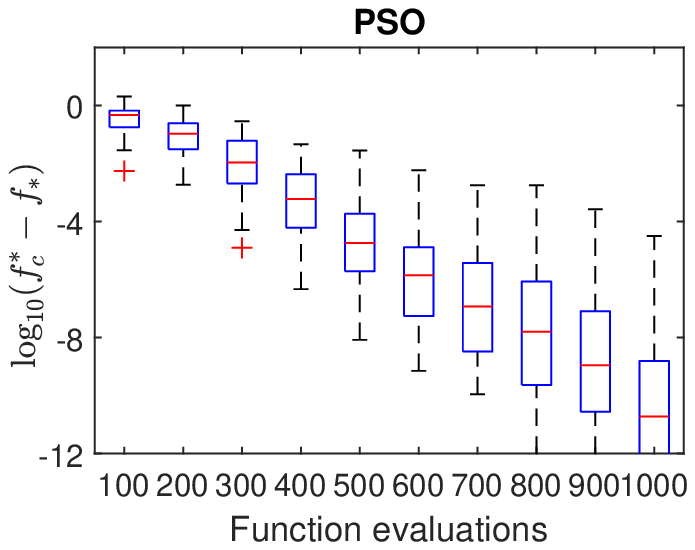}}\\
  \subfigure{\includegraphics[width=0.32\textwidth]{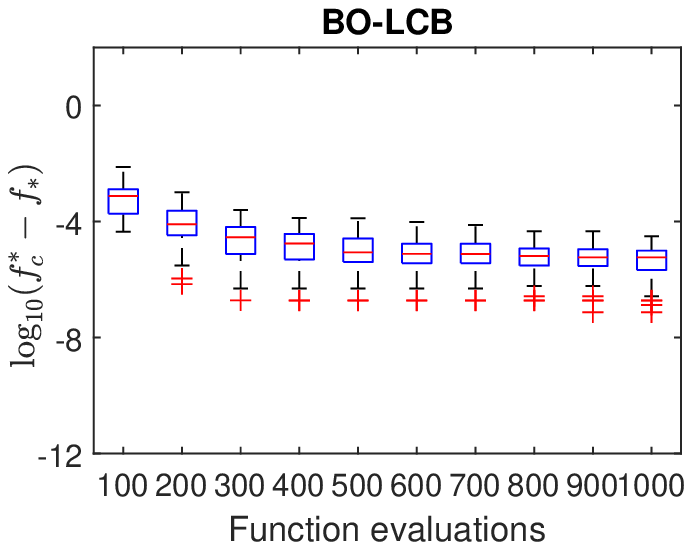}}
  \subfigure{\includegraphics[width=0.32\textwidth]{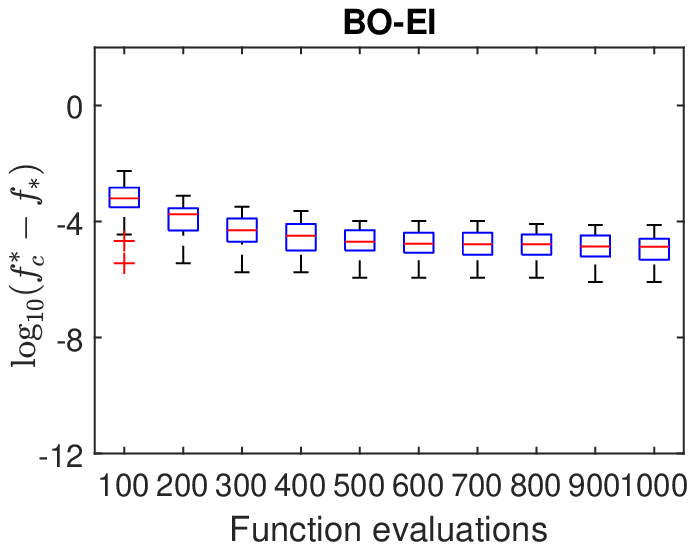}}
  \subfigure{\includegraphics[width=0.32\textwidth]{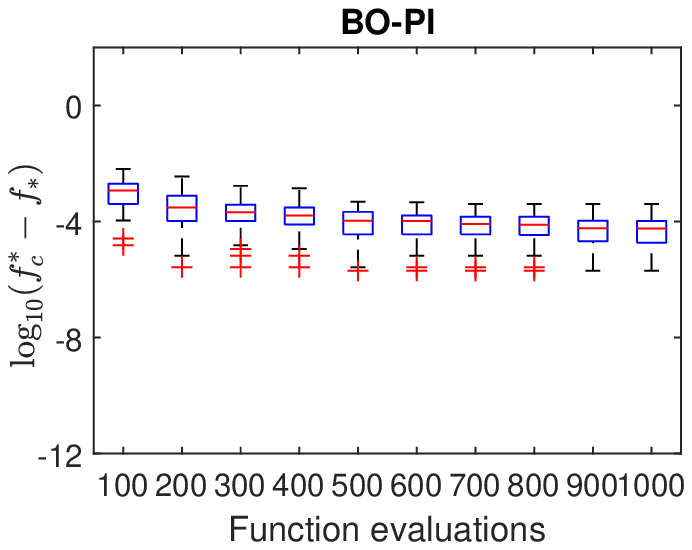}}\\
  \subfigure{\includegraphics[width=0.32\textwidth]{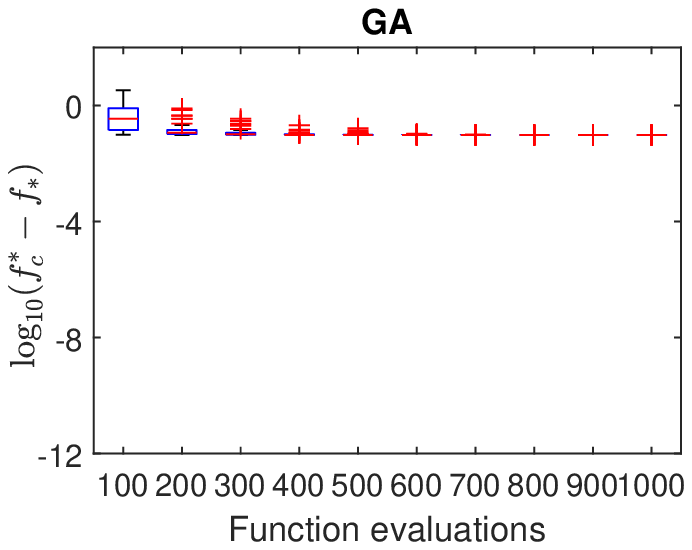}}
  \subfigure{\includegraphics[width=0.32\textwidth]{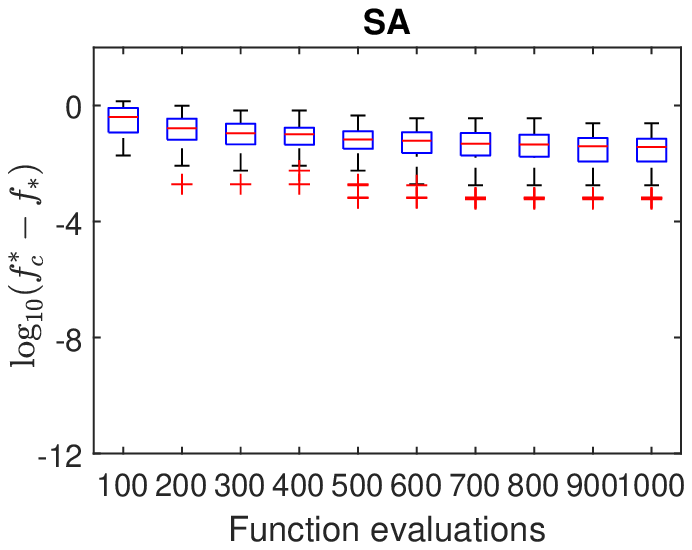}}
  \subfigure{\includegraphics[width=0.32\textwidth]{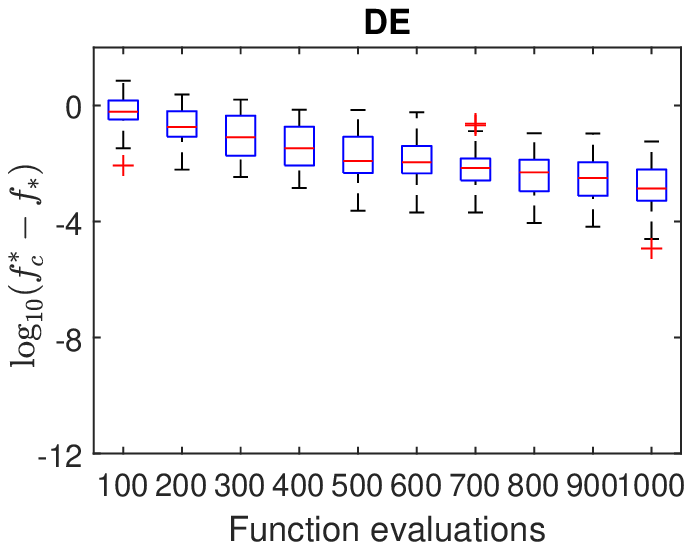}}
  \caption{Comparison for the Branin function. Upper left: the four curves correspond to the four most representative methods for this example, and each curve shows the averaged optimality gap over $50$ independent runs. Upper middle: the box plot shows the results of multiple runs for CM, with parameter setting $K=10$, $m=1$, $\textrm{minIterInner}=1$, $\omega=1$, $c_k=\bar{c}=50$ and $t_k=\bar{t}=2$. Upper right: multiple runs for the current most powerful competitor, i.e., PSO, with swarm size $20$ and default other parameters in MATLAB 2020b. Centre row: the results of multiple runs for three different types of BO with default parameter setting in MATLAB 2020b. Lower left: the results of multiple runs for GA with population size $50$ and default other parameters in MATLAB 2020b. Lower middle: the results of multiple runs for SA with default parameter setting in MATLAB 2020b. Lower right: the results of multiple runs for DE/rand/1/bin with parameter setting $N=10$, $F=0.85$ and $Cr=0.5$. The choice of those parameters follows the existing experiences.}
  \label{CM:fig:Branin}
\end{figure}

\begin{figure}[tbhp]
  \centering
  \subfigure{\includegraphics[width=0.32\textwidth]{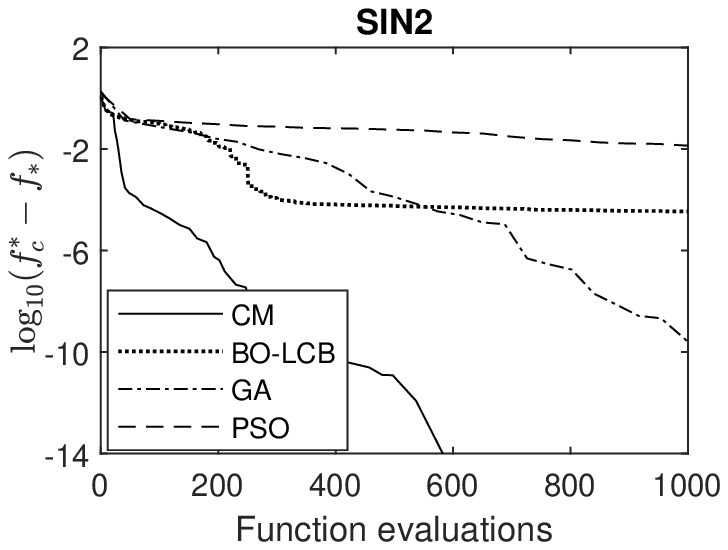}}
  \subfigure{\includegraphics[width=0.32\textwidth]{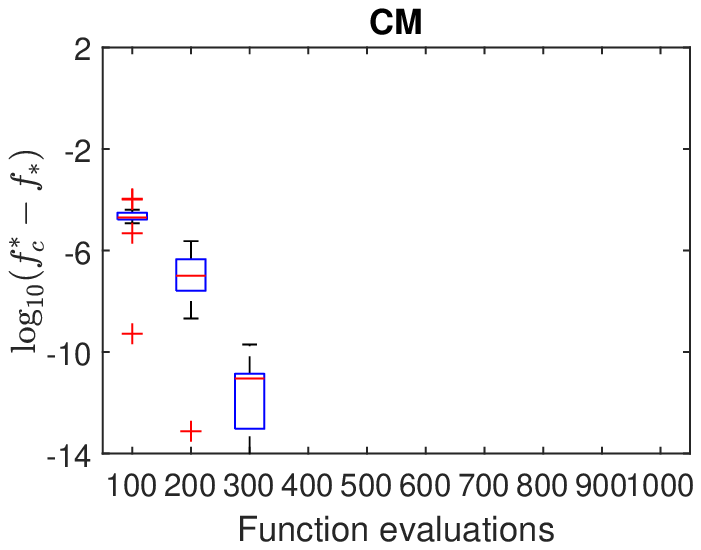}}
  \subfigure{\includegraphics[width=0.32\textwidth]{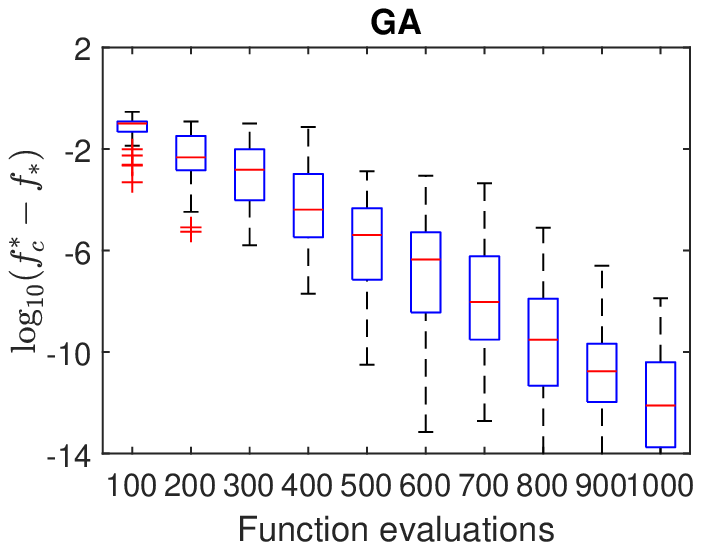}}\\
  \subfigure{\includegraphics[width=0.32\textwidth]{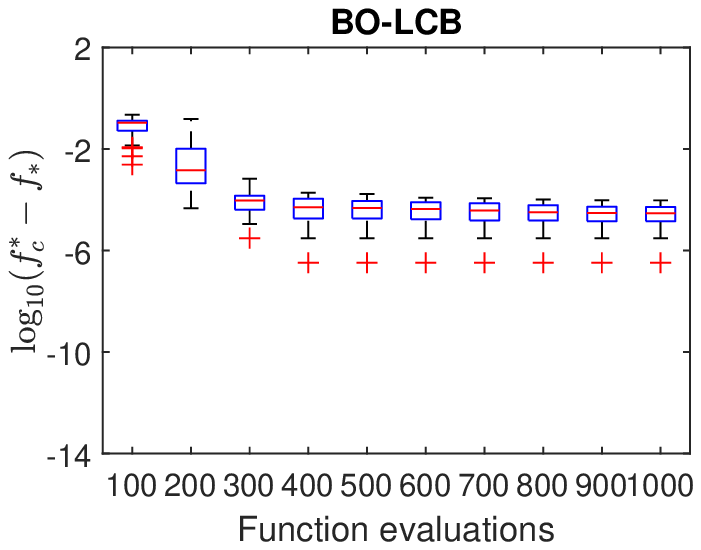}}
  \subfigure{\includegraphics[width=0.32\textwidth]{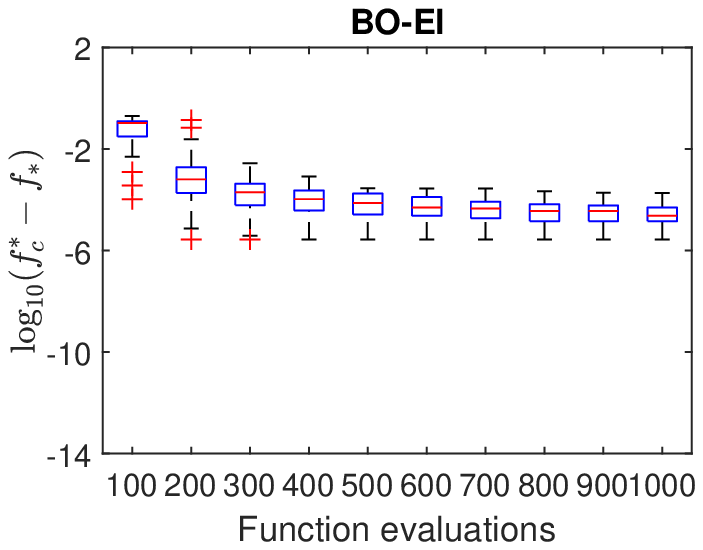}}
  \subfigure{\includegraphics[width=0.32\textwidth]{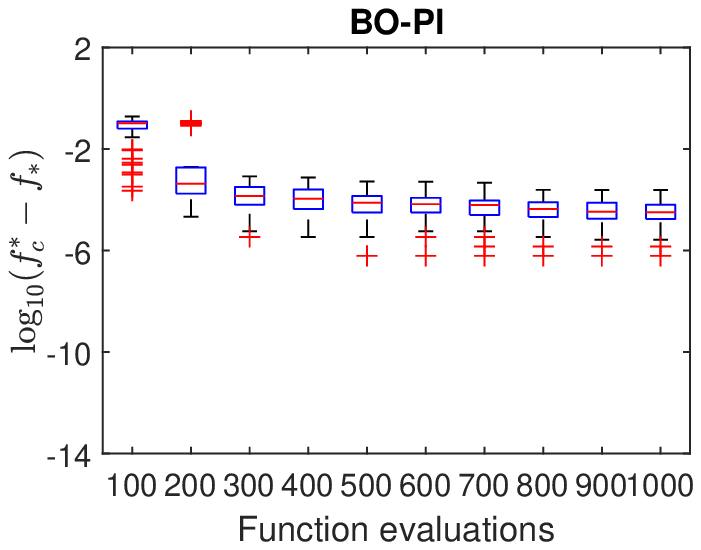}}\\
  \subfigure{\includegraphics[width=0.32\textwidth]{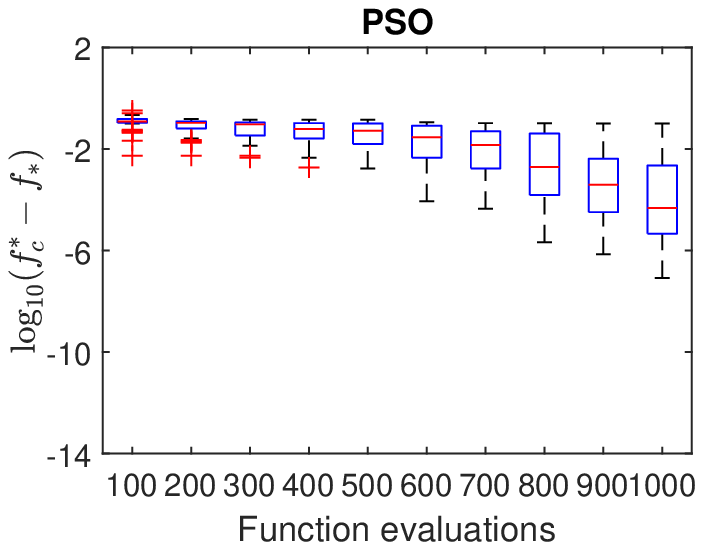}}
  \subfigure{\includegraphics[width=0.32\textwidth]{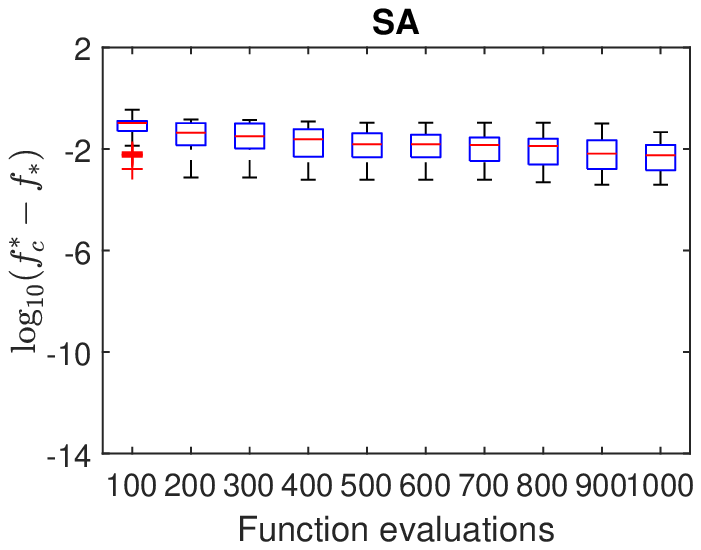}}
  \subfigure{\includegraphics[width=0.32\textwidth]{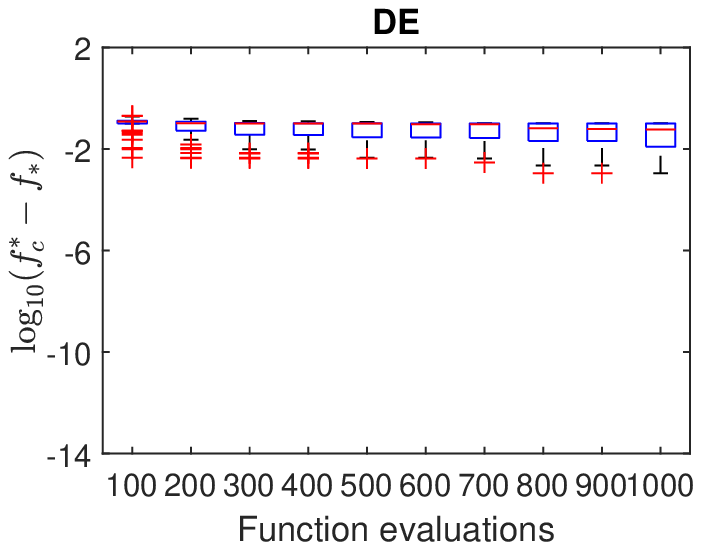}}
  \caption{Comparison for the SIN2 function. Upper left: the four curves correspond to the four most representative methods for this example, and each curve shows the averaged optimality gap over $50$ independent runs. Upper middle: the box plot shows the results of multiple runs for CM, with parameter setting $K=10$, $m=2$, $\textrm{minIterInner}=1$, $\omega=1$, $c_k=\bar{c}=50$ and $t_k=\bar{t}=2.5$. Upper right: multiple runs for the current most powerful competitor, i.e., GA, with population size $40$ and default other parameters in MATLAB 2020b. Centre row: the results of multiple runs for three different types of BO with default parameter setting in MATLAB 2020b. Lower left: the results of multiple runs for PSO, with swarm size $50$ and default other parameters in MATLAB 2020b. Lower middle: the results of multiple runs for SA with default parameter setting in MATLAB 2020b. Lower right: the results of multiple runs for DE/rand/1/bin with parameter setting $N=20$, $F=0.85$ and $Cr=0.5$. The choice of those parameters follows the existing experiences.}
  \label{CM:fig:SIN2}
\end{figure}

\begin{figure}[tbhp]
  \centering
  \subfigure{\includegraphics[width=0.32\textwidth]{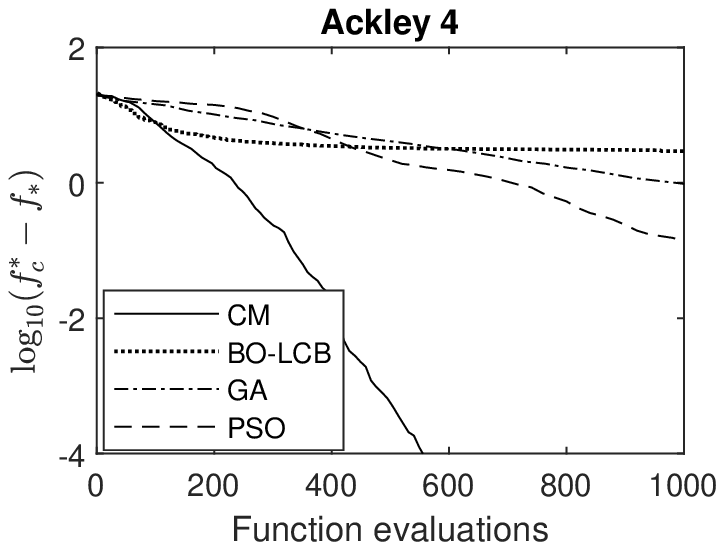}}
  \subfigure{\includegraphics[width=0.32\textwidth]{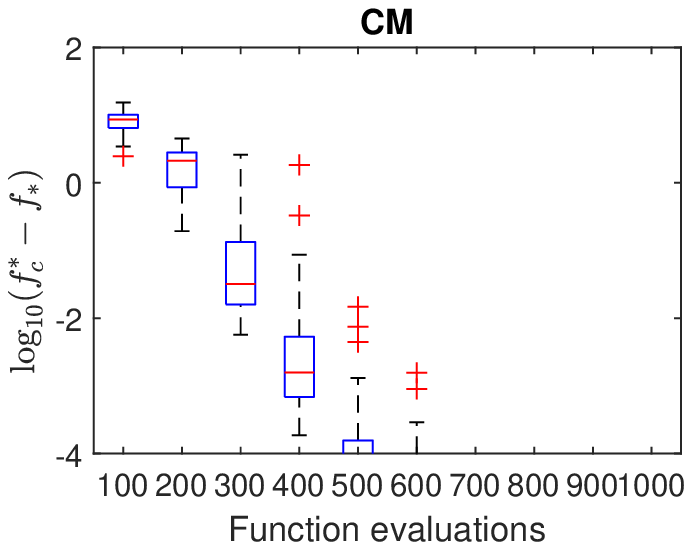}}
  \subfigure{\includegraphics[width=0.32\textwidth]{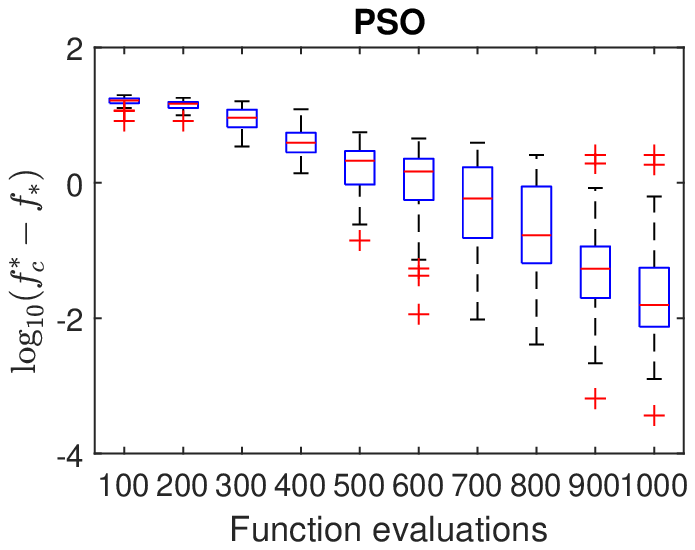}}\\
  \subfigure{\includegraphics[width=0.32\textwidth]{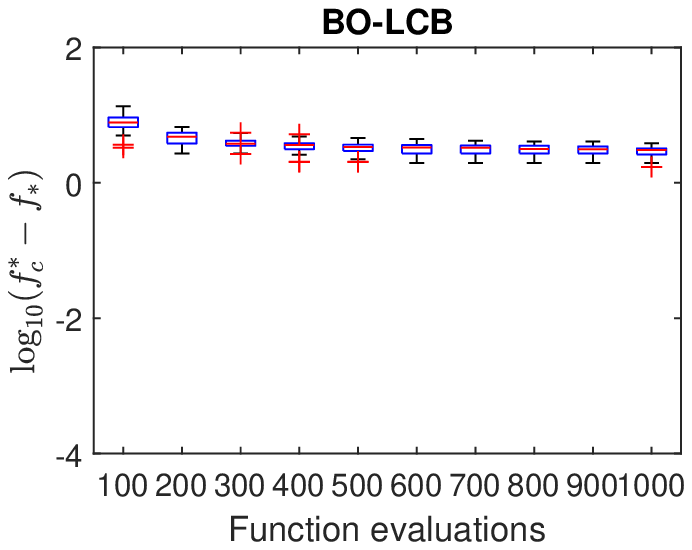}}
  \subfigure{\includegraphics[width=0.32\textwidth]{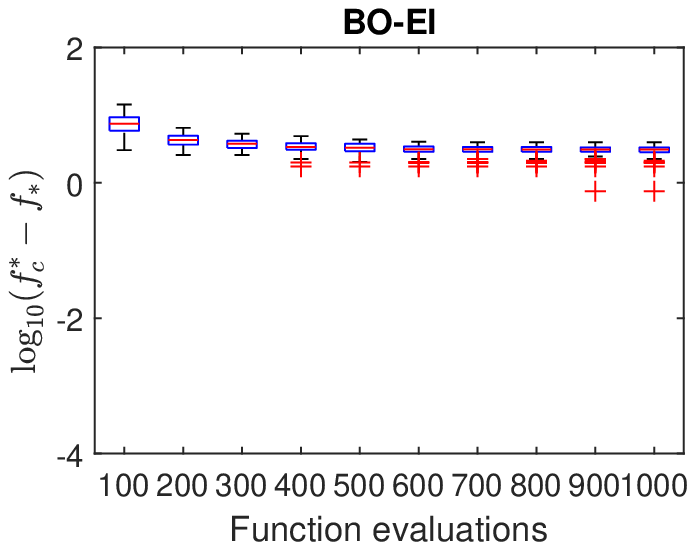}}
  \subfigure{\includegraphics[width=0.32\textwidth]{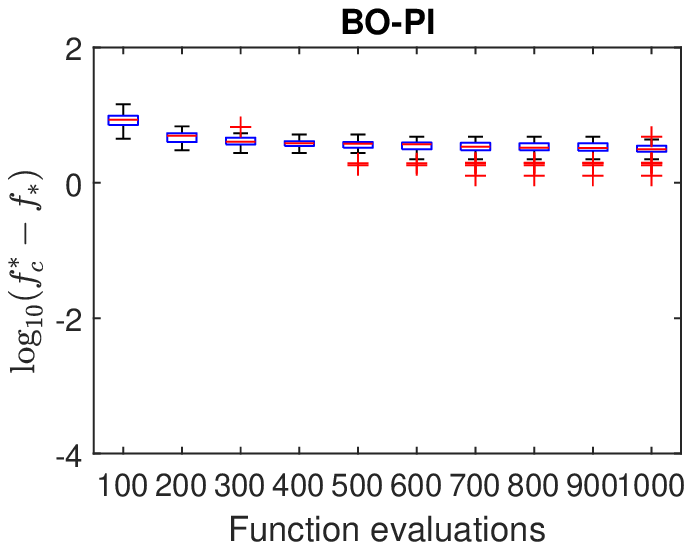}}\\
  \subfigure{\includegraphics[width=0.32\textwidth]{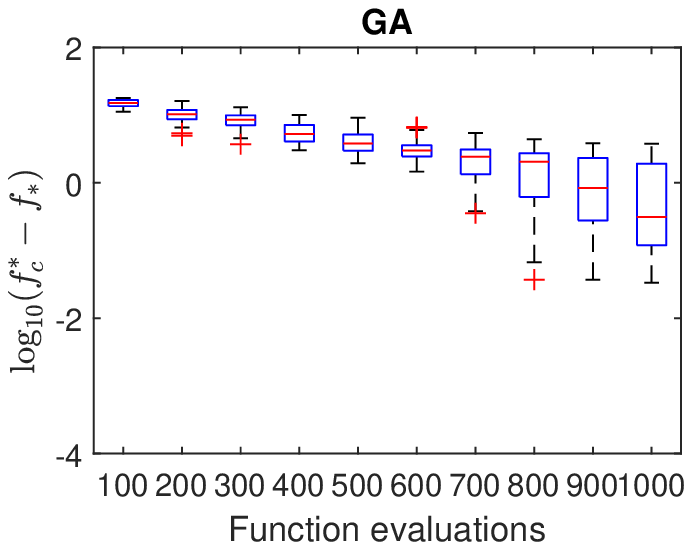}}
  \subfigure{\includegraphics[width=0.32\textwidth]{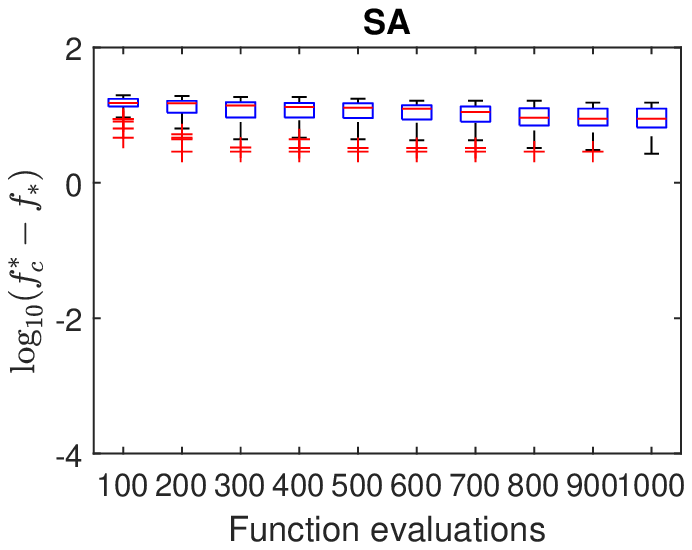}}
  \subfigure{\includegraphics[width=0.32\textwidth]{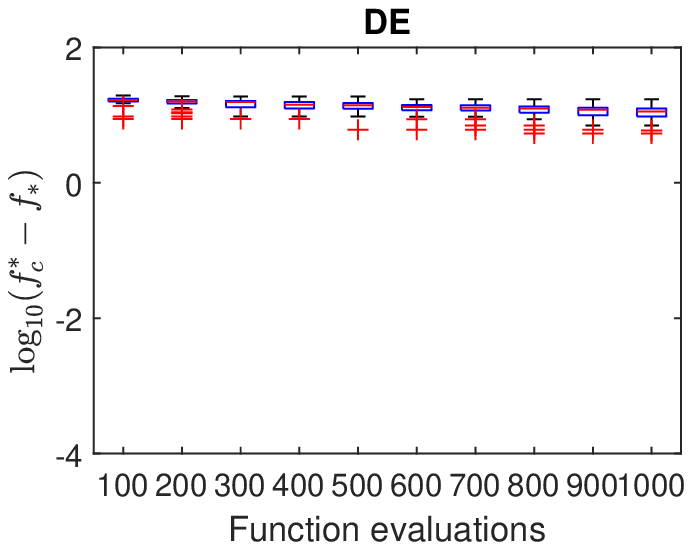}}
  \caption{Comparison for the Ackley function in $4$ dimension. Upper left: the four curves correspond to the four most representative methods for this example, each curve shows the averaged optimality gap over $50$ independent runs. Upper middle: the box plot shows the results of multiple runs for CM, with parameter setting $K=50$, $m=4$, $\textrm{minIterInner}=3$, $\omega=1$, $c_k=\bar{c}=50$ and $t_k=\bar{t}=4$. Upper right: multiple runs for the current most powerful competitor, i.e., PSO, with swarm size $20$ and default other parameters in MATLAB 2020b. Centre row: the results of multiple runs for three different types of BO with default parameter setting in MATLAB 2020b. Lower left: the results of multiple runs for GA with population size $40$ and default other parameters in MATLAB 2020b. Lower middle: the results of multiple runs for SA with default parameter setting in MATLAB 2020b. Lower right: the results of multiple runs for DE/rand/1/bin with parameter setting $N=40$, $F=0.8$ and $Cr=0.5$. The choice of those parameters follows the existing experiences.}
  \label{CM:fig:Ackley4D}
\end{figure}

\begin{figure}[tbhp]
  \centering
  \subfigure{\includegraphics[width=0.32\textwidth]{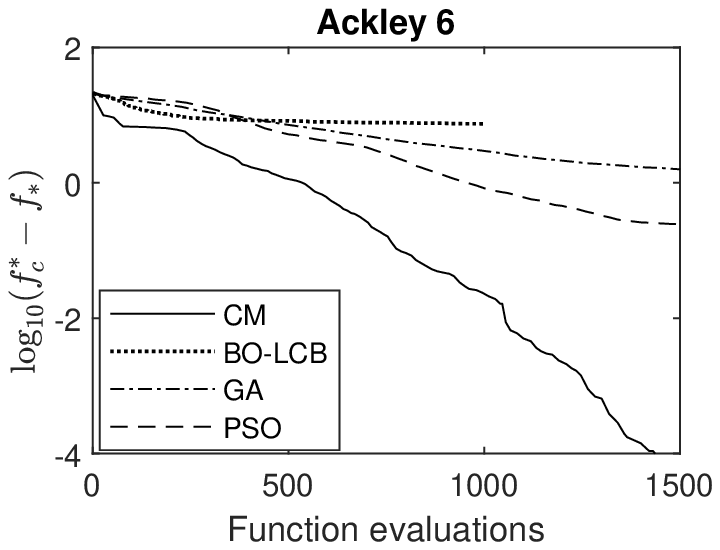}}
  \subfigure{\includegraphics[width=0.32\textwidth]{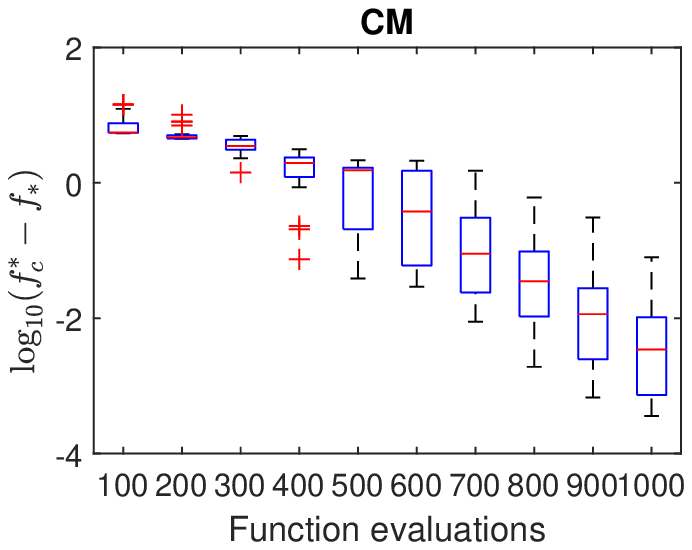}}
  \subfigure{\includegraphics[width=0.32\textwidth]{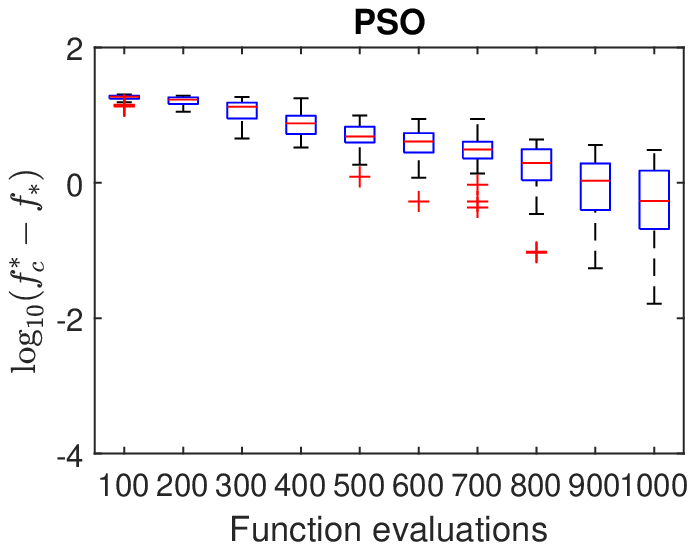}}\\
  \subfigure{\includegraphics[width=0.32\textwidth]{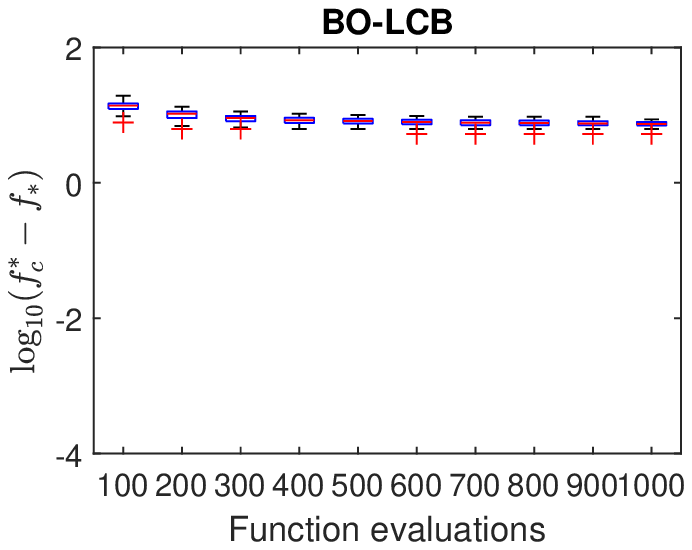}}
  \subfigure{\includegraphics[width=0.32\textwidth]{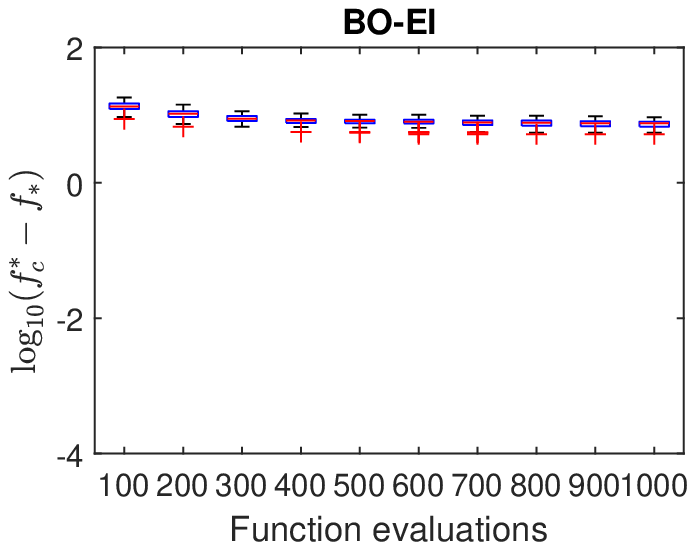}}
  \subfigure{\includegraphics[width=0.32\textwidth]{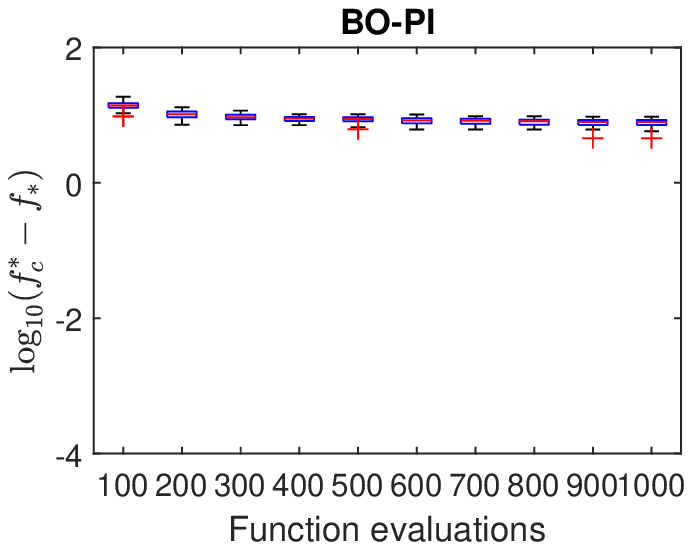}}\\
  \subfigure{\includegraphics[width=0.32\textwidth]{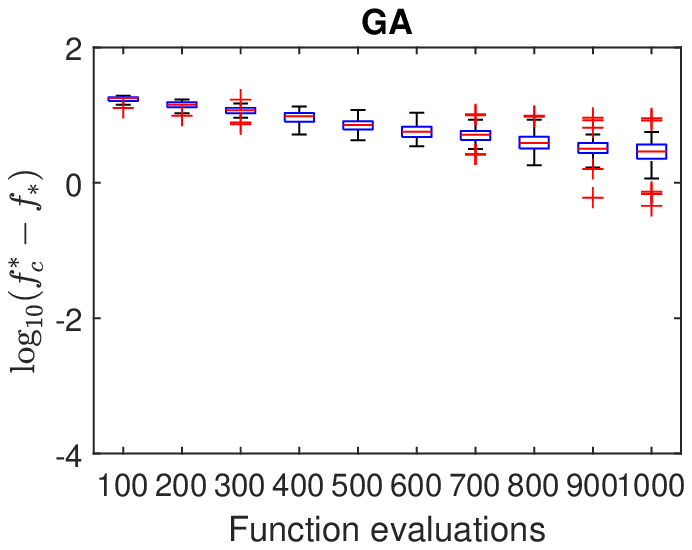}}
  \subfigure{\includegraphics[width=0.32\textwidth]{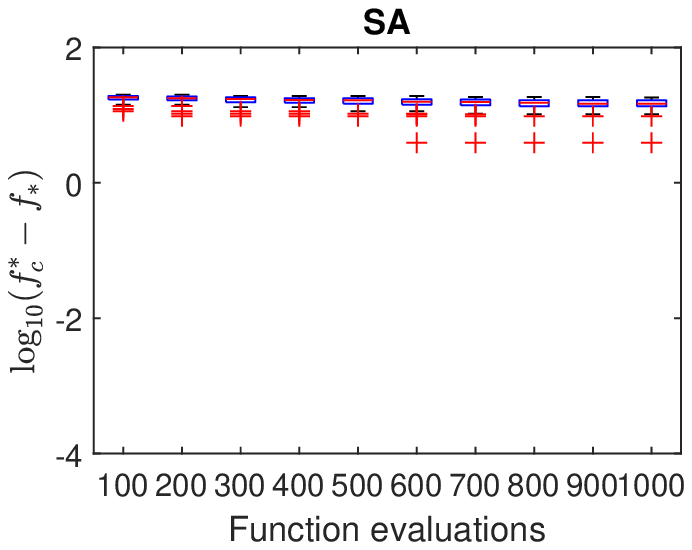}}
  \subfigure{\includegraphics[width=0.32\textwidth]{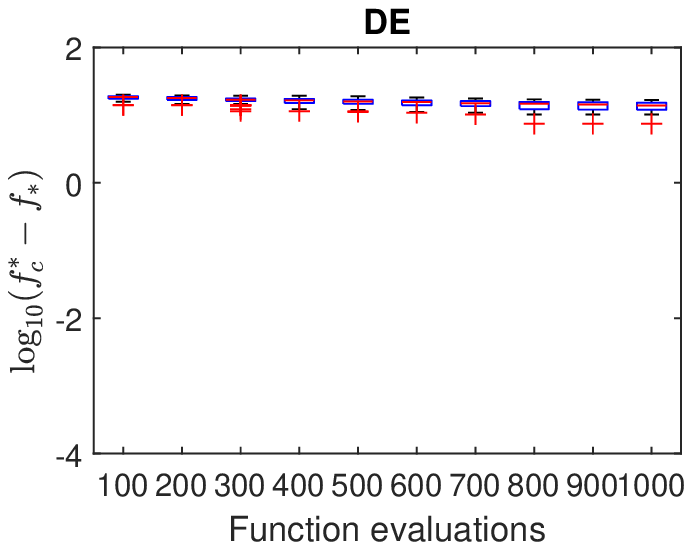}}
  \caption{Comparison for the Ackley function in $6$ dimension. Upper left: the four curves correspond to the four most representative methods for this example, each curve shows the averaged optimality gap over $50$ independent runs. Upper middle: the box plot shows the results of multiple runs for CM, with parameter setting $K=50$, $m=6$, $\textrm{minIterInner}=5$, $\omega=1$, $c_k=\bar{c}=50$ and $t_k=\bar{t}=4.5$. Upper right: multiple runs for the current most powerful competitor, i.e., PSO, with swarm size $20$ and default other parameters in MATLAB 2020b. Centre row: the results of multiple runs for three different types of BO with default parameter setting in MATLAB 2020b. Lower left: the results of multiple runs for GA with population size $40$ and default other parameters in MATLAB 2020b. Lower middle: the results of multiple runs for SA with default parameter setting in MATLAB 2020b. Lower right: the results of multiple runs for DE/rand/1/bin with parameter setting $N=50$, $F=0.8$ and $Cr=0.5$. The choice of those parameters follows the existing experiences.}
  \label{CM:fig:Ackley6D}
\end{figure}

\begin{figure}[tbhp]
  \centering
  \subfigure{\includegraphics[width=0.32\textwidth]{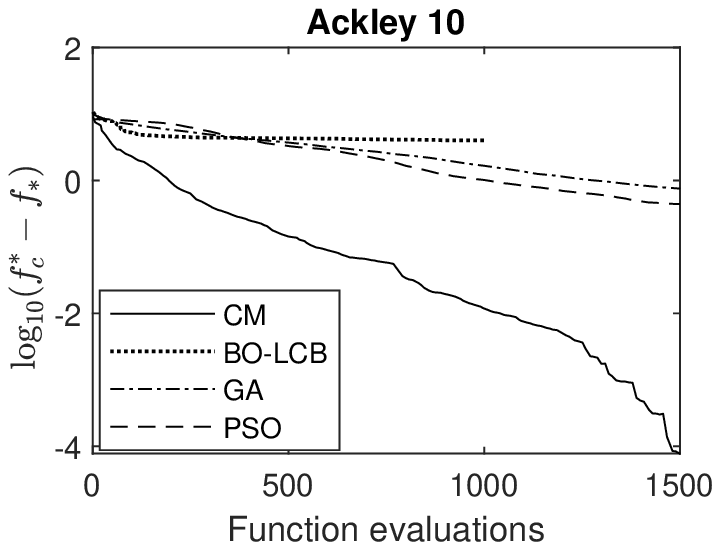}}
  \subfigure{\includegraphics[width=0.32\textwidth]{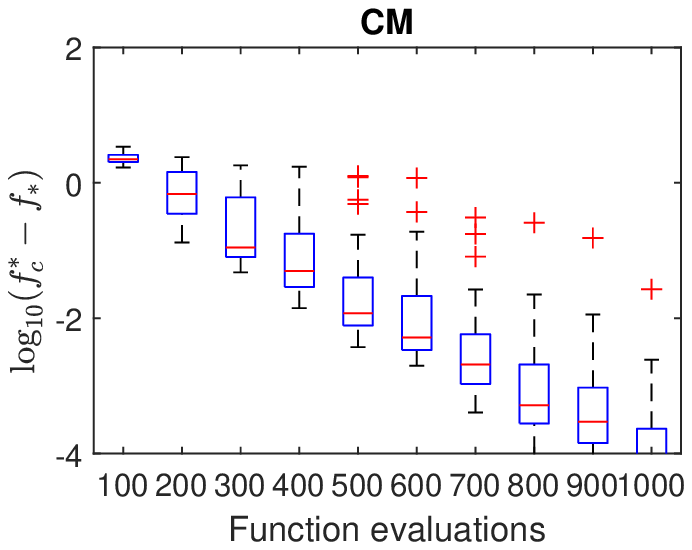}}
  \subfigure{\includegraphics[width=0.32\textwidth]{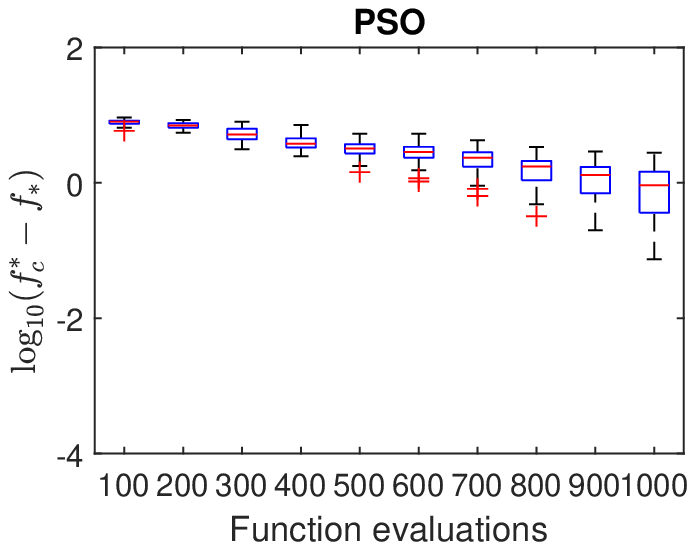}}\\
  \subfigure{\includegraphics[width=0.32\textwidth]{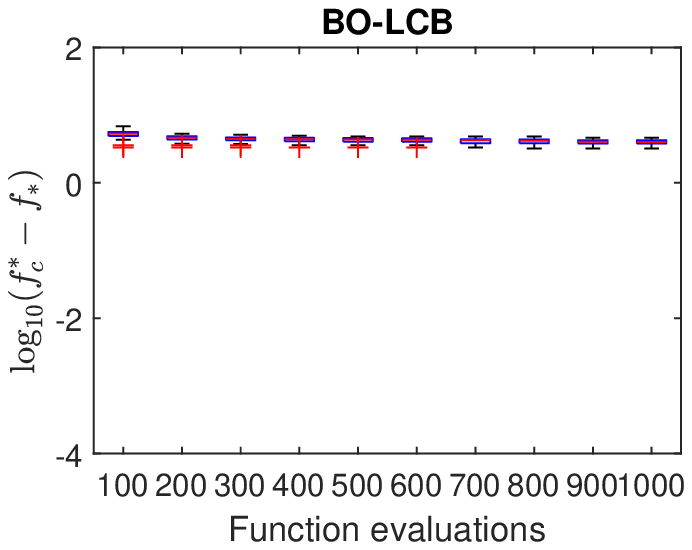}}
  \subfigure{\includegraphics[width=0.32\textwidth]{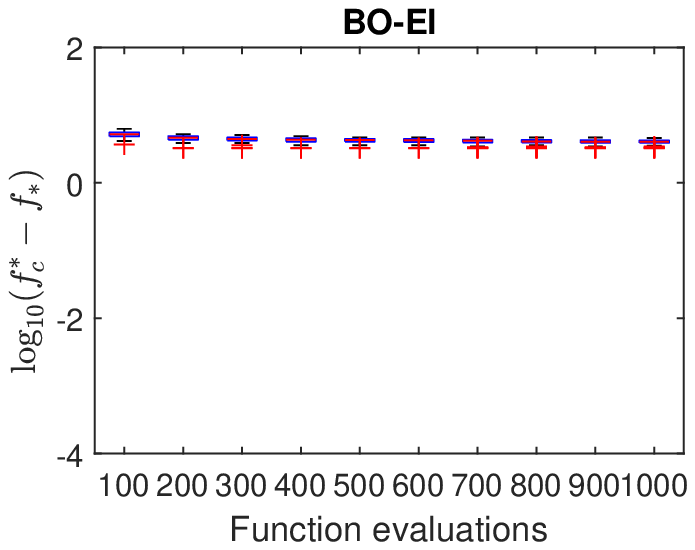}}
  \subfigure{\includegraphics[width=0.32\textwidth]{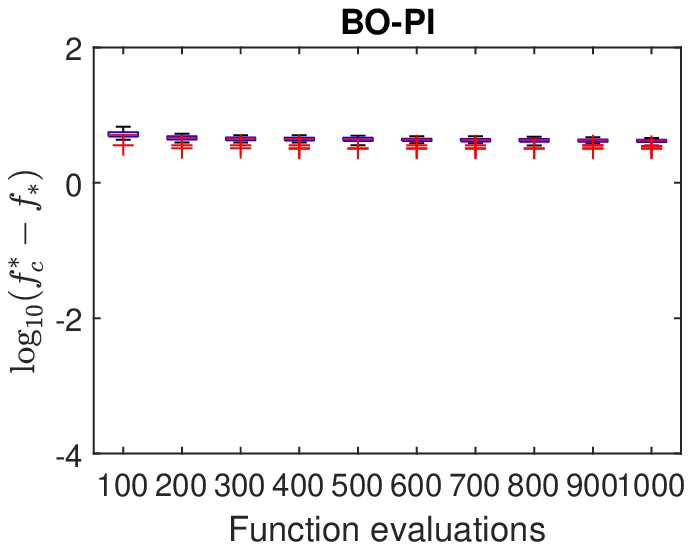}}\\
  \subfigure{\includegraphics[width=0.32\textwidth]{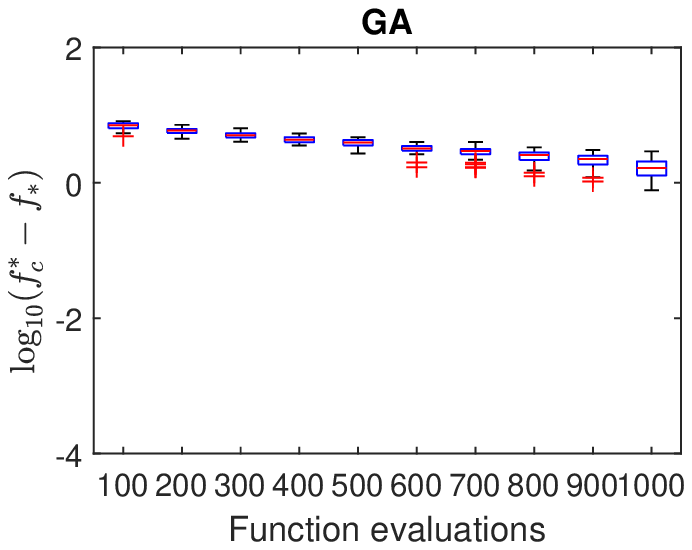}}
  \subfigure{\includegraphics[width=0.32\textwidth]{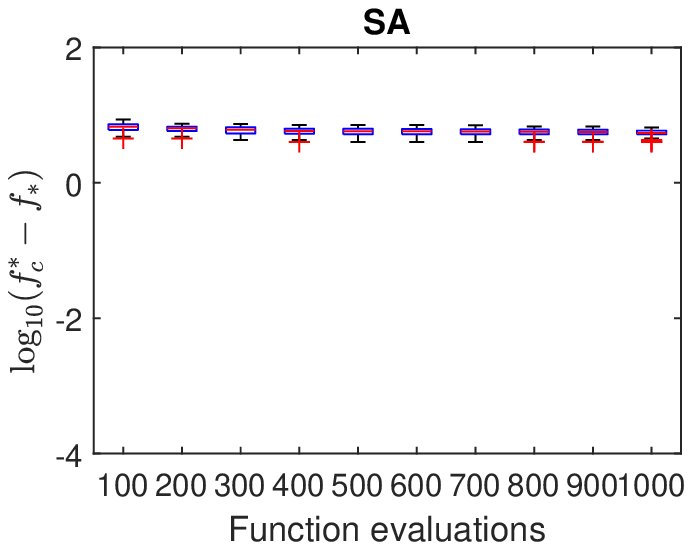}}
  \subfigure{\includegraphics[width=0.32\textwidth]{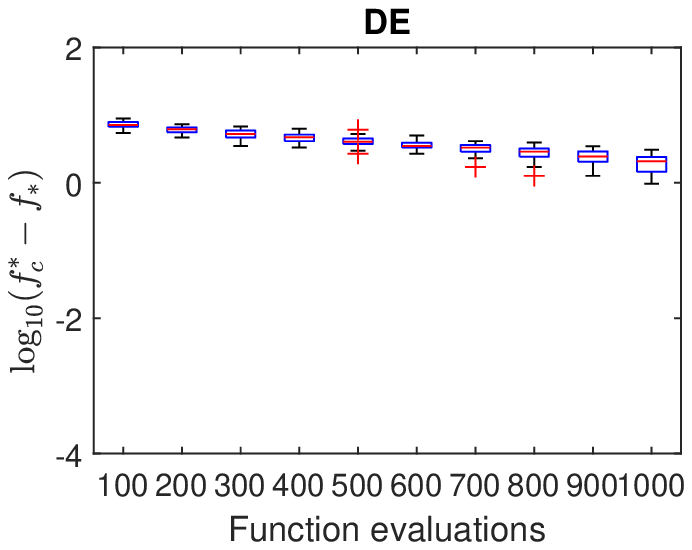}}
  \caption{Comparison for the Ackley function in $10$ dimension, especially, in order to make the performance of the algorithms easier to observe, $x_i$ is limited in $[-5,5]$ for $i=1,\cdots,10$. Upper left: the four curves correspond to the four most representative methods for this example, and each curve shows the averaged optimality gap over $50$ independent runs. Upper middle: the box plot shows the results of multiple runs for CM, with parameter setting $K=100$, $m=2$, $\textrm{minIterInner}=25$, $\omega=1$, $c_k=\bar{c}=50$ and $t_k=4.5$. Upper right: multiple runs for the current most powerful competitor, i.e., PSO, with swarm size $20$ and default other parameters in MATLAB 2020b. Centre row: the results of multiple runs for three different types of BO with default parameter setting in MATLAB 2020b. Lower left: the results of multiple runs for GA with population size $50$ and default other parameters in MATLAB 2020b. Lower middle: the results of multiple runs for SA with default parameter setting in MATLAB 2020b. Lower right: the results of multiple runs for DE/rand/1/bin with parameter setting $N=10$, $F=0.8$ and $Cr=0.5$. The choice of those parameters follows the existing experiences.}
  \label{CM:fig:Ackley10D}
\end{figure}

\begin{figure}[tbhp]
  \centering
  \subfigure{\includegraphics[width=0.32\textwidth]{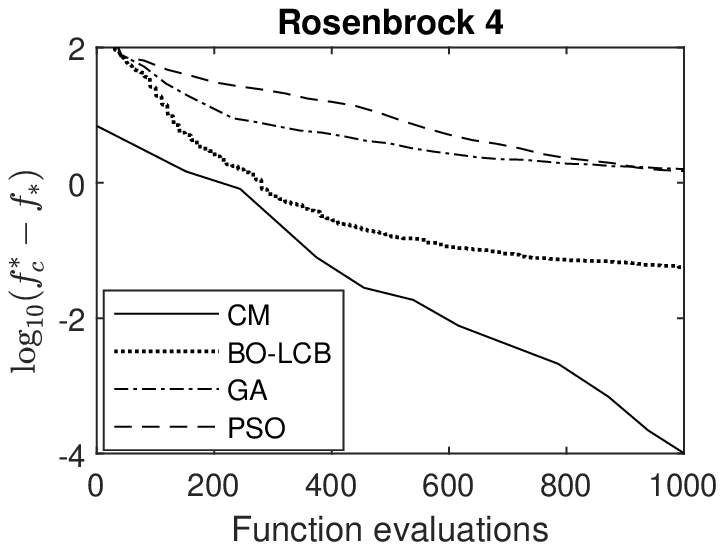}}
  \subfigure{\includegraphics[width=0.32\textwidth]{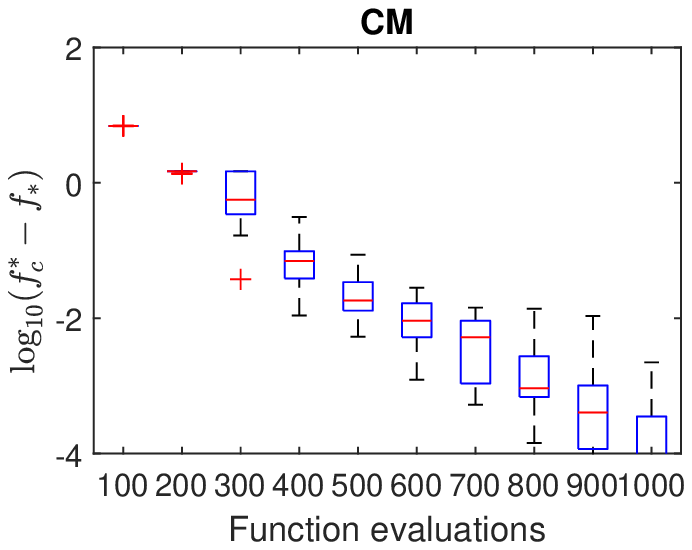}}
  \subfigure{\includegraphics[width=0.32\textwidth]{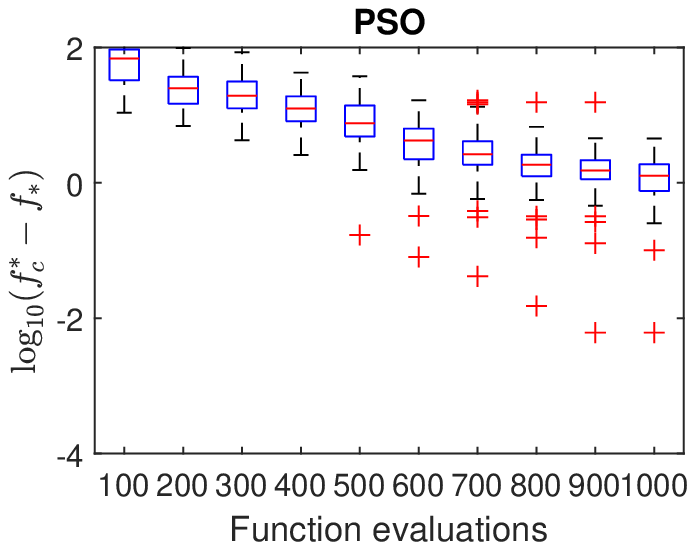}}\\
  \subfigure{\includegraphics[width=0.32\textwidth]{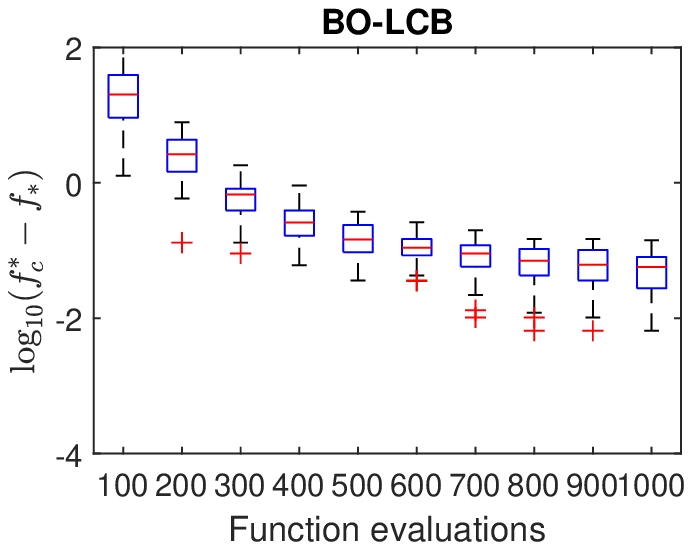}}
  \subfigure{\includegraphics[width=0.32\textwidth]{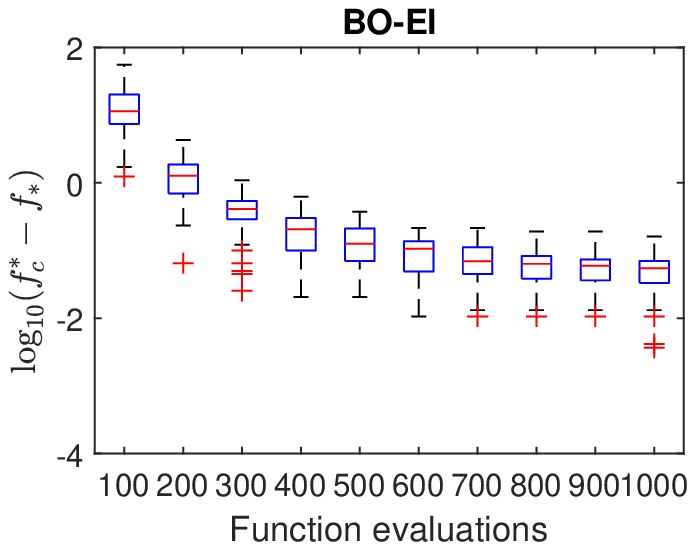}}
  \subfigure{\includegraphics[width=0.32\textwidth]{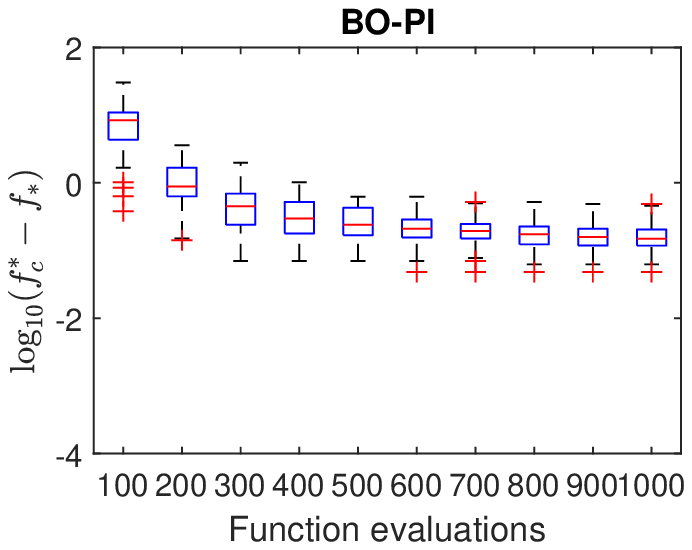}}\\
  \subfigure{\includegraphics[width=0.32\textwidth]{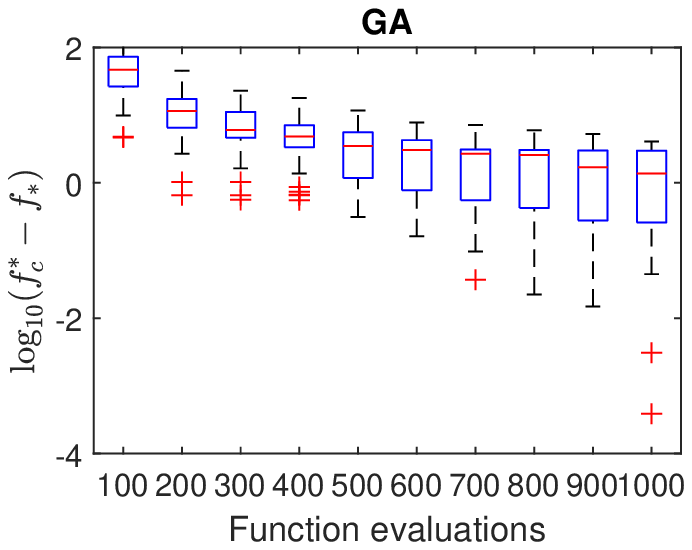}}
  \subfigure{\includegraphics[width=0.32\textwidth]{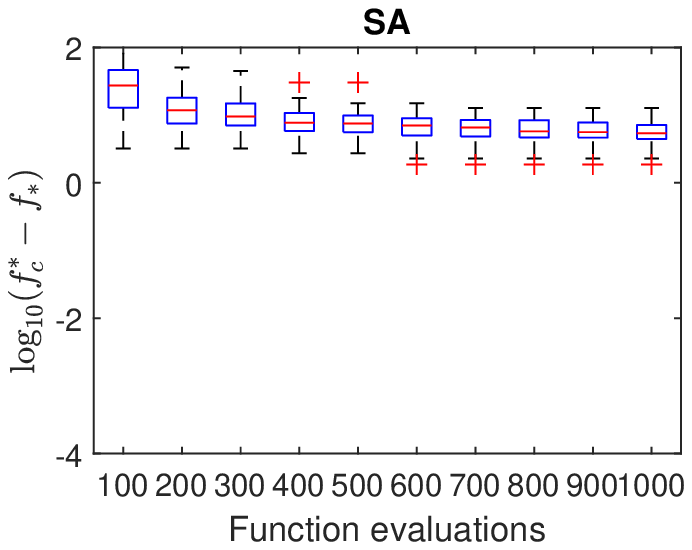}}
  \subfigure{\includegraphics[width=0.32\textwidth]{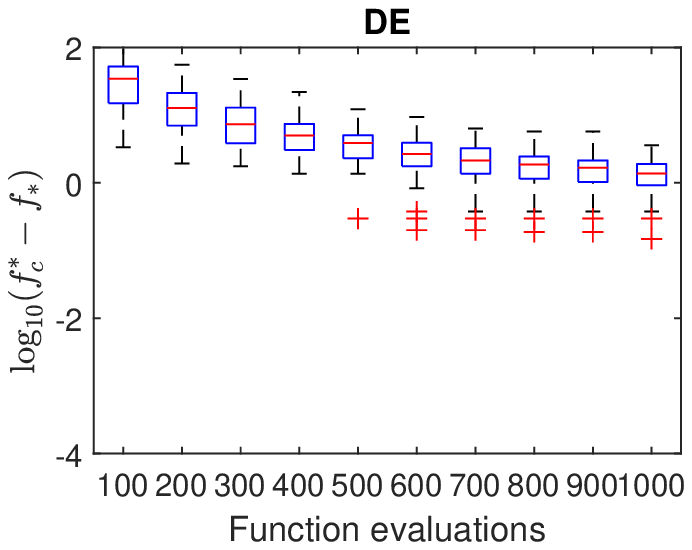}}
  \caption{Comparison for the Rosenbrock function in $4$ dimension. Upper left: the four curves correspond to the four most representative methods for this example, and each curve shows the averaged optimality gap over $50$ independent runs. Upper middle: the box plot shows the results of multiple runs for CM, with parameter setting $K=50$, $m=2$, $\textrm{minIterInner}=8$, $\omega=1$, $c_k=\bar{c}=50$ and $t_k=3.5\frac{k}{K}$. Upper right: multiple runs for the current most powerful competitor, i.e., PSO, with swarm size $40$ and default other parameters in MATLAB 2020b. Centre row: the results of multiple runs for three different types of BO with default parameter setting in MATLAB 2020b. Lower left: the results of multiple runs for GA with population size $40$ and default other parameters in MATLAB 2020b. Lower middle: the results of multiple runs for SA with default parameter setting in MATLAB 2020b. Lower right: the results of multiple runs for DE/rand/1/bin with parameter setting $N=10$, $F=0.85$ and $Cr=0.5$. The choice of those parameters follows the existing experiences.}
  \label{CM:fig:Rosenbrock4D}
\end{figure}

\begin{figure}[tbhp]
  \centering
  \subfigure{\includegraphics[width=0.32\textwidth]{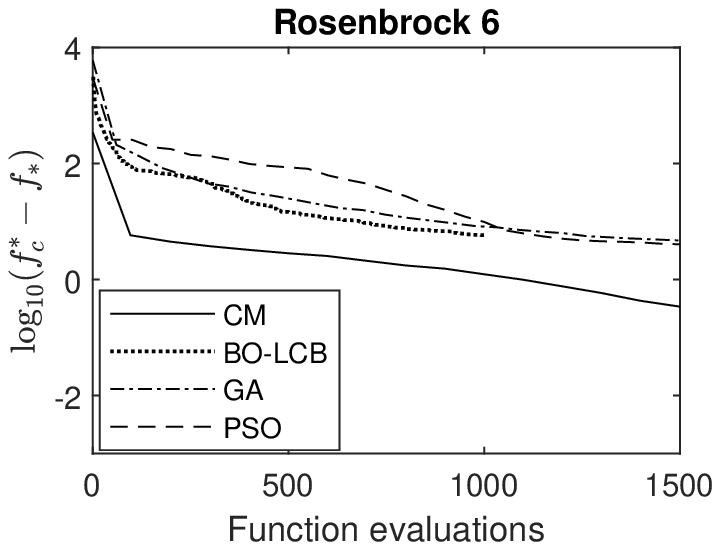}}
  \subfigure{\includegraphics[width=0.32\textwidth]{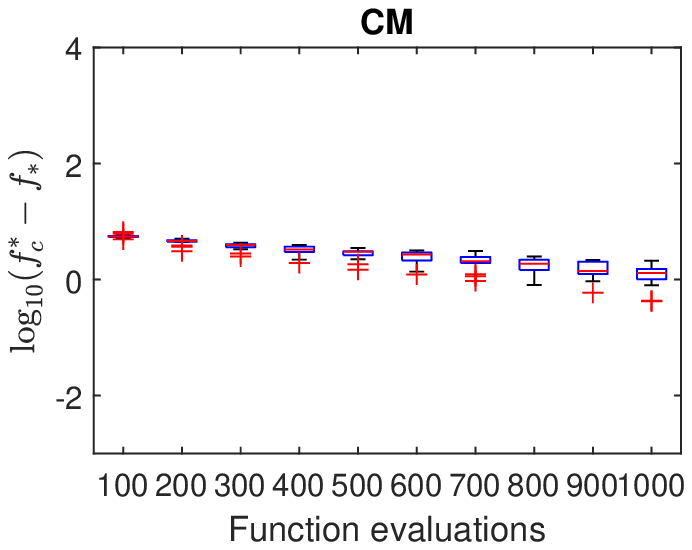}}
  \subfigure{\includegraphics[width=0.32\textwidth]{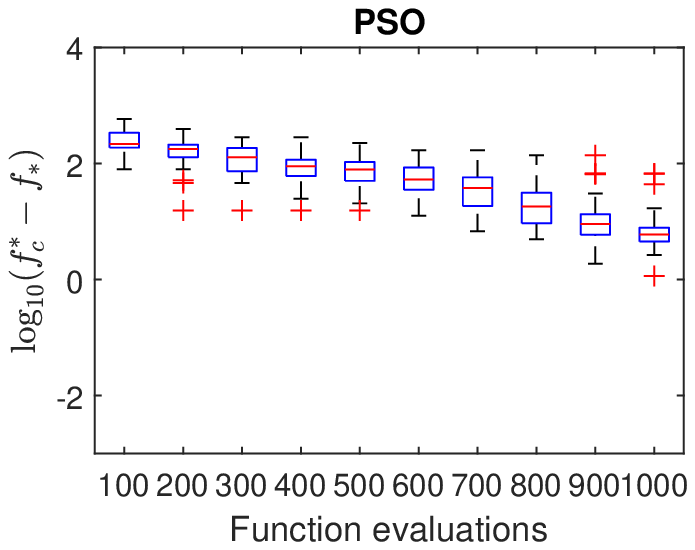}}\\
  \subfigure{\includegraphics[width=0.32\textwidth]{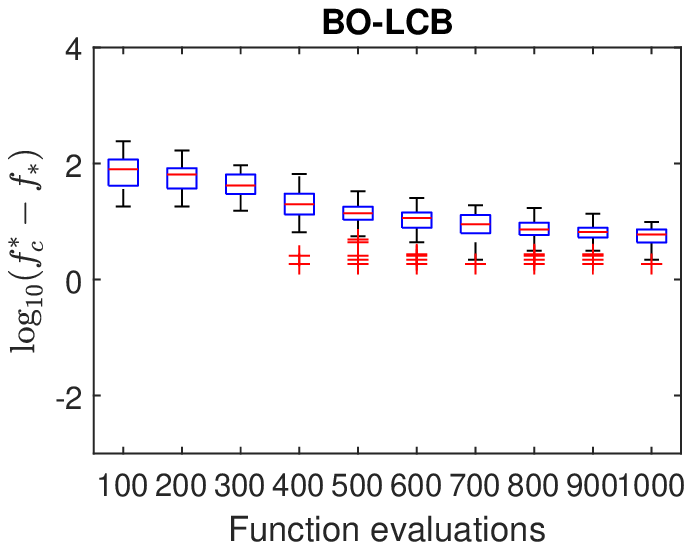}}
  \subfigure{\includegraphics[width=0.32\textwidth]{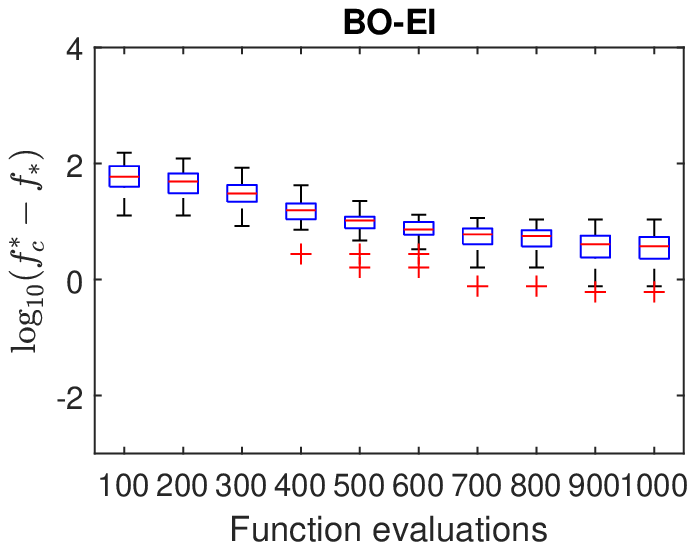}}
  \subfigure{\includegraphics[width=0.32\textwidth]{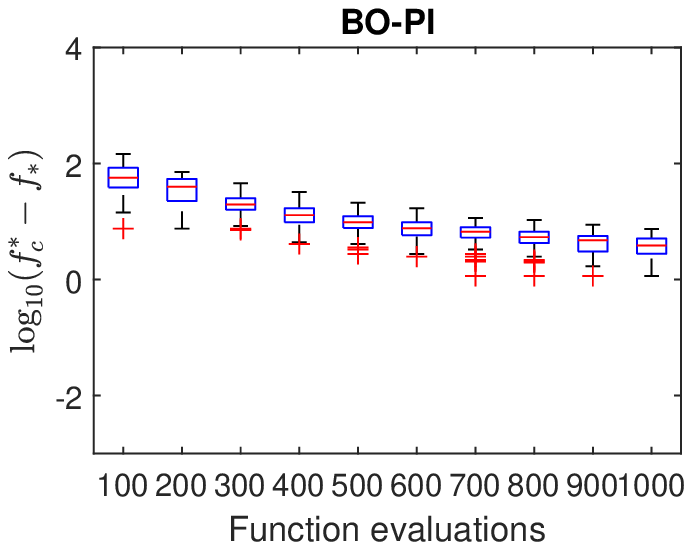}}\\
  \subfigure{\includegraphics[width=0.32\textwidth]{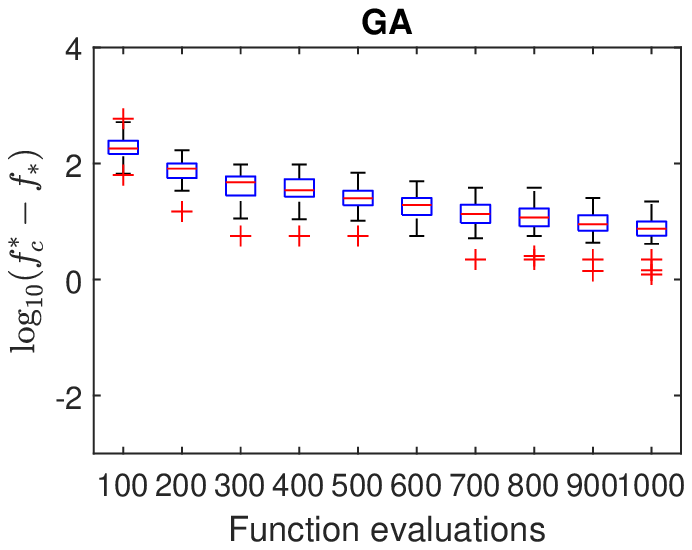}}
  \subfigure{\includegraphics[width=0.32\textwidth]{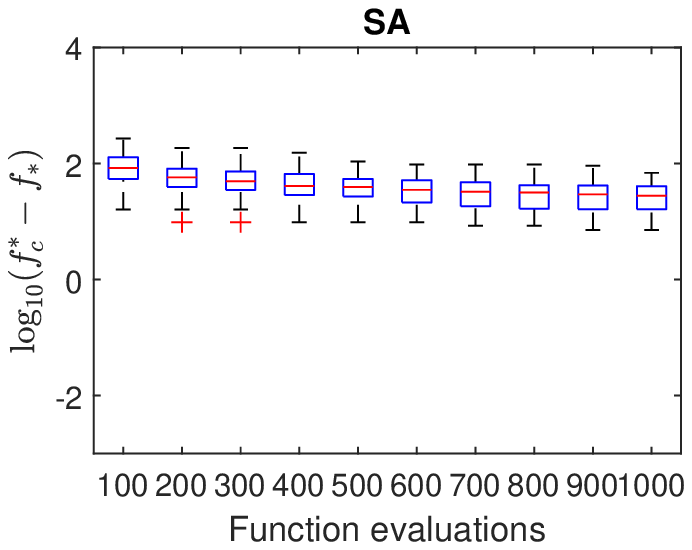}}
  \subfigure{\includegraphics[width=0.32\textwidth]{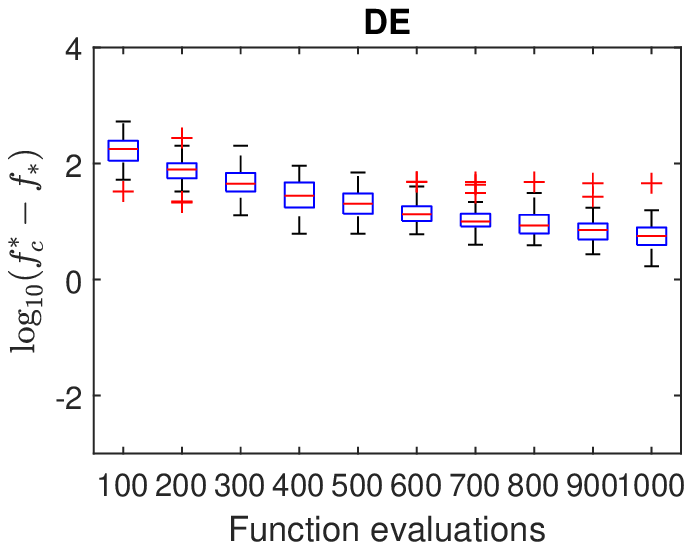}}
  \caption{Comparison for the Rosenbrock function in $6$ dimension. Upper left: the four curves correspond to the four most representative methods for this example, and each curve shows the averaged optimality gap over $50$ independent runs. Upper middle: the box plot shows the results of multiple runs for CM, with parameter setting $K=50$, $m=2$, $\textrm{minIterInner}=10$, $\omega=1$, $c_k=\bar{c}=50$ and $t_k=3.5\frac{k}{K}$. Upper right: multiple runs for the current most powerful competitor, i.e., PSO, with swarm size $50$ and default other parameters in MATLAB 2020b. Centre row: the results of multiple runs for three different types of BO with default parameter setting in MATLAB 2020b. Lower left: the results of multiple runs for GA with population size $60$ and default other parameters in MATLAB 2020b. Lower middle: the results of multiple runs for SA with default parameter setting in MATLAB 2020b. Lower right: the results of multiple runs for DE/rand/1/bin with parameter setting $N=10$, $F=0.85$ and $Cr=0.5$. The choice of those parameters follows the existing experiences.}
  \label{CM:fig:Rosenbrock6D}
\end{figure}

\begin{figure}[tbhp]
  \centering
  \subfigure{\includegraphics[width=0.32\textwidth]{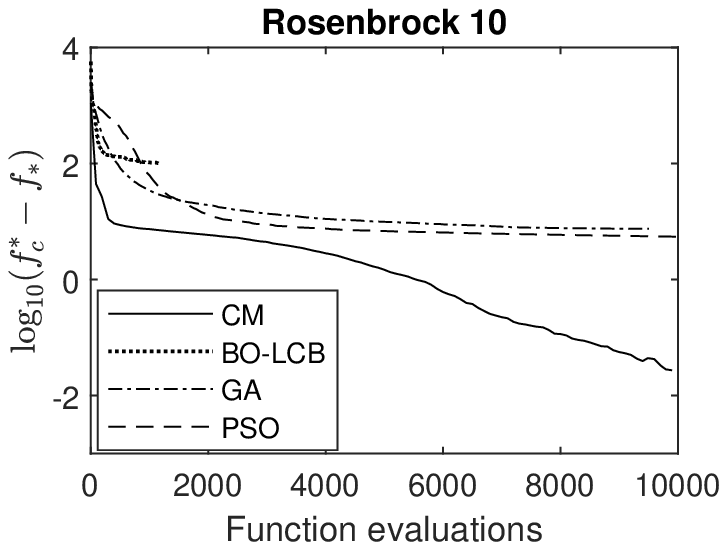}}
  \subfigure{\includegraphics[width=0.32\textwidth]{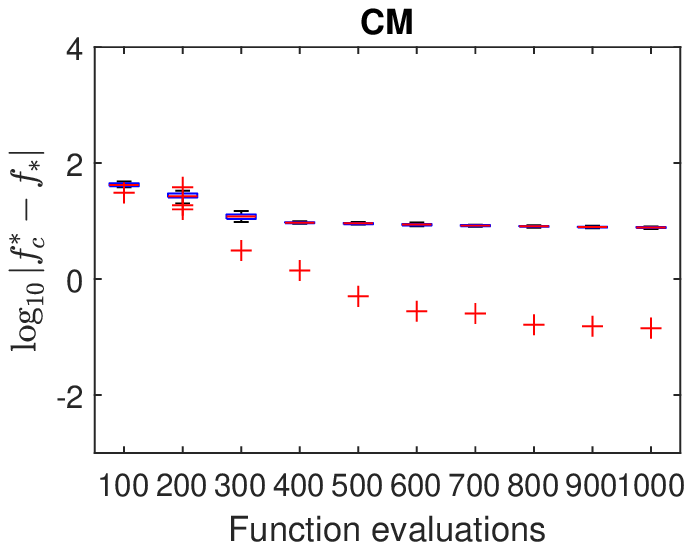}}
  \subfigure{\includegraphics[width=0.32\textwidth]{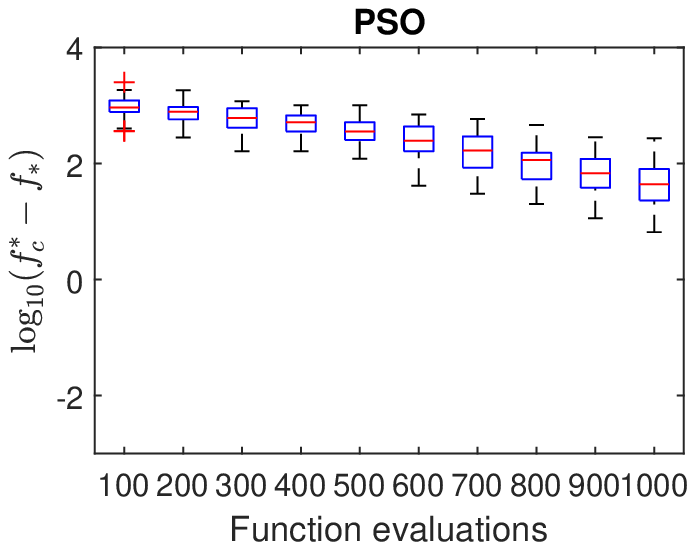}}\\
  \subfigure{\includegraphics[width=0.32\textwidth]{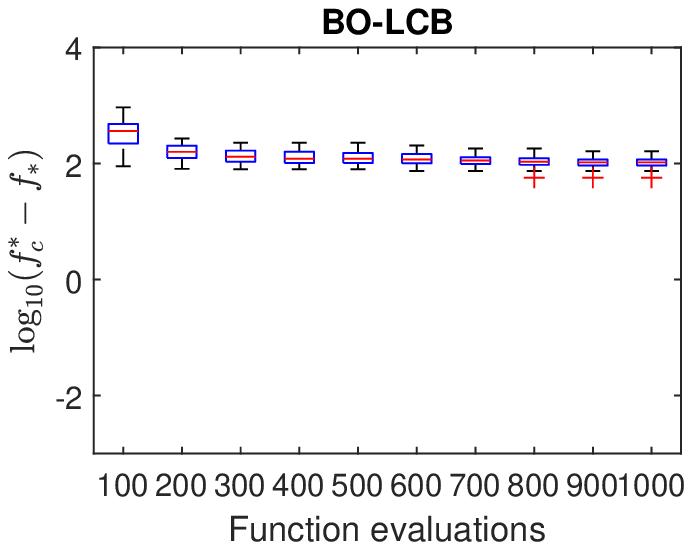}}
  \subfigure{\includegraphics[width=0.32\textwidth]{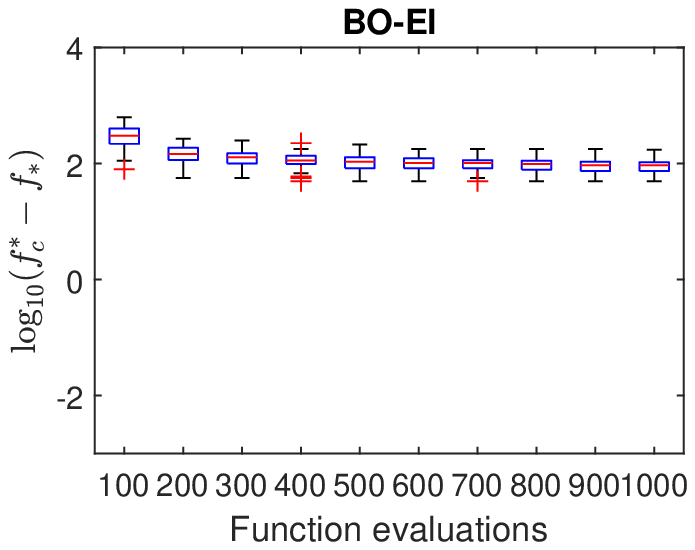}}
  \subfigure{\includegraphics[width=0.32\textwidth]{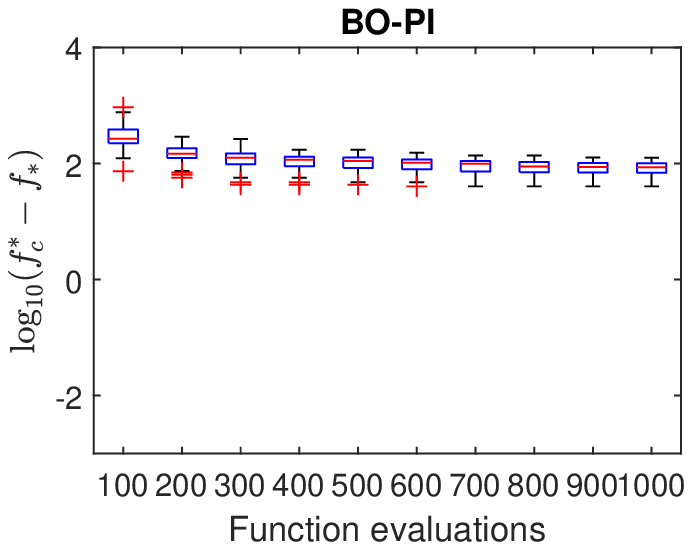}}\\
  \subfigure{\includegraphics[width=0.32\textwidth]{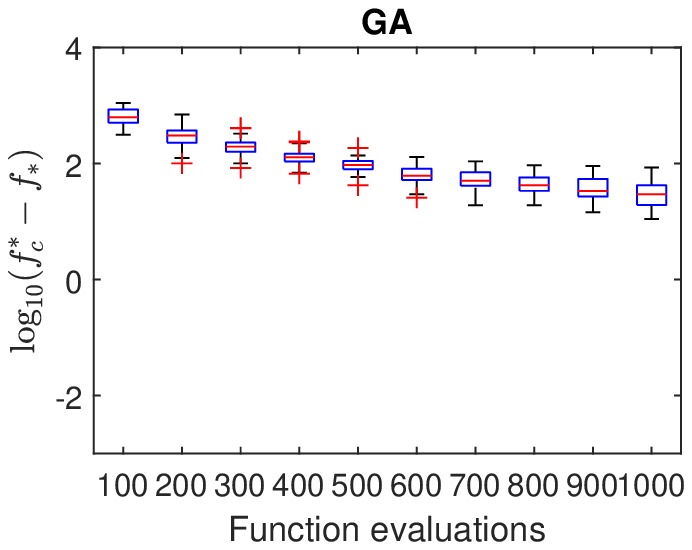}}
  \subfigure{\includegraphics[width=0.32\textwidth]{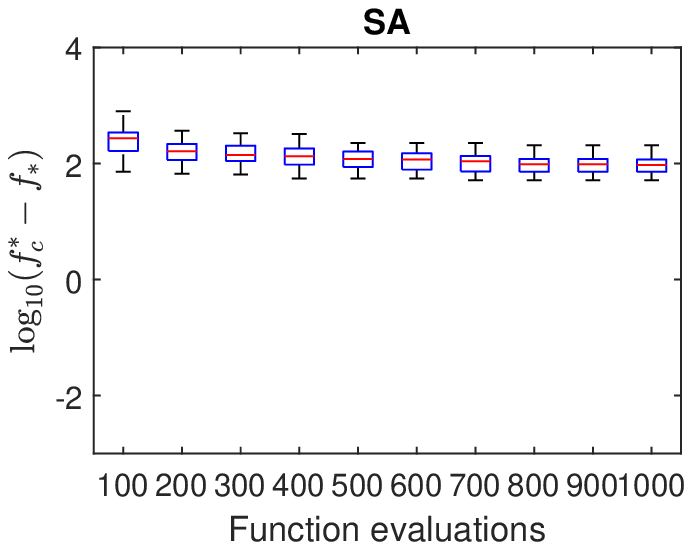}}
  \subfigure{\includegraphics[width=0.32\textwidth]{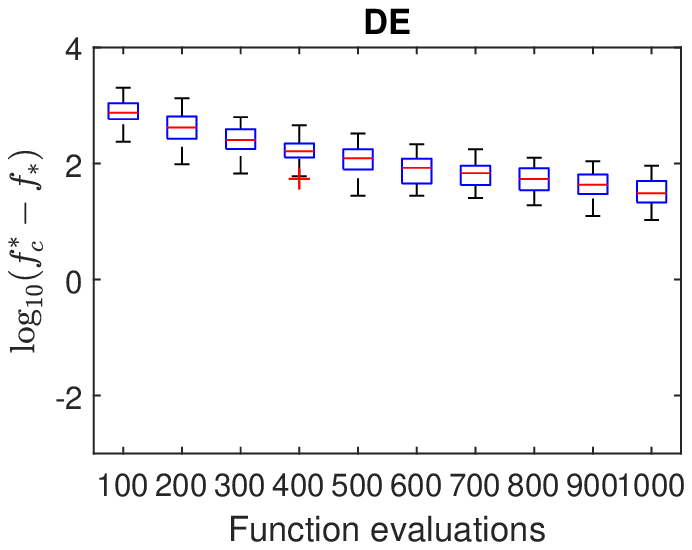}}
  \caption{Comparison for the Rosenbrock function in $10$ dimension. Upper left: the four curves correspond to the four most representative methods for this example, and each curve shows the averaged optimality gap over $50$ independent runs. Upper middle: the box plot shows the results of multiple runs for CM, with parameter setting $K=200$, $m=2$, $\textrm{minIterInner}=30$, $\omega=1$, $c_k=\bar{c}=50$ and $t_k=3.5\frac{k}{K}$. Upper right: multiple runs for the current most powerful competitor, i.e., PSO, with swarm size $50$ and default other parameters in MATLAB 2020b. Centre row: the results of multiple runs for three different types of BO with default parameter setting in MATLAB 2020b. Lower left: the results of multiple runs for GA with population size $50$ and default other parameters in MATLAB 2020b. Lower middle: the results of multiple runs for SA with default parameter setting in MATLAB 2020b. Lower right: the results of multiple runs for DE/rand/1/bin with parameter setting $N=10$, $F=0.85$ and $Cr=0.5$. The choice of those parameters follows the existing experiences.}
  \label{CM:fig:Rosenbrock10D}
\end{figure}

\subsection{Real world application: Lennard-Jones molecular conformation}

Now we consider a real world application about Lennard-Jones (LJ) microclusters which perhaps is the most intensely studied molecular conformation problem. LJ conformations of a cluster of $s$ identical neutral atoms interacting pairwise via the LJ potential. And a conformation is actually a point in the $3s$-dimensional Euclidean space of coordinates of atomic centers. For a single pair of atoms, the LJ potential in reduced units is given by
\begin{equation*}
  u(r)=r^{-12}-2r^{-6}
\end{equation*}
where $r$ is the Euclidean interatomic distance; then the total potential energy
\begin{equation*}
  U_s=\sum_{i=1}^{s-1} \sum_{j=i+1}^s u(r_{ij}),
\end{equation*}
where $r_{ij}$ is the distance between atoms $i$ and $j$ in reduced units. The putative global minima are $U^*_4=-6$ and $U^*_5=-9.103852$ for $s=4$ and $s=5$, respectively \citep{LearyR1997M_LennardJones}. Figure \ref{CM:fig:3} shows the relevant performances of Algorithm \ref{CM:alg:CM} that accurately finds the lowest energy conformations.

\begin{figure}[tbhp]
  \centering
  \subfigure{\includegraphics[width=0.32\textwidth]{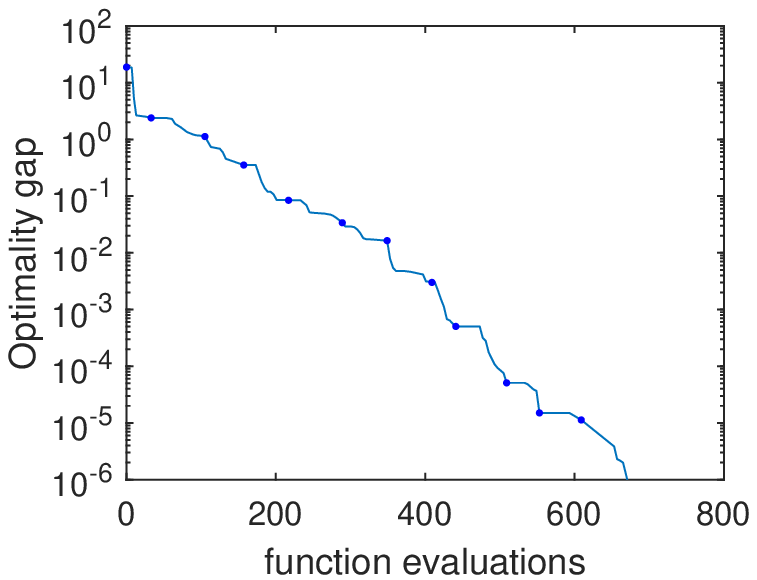}}
  \subfigure{\includegraphics[width=0.32\textwidth]{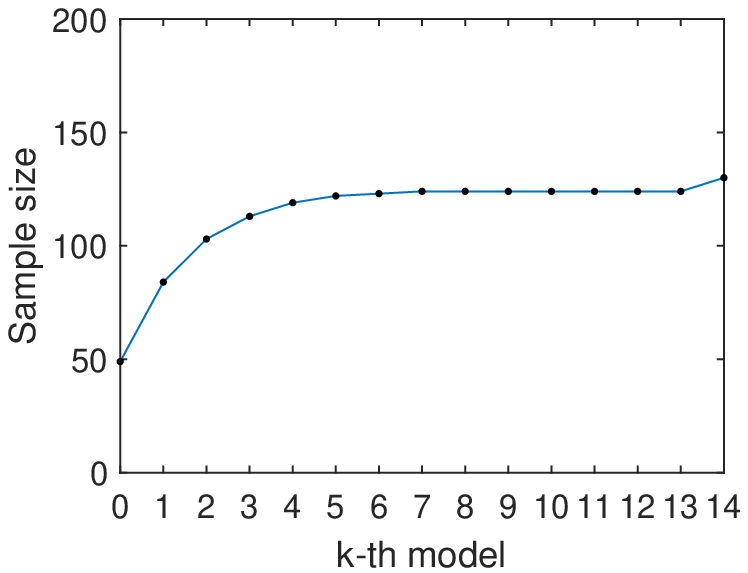}}
  \subfigure{\includegraphics[width=0.32\textwidth]{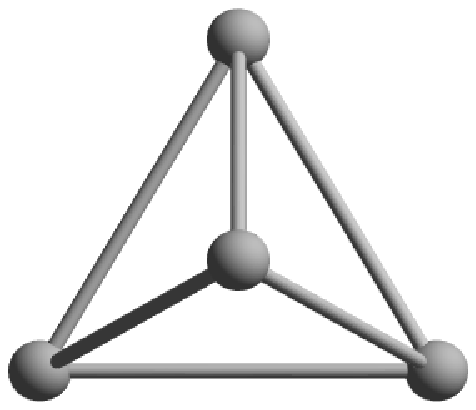}}\\
  \subfigure{\includegraphics[width=0.32\textwidth]{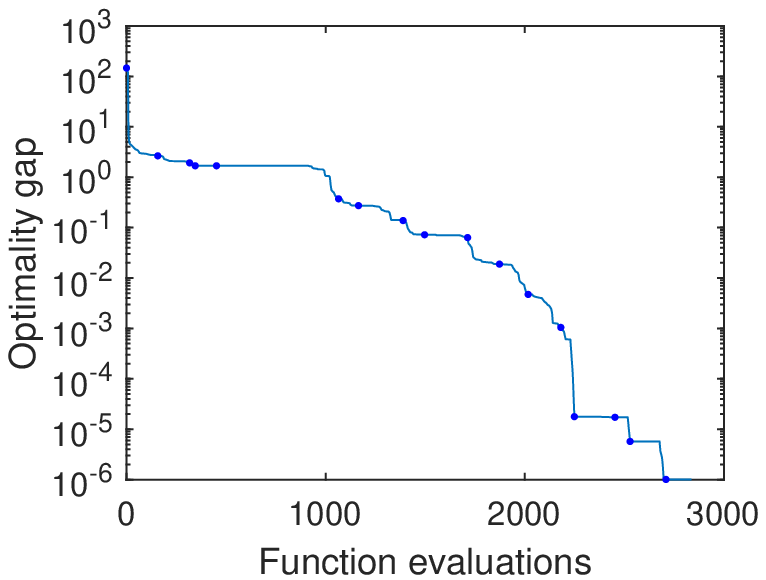}}
  \subfigure{\includegraphics[width=0.32\textwidth]{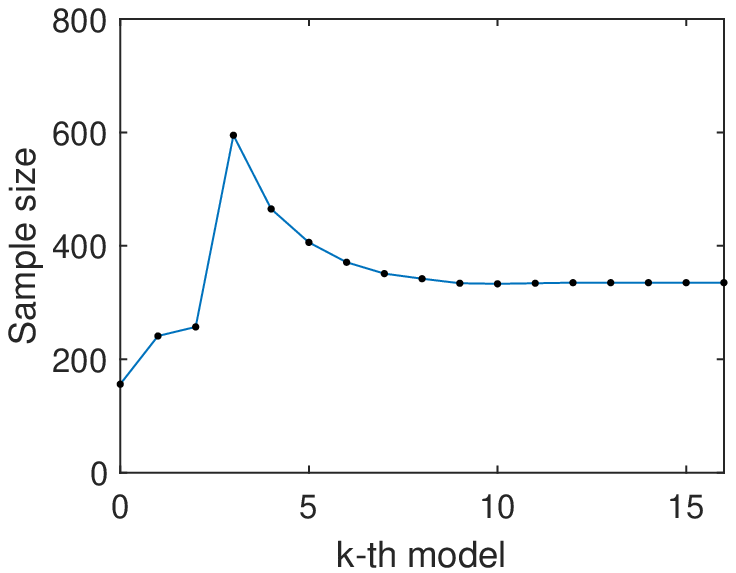}}
  \subfigure{\includegraphics[width=0.32\textwidth]{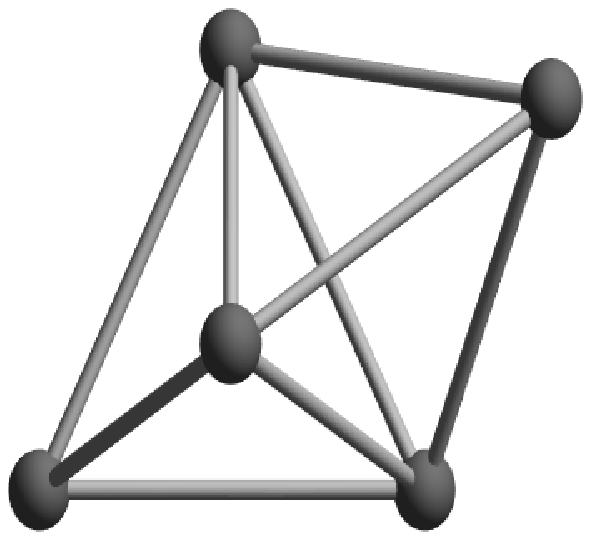}}
  \caption{Minimize the LJ potentials for $s=4$ and $s=5$ by Algorithm \ref{CM:alg:CM}. Upper row: the convergence plot, the sample size used in each model, and the global minima conformation at $s=4$. The parameter setting is $K=15$, $m=3$, $\textrm{minIterInner}=15$, $\omega=1$, $c_k=\bar{c}=50$ and $t_k=\bar{t}=2$. Lower row: the convergence plot, the sample size used in each model, and the global minima conformation at $s=5$. The parameter setting is $K=50$, $m=3$, $\textrm{minIterInner}=40$, $\omega=1$, $c_k=\bar{c}=50$ and $t_k=\bar{t}=2$.}
  \label{CM:fig:3}
\end{figure}

\section{Conclusions}
\label{CM:s8}

For a fairly general class of problems, it is often impossible to find a universal method that performs very well on all possible situations, which is the important connotation of the no free lunch theorems \citep{WolpertD1997T_NFL,MacreadyW1996T_NFL,WolpertD1996T_NFL}. This is not due to some kind of curse, but a lack of common features. Actually, a certain commonality is the premise of efficiency. And when there is no such a premise, it is always wise to find a subclass that has enough in common and maintains a proper level of generality. Although the consequent loss of generality is what we have to pay, the corresponding discriminant conditions may help us understand the problem better.

In this work, we described the concept of contractibility and then proposed a class of contraction algorithms. From the efficiency of algorithms, we have classified all possible continuous problems. Experiments with various categories of examples show that these categories seem to be reasonable. To ensure the existence of efficient optimization algorithms \citep{WolpertD1997T_NFL}, we mainly impose a class of hierarchical low-frequency dominant condition on the problems. And now, we knew that a sufficiently smooth or contractible problem can be effectively predicted using a priori information. Hence, the contractibility might be viewed as a complement to smoothness.

The algorithm is implemented in MATLAB. The source codes for the implementation of the algorithm and all examples is available at https://github.com/xiaopengluo/contropt.

Future research is currently being conducted in several areas. One of the attempts is to create possible time complexities that is not exponentially related to dimensions for some certain function classes. We hope that the expected results could provide valuable suggestions for efficiency of high-dimensional continuous optimization. Secondly, we are also considering how to establish an adaptive contraction condition to achieve the optimal efficiency for various different problems. A successful achievement will be very helpful in practice. Thirdly, we hope that some certain difficult problems can be translated into relevant easy ones by applying some preconditioning and postconditioning steps before and after each contraction. Moreover, this requires us to further distinguish which problems are inherently difficult to solve, and which are only seemingly intractable.

\vskip 0.2in
\bibliographystyle{siamplain}
\bibliography{MReferences}

\end{document}